\documentclass[12pt]{article}

\usepackage[margin=1in]{geometry}
\usepackage{setspace}
\usepackage[T1]{fontenc}
\usepackage[utf8]{inputenc}

\usepackage{amsmath,amssymb,amsfonts,amsthm,mathtools}
\usepackage{natbib}
\usepackage{graphicx}
\usepackage{booktabs}
\usepackage{multirow}
\usepackage{algorithm}
\usepackage{algorithmic}
\usepackage{enumitem}
\usepackage{bbm, bm}
\usepackage{multirow}
\usepackage{booktabs}
\usepackage{mathtools}
\usepackage{commath}
\usepackage{algorithm}
\usepackage{algorithmic}

\allowdisplaybreaks

\usepackage[
    colorlinks=true,
    citecolor=blue,
    linkcolor=blue,
    urlcolor=blue
]{hyperref}

\newtheorem{theorem}{Theorem}[section]
\newtheorem{lemma}[theorem]{Lemma}
\newtheorem{proposition}[theorem]{Proposition}

\theoremstyle{definition}

\newtheorem{remark}[theorem]{Remark}

\newcommand{\ind}[0]{\mathbbm{1}}
\DeclareMathOperator{\E}{\mathbb{E}}
\newcommand{\R}{\mathbb{R}}

\newcommand{\Exp}{\mathbb{E}}
\newcommand{\Prob}{\mathbf{P}}

\newcommand{\sni}{\sum_{i_1 = 1}^{n}}

\newcommand{\sm}{\sum_{m = 1}^{d}}

\newcommand{\Ga}{\Gamma_n}

\newcommand{\Ghat}{\widehat{\Gamma}_n}
\newcommand{\Gtill}{\widehat{\Gamma}_n^\text{JMB}}

\newcommand{\inr}{\mathcal{I}_n^r}
\newcommand{\inrj}{\mathcal{I}_{n,-j}^r}

\newcommand{\iinr}{(i_1,\dots,i_r)\in\mathcal{I}_n^r}

\newcommand{\jjnr}{(j_1,\dots,j_r)\in\mathcal{I}_n^r}

\newcommand{\iinri}{(i_2,\dots,i_r)\in\mathcal{I}_{n,-i_1}^r}
\newcommand{\jjnri}{(j_2,\dots,j_r)\in\mathcal{I}_{n,-i_1}^r}
\newcommand{\kknri}{(k_2,\dots,k_r)\in\mathcal{I}_{n,-i_1}^r}

\newcommand{\npr}{P_{n,r}}

\newcommand{\nprc}{P_{n,2r-c}}

\newcommand{\nopr}{P_{n-1,r-1}}

\newcommand{\cpr}{P_{r,c}}

\newcommand{\xir}{X_{i_{[1:r]}}}
\newcommand{\xjr}{X_{j_{[1:r]}}}

\newcommand{\xijr}{X_{i_1},X_{j_{[2:r]}}}
\newcommand{\xiir}{X_{i_1},X_{i_{[2:r]}}}

\newcommand{\icnri}{(i_1,\dots,i_{2r-c})\in\mathcal{I}_{n}^{2r-c}}

\newcommand{\xcri}{X_{i_{[1:c]}},X_{i_{[(c+1):r]}}}
\newcommand{\xcrii}{X_{i_{[1:c]}},X_{i_{[(r+1):(2r-c)]}}}

\newcommand{\jjcr}{(j_1,\dots,j_{r})\in\mathcal{C}_{n}^{r}}

\newcommand{\xixr}{X\xi_{[1:r]}}
\newcommand{\ixr}{\xi_{i_{[1:r]}}}

\newcommand{\ixrr}{\xi_{{[1:c]}},\xi_{{[(c+1):r]}}}
\newcommand{\ixrrr}{\xi_{{[1:c]}},\xi_{{[(r+1):(2r-c)]}}}

\newcommand{\ixjr}{\xi_{j_{[1:r]}}}

\newcommand{\ixccr}{\xi_{i_{[(r+1):2r-c]}}}
\newcommand{\ixcor}{\xi_{i_{[(c+1):r]}}}

\newcommand{\h}{\mathcal{H}_n}

\newcommand{\hs}{\text{HS}}
\newcommand{\opr}{\text{op}}
\newcommand{\la}{\langle}
\newcommand{\ra}{\rangle}

\newcommand{\iicr}{(i_1,\dots,i_r)\in\mathcal{C}_n^r}

\newcommand{\iicrc}{(i_1,\dots,i_{2r-c})\in\mathcal{C}_n^{2r-c}}
\newcommand{\jjnci}{(j_2,\dots,j_r)\in\mathcal{C}_{n,-i_1}^r}
\newcommand{\kknci}{(k_2,\dots,k_r)\in\mathcal{C}_{n,-i_1}^r}

\newcommand{\Pp}{\mathbb P}

\newcommand{\cP}{\mathcal P}
\newcommand{\cG}{\mathcal G}
\newcommand{\TV}{\mathrm{TV}}

\DeclareMathOperator{\trc}{tr}
\DeclareMathOperator{\tr}{tr}

\title{\bfseries Berry--Esseen bounds and bootstrap approximations for the Hilbert-space norm of $U$-statistics}

\author{
Nilanjan Chakraborty\thanks{Department of Mathematics and Statistics, Missouri University of Science and Technology, Rolla, MO, USA}
\and
Sayan Das\thanks{Department of Statistics and Data Science, Washington University in St. Louis, St. Louis, MO, USA}
}

\date{}

\begin{document}

\maketitle


\begin{abstract}
We establish non-asymptotic Berry--Esseen bounds and bootstrap approximations for the Hilbert-space norm of nondegenerate $U$-statistics. Our triangular-array framework allows the kernel to take values in a separable Hilbert space $\h$ that may vary with the sample size, thereby covering both infinite-dimensional spaces and Euclidean spaces of increasing dimension. The Gaussian approximation combines a Hilbert-space version of Stein’s method with refined control of the higher-order terms in the Hoeffding decomposition. Its error depends on fourth moments of the kernel and on the spectral geometry of the covariance operator of the Hájek projection, particularly the proportion of squared spectral mass lying outside the leading eigendirection. Consequently, the theory accommodates singular and approximately low-rank covariance operators, provided that sufficient spectral mass remains beyond the leading direction. For $\h=\R^{d_n}$, we obtain dimension-explicit bounds under coordinatewise fourth-moment conditions. We also establish approximation bounds for the empirical, Gaussian-weighted, and jackknife multiplier bootstraps, yielding tests with asymptotically correct size and consistency against alternatives at covariance-dependent separation rates. For testing the vector of pairwise Kendall’s tau coefficients, a matching minimax lower bound shows that the resulting separation rate is minimax rate-optimal over dense Gaussian correlation alternatives.
\end{abstract}

\vspace{0.5em}


\section{Introduction}
\subsection{Overview}
Quantifying the accuracy of distributional approximations for high-dimensional statistics is one of the fundamental problems in modern probability and statistics. Of particular interest is the interplay between a class of sets over which one compares the distributions, the dimension of the statistics, and the sample size of the data. 
The recent literature has seen a rapid development of the understanding of such interplay for the sample mean in high dimensions; see, e.g., a recent review article \cite{chernozhukov2023high, Chakraborty2025} and a literature review below. Yet, relatively less is known for \textit{nonlinear} statistics regarding sharp conditions on the dimension for tractable distributional approximations to hold. The study of $U$-statistics constitutes a cornerstone of non-parametric statistics, providing a powerful and general framework for estimation and inference. $U$-statistics allow researchers to study complex, non-linear functionals of a data-generating process, examples of such applications include independence testing (see, e.g., \cite{bergsma2014consistent,yao2018testing,leung2018testing,berrett2021optimal} for recent contributions), testing for qualitative features of nonparametric functions \cite{abrevaya2005nonparametric,ghosal2000testing,lee2009testing}, and cross validation for density estimation \cite{nolan1987u}. We refer the readers to \cite{serfling2009, bose2018, lee2019u, korolyuk2013theory} for excellent references on $U$-statistics and $U$-processes. 

In modern statistical applications such as functional data analysis (FDA) and high-dimensional inference, the data-generating mechanisms increasingly reside in the infinite-dimensional Hilbert spaces $\h$ or high-dimensional Euclidean spaces $\R^{d_n}$. Conducting valid statistical inference in these environments requires a deep understanding of the non-asymptotic behavior of these statistics, particularly the behavior of their geometric norms. The classical asymptotic theory for $U$-statistics relies heavily on the Hoeffding decomposition and the H\'ajek projection. By approximating the non-linear kernel with a sum of independent linear terms, one typically establishes a central limit theorem (CLT) \cite{HoeffdingU}. The classical Berry--Esseen theorem originates in the work of
\cite{Berry1941,Esseen1942}, which established quantitative rates of normal
approximation for sums of independent real-valued random variables.
For real-valued nondegenerate $U$-statistics, an analogous univariate
Berry--Esseen theory was developed by \cite{callaert1978berry}, who
obtained the canonical $n^{-1/2}$ rate under suitable moment conditions on
the kernel. Multivariate and infinite-dimensional extensions were developed
by, among others, \cite{bhattacharya2010normal,kuelbs1974berry,
sazonov1989asymptotically, bentkus2003dependence}. Of particular relevance to the present work,
\cite{korolyuk1990rate, borovskikh1997normal} studied rates of normal approximation for
Hilbert-valued $U$-statistics over balls. 

Existing Berry--Esseen-type bounds for Hilbert-valued $U$-statistics often rely on global structural constraints; for instance, classical techniques relying on characteristic function smoothing frequently require that the underlying covariance operator possesses an infinite sequence (or at least a substantial number) of strictly positive eigenvalues bounded away from zero \cite{korolyuk1990rate, borovskikh1997normal}. 

In this paper, we attempt to address these limitations by establishing non-asymptotic Gaussian and bootstrap approximation bounds for Hilbert-valued and large-dimensional $U$-statistics under relaxed spectral and moment conditions. 
We assume that $\bm X = \{X_1, X_2,\dots,$ $X_n\}$ are independent but not necessarily identically distributed random elements taking values in some measurable space $\mathcal{X}$. Further, let \((\h,\langle\cdot,\cdot\rangle)\) be a separable Hilbert space, equipped with the norm induced by the inner product,
$
\|u\|:=\sqrt{\langle u,u\rangle}, u\in\h.
$ For a measurable and symmetric kernel $\tilde{h}_n : \mathcal{X}^r  \to \h $ of order $r$, we consider the $U$-statistic 
\begin{equation*}
    U_n = U_n(\tilde{h}_n) = \frac{1}{\binom{n}{r}} \sum_{1 \le i_1 < \cdots < i_r \le n} \tilde{h}_n(X_{i_1},\dots,X_{i_r}).
\end{equation*}
We will assume that
the order $r$ is fixed and denote the expectation of $U_n$ by $\theta_n$, i.e.,
\[
\theta_n = \frac{1}{\binom{n}{r}} \sum_{1 \le i_1 < \cdots < i_r \le n}\E[\tilde{h}_{n}(X_{i_1},\dots,X_{i_r})],
\]
where we assume that the expectations on the right-hand side are all finite. Henceforth, we suppress $n$ in $\tilde h_n$ and $\theta_n$ and denote them as $\tilde h$ and $\theta$ respectively. The expectation of a random element in a separable Hilbert space will be in the Bochner sense (see \cite{araujo1980central}). 
The goal of this paper is to establish the accuracy assessment for the Gaussian and bootstrap approximations to the normalized statistic 
\[
T_n =\sqrt{n}(U_n-\theta)/r,
\]
under the Hilbert norm induced by the inner product, which amounts to approximating the sampling distribution of $T_n$ over the class of centered balls in $\h$. In view of the Hoeffding decomposition, the Gaussian analog of $T_n$ is given by $Z_n \sim N_{\h} (0,\Gamma_n)$ with $\Gamma_n$ being the covariance operator of the H\'ajek projection. Our first main result establishes finite-sample upper bounds on 
\begin{equation}\label{berry}
d_{\mathrm{Kol}}(\|T_n\|,\|Z_n\|) := \sup_{a \ge 0} |\Prob(\|T_n\| \le a) - \Prob(\|Z_n\| \le a)|,
\end{equation}
under a finite fourth moment condition on the kernel $\tilde{h}_n$
 and other mild conditions. The approximation rate for \eqref{berry} is $O(n^{-1/8})$.

 Moving beyond the Hilbert-valued $U$-statistic, we also study the large-dimensional regime where $X_i's, i = 1,2,\ldots,n$ are random vectors in $\R^{d_n}$, where we explicitly track the dependence between $d_n$ (hereafter denoted as $d$) and $n$. We show that under a uniform fourth-moment condition on the kernel coordinates, the Berry--Esseen type results analogous to \eqref{berry} are defined via the interplay between the Frobenius norm and the operator norm of the covariance matrices associated with the $U$-statistic.
 The approximation rates obtained here are similar to those of mean vectors obtained in \cite{fang2024large}. 
 

 In practice, the approximating Gaussian distribution $Z_n$ is rarely feasible as the covariance operator $\Gamma_n$ is unknown. In contrast to the sample mean case, the analytical estimation of $\Gamma_n$ is not straightforward as it is given by the covariance operator of the H\'ajek projection that is defined by the conditional expectation of the kernel. As such, bootstrapping is a particularly appealing alternative to inference for $U$-statistics. The bootstrap, introduced by \cite{efron1979bootstrap}, provides a data-driven
alternative when the covariance structure of the limiting Gaussian law is
unknown. Its first-order validity and asymptotic accuracy were studied
systematically by \cite{bickel1981some}, while
\cite{arcones1992bootstrap} developed bootstrap distributional limit theory for
$U$- and $V$-statistics under weak moment conditions. In high-dimensional
settings, non-asymptotic guarantees for empirical and multiplier bootstrap
approximations of sums of independent random vectors were established by
\cite{chernozhukov2017clt, chernozhuokov2022improved}; see also \cite{zhilova2020nonclassical} for
finite-sample bootstrap approximations over Euclidean balls and for smooth
functions of sums. For nondegenerate high-dimensional $U$-statistics,
\cite{chen2018gaussian, chen2019randomized} obtained corresponding empirical and Gaussian multiplier
bootstrap bounds over hyperrectangles.
 
 In this paper, we consider the following three bootstrap methods: (i) the jackknife multiplier bootstrap, (ii) the weighted bootstrap, and (iii) the empirical bootstrap. For these bootstraps, we will establish finite-sample stochastic upper bounds on 
 \[
d_{\mathrm{Kol}}^B(\|T_n^\star\|,\|Z_n\|):=\sup_{t \ge 0} |\Prob^\star(\|T_n^\star\| \le t) - \Prob(\|Z_n\| \le t)|,
 \]
where $\Prob^\star$ denotes the conditional probability given the sample $\bm X,$ and $T_n^\star$ is the bootstrap statistic constructed by any of the methods (i)--(iii).





\subsection{Our contributions} 
\begin{itemize}
    \item Our first main contribution is that we develop the Gaussian approximation of $U$-statistics under the setup where the kernel $\tilde h$ is allowed to take values in an abstract separable Hilbert space $\h$, and thus allows us to develop a flexible framework to handle both functional and high-dimensional data in a rather unified manner. We derived the Gaussian approximation result for independent nondegenerate $U$-statistics of order $r$ under the inner product-induced norm (e.g., Euclidean or $\ell_2$-norm in $\R^d$) and under relaxed moment conditions, as compared to sub-exponential tail or polynomial tail type assumptions in \cite{chen2018gaussian} where the results are derived using sup-norm or $\ell_{\infty}$ norm in $\R^d$. 
    The main challenge here is to obtain bounds for the non-linear terms or the remainder terms in a way that they turn out to be of smaller order as compared to the linear terms or the H\'ajek projection terms, which have an order similar to the linear statistic in \cite{fang2024large}. 
    
    \item The second major contribution of the paper is deriving the Berry--Esseen bounds for the bootstrap approximations under three widely used bootstrap methods for both Hilbert-valued and $d$-dimensional $U$-statistics. To the best of our knowledge, we are not aware of previous explicit non-asymptotic Berry--Esseen bounds for bootstrap approximation of the Hilbert norm of a nondegenerate $U$-statistic over centered balls
and thus our results are particularly appealing to inference problems in functional and high-dimensional data analysis. The bootstrapping guarantees derived here are quite non-trivial as compared to the ones in \cite{fang2024large}, as bounds are to be derived for covariance matrices for any general kernel $\tilde h$.
    \item Finally, our technical guarantees cover the case where the observational vectors are non-identically distributed, which allows us to accommodate heterogeneous data structures that arise in many practical applications \cite{alberink2001berry, han2018inference}. However, relaxing the identicalness in the assumption substantially increases the technical complexity of our analysis, particularly in the bootstrap arguments.
\end{itemize}  

The remainder of the paper is organized as follows. Section \ref{sec:notation} introduces the fundamental notation, and Section \ref{sec:prelim} reviews the preliminaries for $U$-statistics. Section \ref{sec:limit} establishes the non-asymptotic Gaussian approximation rates for both Hilbert spaces and high-dimensional Euclidean spaces. Section \ref{sec: boot} derives convergence rates for the three bootstrap procedures. Section \ref{sec: stat app} presents some applications of the proposed results to problems in statistical inference. Finally, we defer all technical proofs and auxiliary lemmas to the Appendices.

\subsection{Notation}\label{sec:notation}
We shall use the following notation and conventions in this paper.

For any two sequences  $a_{n} $ and $b_{n} $ of positive  numbers write $a_{n} \lesssim b_{n} $, if $a_{n} \le Cb_{n} $, for some  universal constant $C > 0$, and write $a_n = o(b_n)$, if $a_n/b_n \to 0$. Further, we write $a_{n} \asymp  b_{n} $ if $a_{n} \lesssim b_{n} $ and $b_{n} \lesssim a_{n} $.

Let $\{e_{j,n}\}_{j\geq 1}$ be a complete orthonormal basis of the
separable real Hilbert space $\h$. 
For a square-integrable random element $X$ in a separable Hilbert space $\h$ with mean $\mu=\E X$, its covariance operator is $\Gamma:=\E[(X-\mu)\otimes(X-\mu)]$, equivalently, $\Gamma u=\E[\langle X-\mu,u\rangle(X-\mu)]$ for $u\in\h$.
For $u,v\in\h$, define
the rank-one operator $u\otimes v:\h\to\h$ by
$
(u\otimes v)z=\langle v,z\rangle u,
\qquad z\in\h.
$
For a bounded linear operator $T:\h\to\h$, define
$
\|T\|_{\mathrm{op}}
=
\sup_{\|x\|=1}\|Tx\|.
$
Its adjoint $T^*:\h\to\h$ is the unique bounded linear
operator satisfying
$
\langle Tx,y\rangle=\langle x,T^*y\rangle,
\qquad x,y\in\h.
$
The operator $T$ is called Hilbert--Schmidt if
$
\sum_{j=1}^{\infty}\|Te_{j,n}\|^2<\infty.
$
For such an operator, its Hilbert--Schmidt norm is
$
\|T\|_\hs
=
\left(\sum_{j=1}^{\infty}\|Te_{j,n}\|^2\right)^{1/2}.
$
This definition is independent of the choice of the orthonormal basis.
For two Hilbert--Schmidt operators $T_1$ and $T_2$, their
Hilbert--Schmidt inner product is
$
\langle T_1,T_2\rangle_{\mathrm{HS}}
=
\sum_{j=1}^{\infty}\langle T_1e_{j,n},T_2e_{j,n}\rangle
=
\operatorname{tr}(T_2^*T_1).
$
For a covariance operator $\Sigma,$ we denote the $k$-th largest eigenvalue as $\lambda_{(k)}(\Sigma)$. Additionally, for $k=1,2$ we define $\Lambda^2_k(\Sigma) = \sum_{j=k}^\infty \lambda_{(j)}^2(\Sigma)$ and when $\Lambda_2(\Sigma)>0,$ we define $\aleph(\Sigma) =(\Lambda_1(\Sigma)\;\Lambda_2(\Sigma))^{-1/2}$.

For a positive integer $d$, $\R^d$ denotes the $d$-dimensional Euclidean space except that we write $\R^1$ as $\R$. For $x\in \R^d$, $x^\top$ denotes its transpose. For any two vectors $x = (x_1,\dots,x_d)^\top \in \R^{d} $ and $y = (y_1,\dots,y_d)^\top \in \R^{d} $, write $x \leq y$ if $x_j \le y_j $ for all $j=1,\dots, d.$ For any $x = (x_1,\dots,x_d)^\top \in \R^{d} $ and $a \in \R $, $x+a = (x_1+a,\dots,x_d+a)^\top$.  For any symmetric positive semi-definite matrix of order $d\times d,$ $A=(a_{ij})$, we denote the spectral norm as $\| A \|_\opr = \lambda_{(1)} (A)$ and the Frobenius norm as $\| A \|_\text{F} = (\sum_{i,j=1}^d a_{ij}^2)^{1/2}$. For $\alpha > 1$, we define the function $\psi_{\alpha}: [0, \infty) \to [0, \infty)$ by $\psi_{\alpha} (x) = \exp(x^{\alpha}) - 1$, and for real valued random variable $\xi$, we define the Orlicz norm as $\|\xi\|_{\psi_\alpha} = \inf\{\lambda > 0: \E[\psi_{\alpha}(|\xi|/\lambda)] \le 1\}.$

For indexing, we define $\inr := \{ (i_1, \dots, i_r): 1\leq i_1 \neq \cdots \neq i_r \leq n\},$ and for $1\leq j\leq n,$ $\inrj := \{ (i_1, \dots, i_{r-1}): 1\leq i_1 \neq \cdots \neq i_{r-1} \leq n,\  i_1, \dots, i_{r-1} \neq j) \}$, Note that, by $i_1\neq\cdots\neq i_r,$ we mean pairwise $i_1,\dots,i_r$ are not same. Define $\npr := n!/(n-r)! = |\inr|$. Further, we define $\mathcal C_n^r := \{ (i_1, \dots, i_r): 1\leq i_1, \dots,i_r \leq n\}$. For the brevity of notation, for any integers $0<N_1<N_2<N_3<N_4<r,$ $1\leq i_{N_1}<i_{N_2}<i_{N_3}<i_{N_4} \leq n$ and for any function $\mathcal{K}(.)$ defined on $\mathcal{X}^r,$ we write $\mathcal{K}(X_{i_{N_1}}, X_{i_{N_1+1}},\dots, X_{i_{N_2}}, X_{i_{N_3}}, X_{i_{N_3+1}},\dots, X_{i_{N_4}})$ as $\mathcal{K}(X_{i_{[N_1:N_2]}}, X_{i_{[N_3:N_4]}})$. 
Throughout the paper, $C>0$ is a constant independent of sample sizes $n$, and the value of $C$ varies from place to place.

\section{Nonasymptotic approximation results} \label{sec:main}
\subsection{Preliminaries}\label{sec:prelim}

We first review $U$-statistics and the H\'ajek projection, which is one
of the fundamental tools for analyzing $U$-statistics. Recall that we
have assumed that the kernel $\widetilde h$ is symmetric; that is,
$
\widetilde h(x_{\sigma(1)},\ldots,x_{\sigma(r)})
=
\widetilde h(x_1,\ldots,x_r)
$
for every permutation $\sigma$ of $\{1,\ldots,r\}$. Consequently, the
$U$-statistic can also be written as
\begin{equation}\label{defU}
U_n
=
\frac{1}{\npr}
\sum_{\iinr}
\widetilde h(\xir).
\end{equation}
For the sake of simplicity, the symmetry of the kernel $\widetilde h$ has been assumed here. 

To accommodate the non-identical distributions of the observations,
for every $\bm i = (i_1,\dots$ $,i_r) \in \inr$, define
\begin{equation*}
\theta_{\bm i}
:=
\E\!\left[
\widetilde h(\xir)
\right],
\end{equation*}
and the tuple-specific centered kernel
\begin{equation*}
h_{\bm i}(x_1,\ldots,x_r)
:=
\widetilde h(x_1,\ldots,x_r)-\theta_{\bm i}.
\end{equation*}
Then
\begin{equation*}
\E\!\left[
h_{\bm i}(\xir)
\right]
=0.
\end{equation*}
Recall that
\begin{equation*}
\theta
=
\frac{1}{\npr}
\sum_{\iinr}
\theta_{\bm i}.
\end{equation*}
Define the centered $U$-statistic by
\begin{equation*}
U_n^\circ
:=
U_n-\theta
=
\frac{1}{\npr}
\sum_{\iinr}
h_{\bm i}(\xir).
\end{equation*}
Thus, $\E U_n^\circ=0$, and we define
$
T_n:=\sqrt n U_n^\circ/r.
$

For $r \ge 2$, $\bm i=(i_1,\ldots,i_r)\in\mathcal I_n^r$ and
$c\in\{1,\ldots,r\}$, define the $c$-th tuple-specific first-order
projection by
\begin{equation*}
\vartheta_{\bm i,c}(X_{i_c})
:=
\E\!\left[
h_{\bm i}(\xir)
\,\middle|\,
X_{i_c}
\right]
=
\E\!\left[
\widetilde h(\xir)
\,\middle|\,
X_{i_c}
\right]
-\theta_{\bm i}.
\end{equation*}
It follows from the tower property that
\begin{equation}\label{eq:psi-centered}
\E\!\left[
\vartheta_{\bm i,c}(X_{i_c})
\right]
=0.
\end{equation}
By symmetry, the H\'ajek projection of $U_n^\circ$ is
$rL_n/\sqrt n$, where
\begin{equation}\label{lin}
L_n
:=
\frac{1}{\sqrt n}
\sni
g^{(i_1)}(X_{i_1}),
\end{equation}
with
\begin{equation*}
g^{(i_1)}(X_{i_1})
:=
\frac{1}{\nopr}
\sum_{\iinri}
\vartheta_{\bm i,1}(X_{i_1}),
\qquad
i_1 = 1,\dots,n.
\end{equation*}
Here, the inner summation is over all
$(i_2,\ldots,i_r)\in\mathcal I_{n,-i_1}^{r-1}$.

Since the observations are independent and
$\E[g^{(i_1)}(X_{i_1})]=0$, the covariance operator of $L_n$ is
\begin{equation}\label{VarLn}
\Gamma_n
:=
\operatorname{Cov}(L_n)
=
\frac{1}{n}
\sni
\E\!\left[
g^{(i_1)}(X_{i_1})
\otimes
g^{(i_1)}(X_{i_1})
\right].
\end{equation}
For $\bm i=(i_1,\ldots,i_r)\in\mathcal I_n^r$, define the
index-dependent remainder kernel by
\begin{equation*}
f^{\bm i}(\xir)
:=
h_{\bm i}(\xir)
-
\sum_{c=1}^r
\vartheta_{\bm i,c}(X_{i_c}).
\end{equation*}
Then the remainder $U_n^\circ-rL_n/\sqrt n$ can be expressed as
$rR_n/\sqrt n$, where
\begin{equation}\label{rem}
R_n
:=
\frac{\sqrt n}{rP_{n,r}}
\sum_{\iinr}
f^{\bm i}(\xir).
\end{equation}

Thus,
\begin{equation*}
\frac{rR_n}{\sqrt n}
=
\frac{1}{P_{n,r}}
\sum_{\iinr}
f^{\bm i}(\xir).
\end{equation*}
This is a generalized $U$-statistic with index-dependent kernels. These
kernels are first-order degenerate in the sense that
\begin{equation*}
\E\!\left[
f^{\bm i}(X_{i_{[1:r]}})
\,\middle|\,
X_{i_c}
\right]
=0,
\qquad
c=1,\ldots,r,\quad
\bm i \in\mathcal I_n^r.
\end{equation*}
Indeed, for $d\neq c$, independence and
\eqref{eq:psi-centered} imply
$
\E[\vartheta_{\bm i,d}(X_{i_d})\mid X_{i_c}]=0.
$ We denote the covariance operator of $R_n$ as $\Gamma_{n, R} := \operatorname{Cov}(R_n).$

Consequently, the normalized centered $U$-statistic admits the exact
decomposition
\begin{equation*}
T_n
=
\frac{\sqrt n}{r}U_n^\circ
=
L_n+R_n.
\end{equation*}
We present the Gaussian approximation result for $\|T_n\|$ in
Section~\ref{sec:limit}.

\subsection{Gaussian approximation}
\label{sec:limit}
In this section, we develop a non-asymptotic Gaussian approximation for the norm of a normalized, nondegenerate $U$-statistic. Our principal result, Theorem~\ref{thm: GA}, is formulated in a triangular-array framework in which the kernel may take values in a separable Hilbert space $\mathcal{H}_n$ that is allowed to vary with $n$. Consequently, the ambient space, the covariance operator, its spectrum, and the relevant moment quantities may all depend on the sample size. The theorem should therefore be viewed as a key approximation result for this article which for each $n$, provides an explicit bound for the Kolmogorov distance between the distribution of $\|T_n\|$ and that of $\|Z_n\|$, where $Z_n$ is a centered Gaussian element in $\mathcal{H}_n$ having the same covariance operator as the H\'ajek projection of $T_n$. Equivalently, the theorem quantifies the accuracy of the Gaussian approximation over the class of centered balls in $\mathcal{H}_n$.

The structure of the bound essentially reflects the Hoeffding decomposition $T_n=L_n+R_n,$
where $L_n$ is the first-order H\'ajek projection and $R_n$ collects the higher-order degenerate components. The resulting error bound separates the contribution of the linear projection from that of the nonlinear remainder and records explicitly how the approximation depends on the moment behavior of the kernel and the spectral geometry of the covariance operator. 

Before stating the main result, we introduce the following moment and spectral quantities:
\begin{align*}
   m_n  = \underset{\iicr}{\max} \Exp\|\widetilde h(\xir)\|^4,
   \text{ and }  \ell_n = 1-\frac{\|\Ga\|^2_\opr}{\|\Ga\|^2_\hs},\\
\end{align*}

whenever $\|\Ga\|_{HS} > 0$. 



\begin{theorem} \label{thm: GA}
    Assume $m_n< \infty$ and $\ell_n>0$, then we have
    \begin{equation*}
    d_{\mathrm{Kol}}(\|T_n\|,\|Z_n\|) \lesssim \Delta_n,
\end{equation*}
where 
\[
\begin{aligned}
\Delta_n & := \ell_n^{-3/16}\mathcal A_{1,n} + \ell_n^{-1/6}\mathcal A_{2,n} + \ell_n^{-1/8}\mathcal A_{3,n},\\
\mathcal A_{1,n}
&:=
\frac{1}{\|\Ga\|_\hs^{3/4}}
\left\{
\frac{\|\Ga\|_\hs\; m_n}{n}
+
\frac{\|\Ga\|_{\mathrm{op}}^{3/2}m_n^{3/4}}{\sqrt n}
\right\}^{1/4},\\
\mathcal A_{2,n}
&:=
\frac{1}{\|\Ga\|_\hs^{2/3}}
\left\{
\frac{\|\Ga\|_{\mathrm{op}}^{1/2}m_n^{3/4}}{\sqrt n}
+
\frac{m_n}{n}
\right\}^{1/3},\\
\mathcal A_{3,n}
&:=
\frac{1}{\|\Ga\|_\hs^{1/2}}\left\{ \E(\tr(\Gamma_{n,R}) + \sqrt{\E(\tr(\Gamma_{n,R})\E(\tr(\Ga)}) \right\}^{1/2}.
\end{aligned}
\]
\end{theorem}
The quantity $m_n$ measures the fourth moment magnitude of the kernel, which essentially controls both the H\'ajek projection and the higher-order terms. The quantity $m_n$ is defined over $\iicr$ as it provides the von-Mises moment analogue required only for empirical bootstrap later in Section \ref{sec: boot}. In the triangular-array setting, $m_n$ is not required to remain bounded uniformly in $n$; its dependence on $n$ is retained explicitly in the bound. The quantity $\ell_n$ describes the covariance-induced spectral geometry relevant to the Gaussian approximation. Since
$
\|\Gamma_n\|_{\mathrm{op}}=\lambda_{(1)}(\Gamma_n)
\qquad\text{and}\qquad
\|\Gamma_n\|_\hs^2
=
\sum_{j\geq 1}\lambda_{(j)}^2(\Gamma_n),
$
we have
$
\ell_n
=
\displaystyle\sum_{j\geq 2}\lambda_{(j)}^2(\Gamma_n)/
     \displaystyle\sum_{j\geq 1}\lambda_{(j)}^2(\Gamma_n).
$
Therefore, $\ell_n$ is the proportion of the squared spectral mass of $\Gamma_n$ lying outside its leading eigendirection. Equivalently, if
$
r_{\mathrm{st}}(\Gamma_n)
:=
\|\Gamma_n\|_\hs^2/
     \|\Gamma_n\|_{\mathrm{op}}^2
$
denotes the stable rank of $\Gamma_n$, then
$
\ell_n
=
1- (r_{\mathrm{st}}(\Gamma_n))^{-1}.
$
Thus, $\ell_n$ measures spectral-dispersion proportion related to stable rank rather than the algebraic dimension of $\mathcal{H}_n$. In particular,
if $\Gamma_n$ has rank one, then $\ell_n=0$, whereas $\ell_n>0$ whenever at least two eigenvalues of $\Gamma_n$ are strictly positive. If $\ell_n$ is small, then most of the squared spectral mass is concentrated in a single direction, and the Gaussian distribution is nearly supported on a one-dimensional subspace. If $\ell_n$ is bounded away from zero, then a nonnegligible proportion of the covariance is distributed over directions orthogonal to the leading eigenvector. The negative powers of $\ell_n$ appearing in Theorem~\ref{thm: GA} quantify the deterioration of the anti-concentration bound as the Gaussian distribution approaches a rank-one configuration.
$\ell_n$ must describe the geometry induced by $\Gamma_n$, rather than the geometry of the ambient Hilbert space alone. Directions belonging to the null space of $\Gamma_n$ do not affect the distribution of $Z_n$ and therefore play no role in the approximation of $\|T_n\|$. Consequently, the ambient space may be infinite-dimensional and $\Gamma_n$ may be singular, provided that its nonzero spectrum has sufficient dispersion beyond the leading eigendirection. This feature allows Theorem~\ref{thm: GA} to accommodate both infinite-dimensional functional data and high-dimensional vectors with approximately low-rank covariance structures.

Theorem~\ref{thm: GA} has two immediate consequences in Proposition~\ref{prop: HGA} and Proposition \ref{prop: GA}.
\begin{proposition}  \label{prop: HGA}
    For a given constant $C > 1,$ independent of $n$, assume the following conditions hold: 
\begin{enumerate}[label={(A.\arabic*)}]
    \item\label{A1} $
        \underset{\iicr}{\max} \Exp\|\widetilde h(\xir)\|^4 < C.
   $
    \item\label{A2} For the covariance operator $\Ga$
    \begin{equation*}
      \lambda_{(2)}(\Ga) > 1/C.
    \end{equation*}
\end{enumerate}
Then we have
    \begin{equation*}
    d_{\mathrm{Kol}}(\|T_n\|,\|Z_n\|) \lesssim n^{-1/8}.
\end{equation*}
\end{proposition}

Proposition~\ref{prop: HGA} provides a Gaussian approximation for $U$-statistics in a fixed separable Hilbert space, that is, $\mathcal{H}_n=\mathcal{H}$ for every $n$, but the distributions, kernels, and covariance operators may still depend on $n$. Under uniform fourth-moment and mild spectral conditions, it yields a dimension-free approximation rate. 
Condition \ref{A1} is mild and assumes that the kernel $\tilde h$ of the $U$-statistic has bounded fourth moments. This condition, along with the Cauchy-Schwarz inequality, also implies that $\trc(\Gamma_n)$ and $ \| \Ga \|_{HS}$ are also bounded above. For ease of understanding of the Conditions \ref{A1} and \ref{A2}, we consider certain cases with a linear kernel $\tilde h$. Condition \ref{A1} is satisfied by several Gaussian processes on $[0,1]$, which include the Brownian bridge and Ornstein-Uhlenbeck processes. As an example, one can consider $X(t) = \sum_{k=1}^{\infty} \sqrt{\lambda_k} Z_k \phi_{k}(t)$, where $Z_{k}'s$ are independent random variables with uniformly bounded fourth moments, for $k \ge 1,$ $\lambda_k = 1/(k^2\pi^2)$ and $\phi_{k}(t) = \sqrt{2}  \sin(k\pi t)$ be an orthonormal basis of $L_2[0,1]$. Condition \ref{A2} implies that at least two eigenvalues of $\Gamma_n$ are bounded away from zero. In particular, Condition \ref{A2} allows the covariance operator $\Gamma_n$ to be singular and some coordinates $\tilde h$ to be nearly or fully degenerate, which includes the possibility of $\tilde h$ concentrating in a subspace of $\h$. 
Overall, it can be said that both Conditions \ref{A1} and \ref{A2} are quite relaxed and encompass many generative processes in infinite-dimensional Hilbert spaces. Some existing literature on functional data analysis requires all eigenvalues of the covariance operator to be strictly positive, see \cite{panaretos2010second, chakraborty2015wilcoxon}, and is thus stronger than our \ref{A2}.  For more discussions on the decaying structure of eigenvalues, see \cite{ritter2000average, lopes2025improved}, and references therein.


Next, we discuss the accuracy of the Gaussian approximation established in Proposition~\ref{prop: HGA}. For the interest of space, we do not discuss all the articles that studied the problem of Gaussian approximation for Hilbert-valued $U$-statistics; instead, we only focus on the ones that obtained an $n^{-1/2}$ rate, see \cite{korolyuk2013theory} for a detailed discussion on the approximation rates. For ease of exposition, consider i.i.d.\ Hilbert-valued random elements \(X_1, X_2, \ldots, X_n\) and the corresponding $U$-statistics in \eqref{defU}. Under assumptions such as 
\(\lambda_{(j)}(\Gamma_n) > 0\) for all \(j \ge 1\), and a finite third moment condition \(\mathbb{E}\|h(X_1, X_2, \ldots, X_r)\|^3 < \infty\), \cite{korolyuk1990rate} obtained a Gaussian approximation with accuracy of order $n^{-1/2}.$ 
Later, \cite{borovskikh1997normal} derived the same rate while only weakening the requirement of infinitely many strictly positive eigenvalues of \(\Gamma\), showing instead that the result holds with only nine positive eigenvalues. On the contrary, our Conditions \ref{A1}--\ref{A2} require only the two largest eigenvalues of $\Ga$ to be positive.  
The proof techniques in \cite{borovskikh1997normal} involve Fourier inversion and a characteristic function smoothing argument tailored to events in $\h$ similar to that in \cite{petrov2012sums}. 
The proof techniques developed in this article also allow us to derive the Gaussian and bootstrap approximation results for $U$-statistics in $\R^d$ under much more relaxed assumptions, beyond those that can already be accommodated by the existing literature on Hilbert-valued $U$ statistics. For high-dimensional real-valued vectors, the result of \cite{borovskikh1997normal} is weaker than our result, in the sense that besides requiring nine positive eigenvalues, the dependence on dimension is $d = o(n^{1/10}),$ whereas our result can incorporate much larger $d$; see Remark \ref{rem: RdGA} for more details. From a comparative standpoint, our weaker convergence rates for Hilbert-valued $U$-statistics come at the cost of weaker assumptions and thus widen the scope of statistical applications. It remains open to see whether, under similar assumptions \ref{A1}--\ref{A2}, the optimal rate of $ n^{-1/2}$ can be achieved, and we leave that for future research. 


Proposition~\ref{prop: HGA} also contains a projective
high-dimensional consequence under its global Hilbert-space fourth-moment
condition. Let $(P_d)_{d\geq 1}$ be a sequence of finite-rank orthogonal
projections converging strongly to the identity operator ${\rm{I}}_d$, and define
$
    T_{n,d}=P_dT_n,\qquad
    \Gamma_{n,d}=P_d\Ga P_d,\qquad
    Z_{n,d}=P_dZ_n,
    \qquad
    \rho_{n,d}=\tr\{({\rm{I}}_d-P_d)\Ga\},
$
where $Z_n\sim N_{\h}(0,\Ga)$. Suppose that
$
    b_{n,d}=\lambda_{(2)}(\Ga)>0
    \qquad\text{and}\qquad
    \rho_{n,d} \le b_{n,d}^2/(16 \|\Ga\|_{\mathrm{op}}).
$
Then Proposition~\ref{prop: HGA}, applied to the projected kernel
$P_dh$, gives
$
    d_{\mathrm{Kol}}\bigl(\|T_{n,d}\|,\|Z_{n,d}\|\bigr)
    \lesssim n^{-1/8},
$
uniformly over all such $d$. Moreover, comparison of the projected and
full Gaussian limits gives
$
    d_{\mathrm{Kol}}\bigl(\|T_{n,d}\|,\|Z_{n,d}\|\bigr)
    \lesssim
    n^{-1/8}+\left(\rho_{n,d}/b_{n,d}\right)^{1/2}.
$
Consequently, no restriction relating the projection dimension $d$ to
the sample size $n$ is required when the Gaussian target is $Z_{n,d}$.
In particular, $d=d_n$ may increase arbitrarily fast relative to $n$.
This is a fixed-energy high-dimensional regime: since
$
    \tr(\Gamma_{n,d})\leq \tr(\Ga)<\infty,
$
both $\|\Gamma_{n,d}\|_{\mathrm{op}}$ and $\|\Gamma_{n,d}\|_\hs$ remain
uniformly bounded. A canonical example is
$
    \Ga=\operatorname{diag}
    \left(1,1,3^{-2},4^{-2},\ldots\right).
$
For the corresponding spectral projections,
$\rho_{n,d}=\sum_{j>d}j^{-2}\asymp d^{-1}$, and hence
$
    d_{\mathrm{Kol}}\bigl(\|T_{n,d}\|,\|Z_{n,d}\|\bigr)
    \lesssim n^{-1/8}+d^{-1/2}.
$ 
Thus relative to the dimension-d Gaussian target $Z_{n,d}$, the approximation rate is $n^{-1/8}$ with no restriction on $d$ and $n$, but relative to the full Gaussian target $Z_n$, the rate $n^{-1/8}$ is achievable when $d \gtrsim n^{1/4}$.

In the following proposition, we present the Gaussian approximation result for $U$-statistics in $\R^d$. Here the abstract moment and operator quantities appearing in Theorem~\ref{thm: GA} are converted into explicit rates involving $n$, $d_n$, through the growth of the Frobenius and operator norms of the covariance matrix.

\begin{proposition}  \label{prop: GA}
For a measurable and symmetric kernel $\widetilde h = (\widetilde h_1,\dots,\widetilde h_d)^\top: \mathcal{X}^r  \to \R^d $ of order $r$, we consider the $U$-statistics 
in \eqref{defU} and the corresponding covariance matrix  $\Ga  = \frac{1}{n} \sni \Exp \big[ g^{(i_1)}(X_{i_1}) g^{(i_1)}(X_{i_1})^\top \big]$. Suppose, there exist constants $0\leq q_1\leq q_2\le 1,$ independent of $n$ such that 
\(\| \Ga \|_\text{op} \asymp d^{q_1}, \text{ and }\| \Ga \|_\text{F} \asymp d^{q_2}.\)
    
    For a given constant $C > 1$  (independent of n), assume the following conditions hold: 
\begin{enumerate}[label={(B.\arabic*)}]
    \item\label{B1} $\underset{1\leq m\leq d}{\max}\; \underset{\iicr}{\max} \Exp|\widetilde h_m(\xir)|^4 < C.$
    \item\label{B2} For the covariance matrix $\Ga$
    \begin{equation*}
      \lambda_{(2)}(\Ga) > 1/C.
    \end{equation*}
\end{enumerate}
Under Conditions \ref{B1}--\ref{B2}, we have
    \begin{equation*}
    d_{\mathrm{Kol}}(\|T_n\|,\|Z_n\|) \lesssim \Delta_{n,d},
\end{equation*}
where
\begin{equation} \label{delta_nr}
    \Delta_{n,d} = \max\Bigg\{ \Bigg(\frac{d^{2-2q_2}}{n \ell_n^{3/4}}\Bigg)^{1/4}, \Bigg(\frac{d^{3+3q_1-6q_2}}{n \ell_n^{3/2}}\Bigg)^{1/8}, \Bigg(\frac{d^{3+q_1-4q_2}}{n \ell_n}\Bigg)^{1/6}\Bigg\}.
\end{equation}
\end{proposition}
Proposition~\ref{prop: GA} addresses the more general Euclidean triangular-array setting under only coordinatewise moment conditions, at the cost of dimension-dependent rates. The Gaussian approximation incorporates the finite-dimensional specialization $\mathcal{H}_n=\mathbb{R}^{d_n}$, where the dimension $d_n$ may increase with $n$. The assumptions \ref{B1}--\ref{B2} constitute a mild and practically relevant set of conditions under which the Gaussian approximation in Proposition~\ref{prop: GA} holds for nondegenerate $r$-th order $U$-statistics in the high-dimensional regime. Condition~\ref{B1} imposes only a uniform fourth-moment bound on the kernel coordinates.  This requirement is considerably weaker than the sub-exponential or Orlicz-type moment assumptions commonly imposed in the high-dimensional Gaussian approximation literature for hyperrectangles in \cite{chen2018gaussian, chen2019randomized, chernozhukov2023high}, and therefore accommodates substantially heavier-tailed settings. While Assumption~\ref{B1} is stated at the kernel level, for quadratic kernels it implicitly corresponds to requiring finite eighth moments of the underlying random vectors.  This is similar to the moment assumptions in earlier high-dimensional Gaussian approximation and covariance-testing results in \cite{li2012two} as they require finiteness of the eighth moments together with a strong factorization property for all mixed moments of orders four or more. Additionally, \cite{li2012two} also cannot incorporate covariance structures like $\Gamma_n = (1-\rho){\rm{I}}_d+\rho 1_d 1_d^\top$ for some $\rho \in (0,1),$ since they require $\tr(\Gamma_n^4) = o(\tr^2(\Gamma_n^2)).$ But for this choice of $\Gamma_n$, $\tr(\Gamma_n^4)/\tr^2(\Gamma_n^2) = C$ and thus the condition in \cite{li2012two} is not satisfied, and therefore the limiting distribution remains unclear. \cite{han2020test} derived results related to the Chi-square approximation of the covariance matrices, which require finite sixth moments, vanishing third moments, and fourth-order summability of the cumulants induced by a finite-order polynomial expansion of the i.i.d. innovations of $X_i$. Moreover, even under these structural assumptions, the best available non-asymptotic Chi-square approximation rate in \cite{han2020test} is of order $n^{-1/14}$, and this rate further deteriorates when the moment conditions are relaxed to the existence of only $(4+\delta)$-th moments for some $\delta>0$. Notably, Remark 2.1  of \cite{han2020test} conjectured that a Stein's method-based approach could lead to a sharper chi-square approximation rate, and the present article contributes to that.  

Condition~\ref{B2} imposes only a minimal spectral requirement, namely that the two largest eigenvalues of the covariance matrix $\Ga$ are bounded away from zero. This is substantially weaker than conditions commonly assumed in the literature, which often require either a uniformly positive lower bound on all coordinatewise variances or strict positivity of the entire eigen-spectrum, see \cite{fan2015power, cai2013two, chen2018gaussian, chen2019randomized, Bickel2008a, Rothman01032009}.  In contrast, Condition \ref{B2} is compatible with settings in which most directions exhibit negligible variance, including approximate factor models and other effective low-rank structures. Conditions akin to \ref{B2} can be found in \cite{lopes2020bootstrapping, lin2023high}, where the authors emphasized that valid inference can still be achieved even when only a few dominant directions contribute substantially to the signal.  

To get a better idea about the approximation rates, we explore the bounds under the following setups.

\begin{remark} \label{rem: RdGA}
     We consider matrices with bounded entries, which holds due to Assumption \ref{B1}, and have the second largest eigenvalue bounded away from zero, which is due to Assumption \ref{B2}. Under the setup in Proposition \ref{prop: GA}, consider the following choices of the covariance matrix $\Ga$, where
    \begin{enumerate}[label={(\alph*)}]
    \item\label{(i)} $\|\Ga\|_{op} \asymp 1$ and $\|\Ga\|_F \asymp \sqrt{d}$ which subsequently implies that $\ell_n \asymp 1$. Then we have $\Delta_{n,d} \lesssim \max\{(d/n)^{1/6}, (1/n)^{1/8}\}.$ Examples of such $\Ga$ include $\Ga = {\rm{I}}_d$ or any diagonal matrix with non-zero and uniformly bounded diagonal elements or any first-order autocorrelation matrix, where $\Ga = (\rho^{|i-j|})^{d\times d},$ with $\rho\in (0,1).$
    \item\label{(ii)} $\Ga = EQ_{\rho,d} = (1-\rho){\rm{I}}_d+\rho J_d$ as the equi-correlation matrix, where $\rho\in(0,1),$ where ${\rm{I}}_d$ is an identity matrix and $J_d$ is a matrix with all entries as one. Then, we have $\|\Ga\|_{op} \asymp d$, $\|\Ga\|_F \asymp d$ with $\ell_n \asymp 1/d$ and subsequently $\Delta_{n,d} \lesssim (d^{3/2}/n)^{1/8}.$
\item\label{(iii)}
\(\Gamma_n= \operatorname{BlockDiag} \bigl(EQ_{\rho_1,d_1},\ldots,EQ_{\rho_K,d_K}\bigr),\) where \(\sum_{k=1}^K d_k=d,
\) \(K\ge2\) is fixed and \(\{EQ_{\rho_k,d_k}:1\le k\le K\}\) is a collection of equicorrelation matrices defined as in \ref{(ii)}. Suppose that at least two distinct block sizes are comparable to \(d\). If the corresponding correlations are bounded away from zero, then at least two block-leading eigenvalues are of order \(d\). Consequently,
$\|\Gamma_n\|_{\opr}\asymp d,$
$\|\Gamma_n\|_F\asymp d,$ and
$\ell_n\asymp1$.
Under this setting, \(\Delta_{n,d}\lesssim n^{-1/8},\)
which is independent of \(d\). This conclusion continues to hold when the correlations depend on \(d\), provided that the correlations corresponding to at least two blocks of size comparable to \(d\) remain bounded away from zero. In particular, these correlations may tend to one as \(d\to\infty\).
\end{enumerate}
\end{remark}

\subsection{Bootstrap approximation}\label{sec: boot}
In this section, we consider the bootstrap approximation results for the quantity $T_n$. The bootstrap procedures discussed in this section serve as a crucial tool to overcome the problem of approximating the covariance operator $\Ga$ associated with $Z_n$. 
The first approximation result is a consequence of Theorem \ref{thm: GA}. The second and third approximation steps follow from a Gaussian comparison principle and some moment inequalities for $U$-statistics, which will be made precise later in the proofs. 
We considered three different kinds of bootstraps, namely Empirical Bootstrap (EB), Gaussian Weighted Bootstrap (GWB) and Jackknife Multiplier Bootstrap (JMB). These bootstraps have also been considered in \cite{chen2018gaussian} for $r = 2$ and for an i.i.d kernel $\widetilde h$.  Before stating the result on bootstrap approximation, we briefly discuss the bootstrap procedures in the following subsections. 

\subsubsection{Empirical bootstrap}

Let
$X_1^*,\ldots,X_n^*$ be i.i.d.\ draws from the empirical distribution.
Conditionally on 
$\bm X=\{X_1,\ldots,X_n\}$,
the empirical bootstrap sample can be represented as
$
X_i^*=X_{\xi_i},
$
where $\xi_1,\ldots,\xi_n$ are i.i.d.\ uniformly distributed on
$\{1,\ldots,n\}$. Define
\begin{equation*}
U_n^{\mathrm{EB}}
=
\frac{1}{\npr}
\sum_{\iinr}
\widetilde h(X_{\xi_{i_1}},\ldots,X_{\xi_{i_r}}),
\end{equation*}
and
\begin{equation*}
V_n
:=
\E^\star[U_n^{\mathrm{EB}}]
=
\frac{1}{n^r}
\sum_{\jjcr}
\widetilde h(\xjr).
\end{equation*}
The centered empirical-bootstrap statistic is
\begin{equation*}
T_n^{\mathrm{EB}}
=
\frac{\sqrt n}{r}
\left(U_n^{\mathrm{EB}}-V_n\right).
\end{equation*}

For $j_1\in\{1,\ldots,n\}$, define the conditional first-order
projection
\begin{equation*}
g^{\bm X}(j_1)
:=
\frac{1}{n^{r-1}}
\sum_{1\leq j_2,\ldots,j_r\leq n}
\widetilde h(\xjr)
-
V_n.
\end{equation*}
Then
$
\E^\star[g^{\bm X}(\xi_1)]=0.
$
The conditional Hoeffding decomposition gives
\begin{equation}\label{Hoeff decom}
T_n^{\mathrm{EB}}
=
L_n^{\mathrm{EB}}+R_n^{\mathrm{EB}},
\end{equation}
where
\begin{equation*}
L_n^{\mathrm{EB}}
=
\frac{1}{\sqrt n}
\sum_{i_1=1}^n
g^{\bm X}(\xi_{i_1}).
\end{equation*}
To define the remainder, set
\begin{equation*}
h^{\bm X}(j_1,\ldots,j_r)
:=
\widetilde h(\xjr)-V_n,
\end{equation*}
and
\begin{equation*}
f^{\bm X}(j_1,\ldots,j_r)
:=
h^{\bm X}(j_1,\ldots,j_r)
-
\sum_{c=1}^r g^{\bm X}(j_c).
\end{equation*}
It follows that
\begin{equation*}
\E^\star\!\left[
f^{\bm X}(\xi_1,\ldots,\xi_r)
\mid \xi_c
\right]
=
0,
\qquad c=1,\ldots,r.
\end{equation*}
Moreover,
\begin{equation*}
R_n^{\mathrm{EB}}
=
\frac{1}{r\sqrt n\,P_{n-1,r-1}}
\sum_{\iinr}
f^{\bm X}(\xi_{i_{[1:r]}}).
\end{equation*}

The conditional covariance operator of the linear component is
\begin{equation*}
\widehat\Gamma_n^{\mathrm{EB}}
:=
\operatorname{Cov}^\star\!\left(
g^{\bm X}(\xi_1)
\right)
=
\frac{1}{n}
\sum_{j=1}^n
g^{\bm X}(j)\otimes g^{\bm X}(j).
\end{equation*}
Accordingly,
\begin{equation*}
L_n^{\mathrm{EB}}\mid\bm X
\quad\text{has covariance operator}\quad
\widehat\Gamma_n^{\mathrm{EB}}.
\end{equation*}
We also write
\begin{equation*}
\widehat\Gamma_{n,R}^{\mathrm{EB}}
:=
\E^\star\!\left[
R_n^{\mathrm{EB}}\otimes R_n^{\mathrm{EB}}
\right],
\end{equation*}
for the conditional covariance operator of the empirical-bootstrap
remainder.

Although the resampling variables $\xi_1,\ldots,\xi_n$ are
conditionally i.i.d., the original observations
$X_1,\ldots,X_n$ need not be identically distributed. Consequently,
$\widehat\Gamma_n^{\mathrm{EB}}$ need not automatically estimate
$\Gamma_n$. The discrepancy between these two covariance operators is
therefore retained explicitly in the general bootstrap approximation
bound.

\subsubsection{Gaussian weighted bootstrap}

Let $w_1,\ldots,w_n$ be i.i.d.\ $N(1,1)$ random variables independent
of $\bm X$, and set
$
\varepsilon_i:=w_i-1
$
and
$
\overline w:=n^{-1}\sum_{i=1}^n w_i.
$
Define
\begin{equation*}
U_n^{\mathrm{GWB}}
=
\frac{1}{\npr}
\sum_{\iinr}
w_{i_1}\cdots w_{i_r}\widetilde h(\xir)
-
r(\overline w-1)U_n.
\end{equation*}
Then
\begin{equation*}
T_n^{\mathrm{GWB}}
:=
\frac{\sqrt n}{r}
\left(U_n^{\mathrm{GWB}}-U_n\right).
\end{equation*}
Expanding
$
w_{i_1}\cdots w_{i_r}
=
\prod_{c=1}^r(1+\varepsilon_{i_c})
$
gives
\begin{equation*}
T_n^{\mathrm{GWB}}
=
L_n^{\mathrm{GWB}}+R_n^{\mathrm{GWB}},
\end{equation*}
where
\begin{equation*}
L_n^{\mathrm{GWB}}
=
\frac{1}{\sqrt n}
\sum_{i_1=1}^n
\widehat g^{(i_1)}(X_{i_1})\varepsilon_{i_1}
\end{equation*}
and
\begin{equation*}
R_n^{\mathrm{GWB}}
=
\sum_{k=2}^r T_n^{(k)}.
\end{equation*}
For $k=2,\ldots,r$, define
\begin{equation*}
\begin{split}
T_n^{(k)}\
&:=
\frac{\binom{r}{k}}
{r\sqrt n\,P_{n-1,k-1}}
\sum_{(i_1,\ldots,i_k)\in\mathcal I_n^k}
\Bigg\{
\frac{1}{P_{n-k,r-k}}
\sum_{(i_{k+1},\ldots,i_r)
\in\mathcal I_{n,-\{i_1,\ldots,i_k\}}^{r-k}}
\widetilde h(\xir)
\Bigg\}
\prod_{l=1}^k\varepsilon_{i_l},
\end{split}
\end{equation*}
where
\begin{equation*}
\mathcal I_{n,-\{j_1,\ldots,j_k\}}^{r-k}
:=
\left\{
(i_{k+1},\ldots,i_r)\in\mathcal I_n^{r-k}:
i_{k+1},\ldots,i_r\notin\{j_1,\ldots,j_k\}
\right\}.
\end{equation*}

Conditionally on $\bm X$,
\begin{equation*}
L_n^{\mathrm{GWB}}
\sim
N_{\h}\!\left(0,\widehat\Gamma_n^{\mathrm{GWB}}\right),
\end{equation*}
where
\begin{equation*}
\widehat\Gamma_n^{\mathrm{GWB}}
:=
\frac{1}{n}
\sum_{i_1=1}^n
\widehat g^{(i_1)}(X_{i_1})
\otimes
\widehat g^{(i_1)}(X_{i_1}).
\end{equation*}
We also define
\begin{equation*}
\widehat\Gamma_{n,R}^{\mathrm{GWB}}
:=
\E^\star\!\left[
R_n^{\mathrm{GWB}}\otimes R_n^{\mathrm{GWB}}
\right].
\end{equation*}

\subsubsection{Jackknife multiplier bootstrap}

Let $z_1,\ldots,z_n$ be i.i.d.\ standard normal random variables
independent of $\bm X$. Recall that
\begin{equation*}
T_n^{\mathrm{JMB}}
=
\frac{1}{\sqrt n}
\sum_{i_1=1}^n
\widehat g^{(i_1)}(X_{i_1})z_{i_1}.
\end{equation*}
Conditionally on $\bm X$,
\begin{equation}\label{condJMB}
T_n^{\mathrm{JMB}}
\sim
N_{\h}\!\left(0,\widehat\Gamma_n^{\mathrm{JMB}}\right),
\end{equation}
where
\begin{equation*}
\widehat\Gamma_n^{\mathrm{JMB}}
:=
\frac{1}{n}
\sum_{i_1=1}^n
\widehat g^{(i_1)}(X_{i_1})
\otimes
\widehat g^{(i_1)}(X_{i_1}).
\end{equation*}
Consequently,
\begin{equation*}
\widehat\Gamma_n^{\mathrm{JMB}}
=
\widehat\Gamma_n^{\mathrm{GWB}}.
\end{equation*}

Although the JMB and GWB have the same conditionally Gaussian linear
component, their nonlinear structures are different. The JMB
approximates the H\'ajek projection directly and has no higher-order
bootstrap remainder. The GWB additionally contains the Gaussian-chaos
remainder $R_n^{\mathrm{GWB}}$.

For non-identically distributed observations,
$\widehat\Gamma_n^{\mathrm{JMB}}$ need not automatically estimate
$\Gamma_n$. The approximation theorem below therefore retains the
covariance discrepancy explicitly.

For
$\star\in\{\mathrm{EB},\mathrm{GWB},\mathrm{JMB}\}$, define
\begin{equation*}
d_{\mathrm{Kol}}^\star
\left(\|T_n^\star\|,\|Z_n\|\right)
:=
\sup_{a\geq0}
\left|
\Prob^\star(\|T_n^\star\|\leq a)
-
\Prob(\|Z_n\|\leq a)
\right|.
\end{equation*}
Since $\h$ is separable and the bootstrap statistics are Bochner
integrable, the corresponding conditional expectations are well
defined; see
\cite[pp.~125--128]{vakhania2012probability}. Next, we state the bootstrap Berry--Esseen bound for $\|T_n^\star\|$ and $\|Z_n\|$ for $\star \in \{\text{EB},\text{GWB},\text{JMB} \}$. To the best of our knowledge, bootstrap approximation results for Hilbert-valued $U$-statistics have not been established in the existing literature. 
\begin{theorem} \label{thm: BA}
    Under the conditions of Theorem \ref{thm: GA}, we have
\[
\begin{aligned}
\E \big[d_{\mathrm{Kol}}^\star
(\|T_n^\star\|,\|Z_n\|)\big]
\lesssim &
\ell_n^{-3/16} \mathcal B_{1,n}^\star
+ \ell_n^{-1/6} \left\{ {\mathcal B_{2,n}^\star} + \mathcal D_{1,n}^\star \right\} 
+ \ell_n^{-1/8} \left\{ \mathcal C_{1,n}^\star + \mathcal D_{2,n}^\star \right\},
\end{aligned}
\]
where for $\star = \text{EB}$,
\[
\begin{aligned}
\mathcal B_{1,n}^\star
&:=
\frac{1}{\|\Ga\|_\hs^{3/4}}
\left\{
\frac{\E\|\widehat\Gamma_n^\star\|_\hs\; m_n}{n}
+
\frac{\{\E\|\widehat\Gamma_n^\star\|_{\mathrm{op}}\}^{3/2}m_n^{3/4}}{\sqrt n}
\right\}^{1/4},\\
\mathcal B_{2,n}^\star
&:=
\frac{1}{\|\Ga\|_\hs^{2/3}}
\left\{
\frac{\{\E\|\widehat\Gamma_n^\star\|_{\mathrm{op}}\}^{1/2}m_n^{3/4}}{\sqrt n}
+
\frac{m_n}{n}
\right\}^{1/3},
\end{aligned}
\]
and $\mathcal B_{1,n}^\star = \mathcal B_{2,n}^\star = 0$ for $\star \in \{\text{GWB},\text{JMB} \}$. 

For $\star \in \{\text{EB},\text{GWB} \},$ set
\[
\begin{aligned}
\mathcal C_{1,n}^\star
&:=
\frac{1}{\|\Ga\|_\hs^{1/2}}\left\{ \E\tr(\widehat\Gamma_{n,R}^\star) + \sqrt{ \E\tr(\widehat\Gamma_{n,R}^\star) \E\tr(\widehat\Gamma_n^\star)} \right\}^{1/2},
\end{aligned}
\]
and $\mathcal C_{1,n}^\star = 0$ for $\star = \text{JMB}$. 

For $\star \in \{\text{EB},\text{GWB},\text{JMB} \},$ set
\[
\begin{aligned}
\mathcal D_{1,n}^\star
&:=
\frac{1}{\|\Ga\|_\hs^{2/3}}
\left\{
\sqrt{
\big\{\E\operatorname{tr}(\widehat\Gamma_n^\star)
+\operatorname{tr}(\Gamma_n)\big\}
\big\{\E\|\widehat\Gamma_n^\star\|_{\mathrm{op}}
+\|\Gamma_n\|_{\mathrm{op}}\big\}
}
\,
\E\|\widehat\Gamma_n^\star-\Gamma_n\|_\hs
\right\}^{1/3},\\
\mathcal D_{2,n}^\star
&:=
\frac{1}{\|\Ga\|_\hs^{1/2}}
\left\{\E\left|
\operatorname{tr}(\widehat\Gamma_n^\star-\Gamma_n)
\right|\right\}^{1/2}.
\end{aligned}
\]
\end{theorem}
Theorem~\ref{thm: BA} separates the bootstrap approximation error into three conceptually different components. 
The terms $\mathcal B_{1,n}^{\star}$ and $\mathcal B_{2,n}^{\star}$ measure the error incurred when the conditional distribution of the linear component of the bootstrap statistic is replaced by a Gaussian distribution. These terms appear only for the empirical bootstrap. Indeed, conditionally on the data, the first-order projection of the empirical bootstrap remains a normalized sum of non-Gaussian random elements sampled from the empirical distribution. By contrast, the linear component of the Gaussian weighted bootstrap is conditionally Gaussian, while the entire jackknife multiplier bootstrap statistic is conditionally Gaussian. Consequently,
$
\mathcal B_{1,n}^{\star}=\mathcal B_{2,n}^{\star}=0,
\qquad
\star\in\{\mathrm{GWB},\mathrm{JMB}\}.
$
The term $\mathcal C_{1,n}^{\star}$ controls the contribution of the higher-order bootstrap remainder. Such a remainder is present for both the empirical and Gaussian weighted bootstraps. For the Gaussian weighted bootstrap, it consists of the higher-order Gaussian-chaos terms generated by products of two or more centered weights. The jackknife multiplier bootstrap, on the other hand, is constructed directly from the estimated first-order projections and therefore has no higher-order bootstrap remainder. Hence,
$
\mathcal C_{1,n}^{\mathrm{JMB}}=0.
$
This distinction is structural as the JMB approximates the distribution of the H\'ajek projection directly, whereas the EB and GWB reproduce, to different extents, the nonlinear structure of the original $U$-statistic.

Finally, $\mathcal D_{1,n}^{\star}$ and $\mathcal D_{2,n}^{\star}$ quantify the discrepancy between the bootstrap covariance operator $\widehat{\Gamma}_n^{\star}$ and the target covariance operator $\Gamma_n$. These terms arise from the Gaussian comparison between
$N_{\mathcal{H}_n}\bigl(0,\widehat{\Gamma}_n^{\star}\bigr)$ and $N_{\mathcal{H}_n}(0,\Gamma_n)$
and therefore are common to all three bootstrap procedures. Thus, even when the bootstrap statistic is conditionally Gaussian, as in the JMB case, the validity of the procedure still depends on estimating the covariance geometry of the H\'ajek projection with sufficient accuracy. In particular, the relevant estimation problem is to simultaneously control the trace, Hilbert--Schmidt, and operator-norm characteristics that determine the distribution and anti-concentration of the Gaussian norm.

For the bootstrap results under independent but non-identically
distributed observations, we impose the following first-order balance
condition. 
For each $i_1=1,\ldots,n$, define
\begin{equation*}
\bar{\theta}_{i_1}
:=
\frac{1}{\nopr}
\sum_{\iinri}
\theta_{\bm i}.
\end{equation*}

\begin{enumerate}[label={(A.\arabic*)}, start = 3]
\item\label{A3} For every $i_1=1,\ldots,n$,
\begin{equation*}
\bar{\theta}_{i_1}=\theta.
\end{equation*}
\end{enumerate}

Condition \ref{A3} allows the distributions of the observations to be
different, but rules out first-order heterogeneity in the tuple means. Although the identity
\[
    \frac{1}{n}\sum_{i_1=1}^n \bar{\theta}_{i_1}=\theta,
\]
holds automatically by double counting, Condition~\ref{A3} requires the
stronger indexwise equality $\bar{\theta}_{i_1}=\theta$ for every
$i_1=1,\ldots,n$. This condition ensures that the leave-one-in
estimators of the H\'ajek projections are correctly centered. Without
this, the bootstrap covariance generally contains an additional
component generated by the deterministic variation of
$\bar{\theta}_{i_1}$ across the observation indices.

To make the role of Condition~\ref{A3} precise, recall that the
population first-order projection corresponding to $X_{i_1}$ is
\begin{equation*}
g^{(i_1)}(x)
:=
\frac{1}{\nopr}
\sum_{\iinri}
\left\{
\E\left[
\widetilde h(\xir)\mid X_{i_1}=x
\right]
-
\theta_{\bm i}
\right\}.
\end{equation*}
By construction,
\begin{equation*}
\E\left[g^{(i_1)}(X_{i_1})\right]=0,
\qquad i_1=1,\ldots,n,
\end{equation*}
and the covariance operator of the H\'ajek projection is
\begin{equation*}
\Ga
=
\frac{1}{n}\sum_{i_1=1}^n
\E\left[
g^{(i_1)}(X_{i_1})
\otimes
g^{(i_1)}(X_{i_1})
\right].
\end{equation*}
If the conditional kernel expectation is instead centered by the
common parameter $\theta$, then
\begin{align*}
\widetilde g^{(i_1)}(x)
&:=
\frac{1}{\nopr}
\sum_{\iinri}
\E\left[
\widetilde h(\xir)\mid X_{i_1}=x
\right]
-\theta \\
&=
g^{(i_1)}(x)+b_{i_1},
\qquad
b_{i_1}:=\bar{\theta}_{i_1}-\theta.
\end{align*}
Consequently,
\begin{equation*}
\E\left[\widetilde g^{(i_1)}(X_{i_1})\right]
=
b_{i_1}.
\end{equation*}
Since $g^{(i_1)}(X_{i_1})$ is centered, we obtain
\begin{equation*}
\frac{1}{n}\sum_{i_1=1}^n
\E\left[
\widetilde g^{(i_1)}(X_{i_1})
\otimes
\widetilde g^{(i_1)}(X_{i_1})
\right]
=
\Ga+B_n,
\end{equation*}
where
\begin{equation*}
B_n
:=
\frac{1}{n}\sum_{i_1=1}^n
b_{i_1}\otimes b_{i_1}
=
\frac{1}{n}\sum_{i_1=1}^n
(\bar{\theta}_{i_1}-\theta)
\otimes
(\bar{\theta}_{i_1}-\theta).
\end{equation*}
The operator $B_n$ is nonnegative and satisfies
\begin{equation*}
\operatorname{tr}(B_n)
=
\frac{1}{n}\sum_{i_1=1}^n
\left\|\bar{\theta}_{i_1}-\theta\right\|^2.
\end{equation*}
Therefore, $B_n=0$ if and only if
$\bar{\theta}_{i_1}=\theta$ for every $i_1$, as required by
Condition~\ref{A3}.

The same identity summarizes the corresponding sample-level and
Gaussian-target consequences. Indeed, the leave-one-out coefficient
used in the jackknife multiplier bootstrap is
\begin{equation*}
\widehat g^{(i_1)}(X_{i_1})
=
\frac{1}{\nopr}
\sum_{\iinri}
\widetilde h(\xir)-U_n,
\end{equation*}
and it satisfies
\begin{equation*}
\E\left[\widehat g^{(i_1)}(X_{i_1})\right]
=
b_{i_1}.
\end{equation*}
Thus, without Condition~\ref{A3}, the bootstrap coefficients contain
deterministic between-index offsets. Up to the usual estimation and
remainder terms, the jackknife multiplier bootstrap and the linear
component of the Gaussian weighted bootstrap consequently target the
second-moment operator $\Ga+B_n$, rather than the covariance operator
$\Ga$ of the H\'ajek projection. Equivalently, the associated Gaussian
target is $\mathcal{N}_{\h}(0,\Ga+B_n)$ instead of
$\mathcal{N}_{\h}(0,\Ga)$. The pooled empirical bootstrap encounters
the same structural issue because it does not distinguish stochastic
within-index variation from deterministic variation in the quantities
$\bar{\theta}_{i_1}$; see \cite{han2018inference}. Hence,
Condition~\ref{A3} prevents deterministic between-index heterogeneity
from being incorporated as an additional covariance component.

\begin{remark}
The exact equality in Condition~\ref{A3} may be replaced by an
asymptotic balance condition requiring $B_n$ to be negligible under
the covariance-comparison normalization used in the bootstrap
theorem. Since $B_n$ is nonnegative,
$
\|B_n\|_{\mathrm{op}}
\leq
\|B_n\|_{\hs}
\leq
\operatorname{tr}(B_n),
$
a convenient sufficient condition is
$
\frac{\operatorname{tr}(B_n)}
     {\|\Ga\|_{\hs}}
=
\frac{
n^{-1}\sum_{i_1=1}^n
\|\bar{\theta}_{i_1}-\theta\|^2
}{
\|\Ga\|_{\hs}
}
\to 0.
$
A stronger rate may be required depending on the anti-concentration
factor appearing in the non-asymptotic bootstrap bound. Under this
weaker formulation, the bound must include additional
covariance-comparison terms involving $\operatorname{tr}(B_n)$,
$\|B_n\|_{\hs}$, and $\|B_n\|_{\mathrm{op}}$. We impose the exact
balance condition to avoid these additional terms and retain a cleaner
finite-sample formulation.
\end{remark}
The following proposition establishes the bootstrap approximation for $\h $ for every $n$.

\begin{proposition}\label{prop: HBA}
Suppose that Conditions \textnormal{\ref{A1}}--\textnormal{\ref{A3}} hold.
Then
\begin{equation*}
\E\left[
d_{\mathrm{Kol}}^\star
\big(
\|T_n^\star\|,\|Z_n\|
\big)
\right]
\lesssim
\begin{cases}
n^{-1/8}, & \star=\mathrm{EB},\\
n^{-1/6}, & \star\in\{\mathrm{GWB},\mathrm{JMB}\}.
\end{cases}
\end{equation*}
\end{proposition}
Theorem~\ref{thm: BA} and Proposition~\ref{prop: HBA} show that, in a fixed separable Hilbert space, the covariance-comparison terms can be controlled without imposing any further finite-rank or a prescribed rate of eigenvalue decay condition in \cite{lopes2025improved, bunea2015fPCA, kolt20, lopes2023bootstrapping, kolt2017a}. Under Conditions~\ref{A1}--\ref{A2}, the global fourth-moment bound on the kernel implies uniform control of the trace and Hilbert--Schmidt norm of $\Gamma_n$, while the lower bound on $\lambda_{(2)}(\Gamma_n)$ prevents the covariance from approaching a rank-one configuration. More precisely,
\[
\ell_n
=
\frac{\displaystyle\sum_{j\geq 2}\lambda_{(j)}^2(\Gamma_n)}
     {\displaystyle\sum_{j\geq 1}\lambda_{(j)}^2(\Gamma_n)}
\geq
\frac{\lambda_{(2)}^2(\Gamma_n)}
     {\|\Gamma_n\|_\hs^2},
\]
so Conditions~\ref{A1}--\ref{A2} imply that $\ell_n$ is bounded away from zero uniformly in $n$. The anti-concentration factors appearing in Theorem~\ref{thm: BA} are therefore stable even when $\Gamma_n$ is singular or has infinitely many nonzero eigenvalues.

The dimension-free rates in Proposition~\ref{prop: HBA} are obtained from the covariance estimates
$
\mathbb{E}
\bigl\|
\widehat{\Gamma}_n^{\star}-\Gamma_n
\bigr\|_\hs
\lesssim n^{-1/2},
\qquad
\mathbb{E}
\left|
\operatorname{tr}\bigl(\widehat{\Gamma}_n^{\star}\bigr)
-
\operatorname{tr}(\Gamma_n)
\right|
\lesssim n^{-1/2},
$
together with uniform control of
$\mathbb{E}\operatorname{tr}
\bigl(\widehat{\Gamma}_n^{\star}\bigr)
\quad\text{and}\quad
\mathbb{E}
\bigl\|
\widehat{\Gamma}_n^{\star}
\bigr\|_{\mathrm{op}}.$
This explains why the conclusion remains valid in an infinite-dimensional Hilbert space.
The difference between the empirical-bootstrap rate and the multiplier-bootstrap rates has a structural origin. For the empirical bootstrap, the additional conditional Gaussian approximation of its non-Gaussian first-order projection produces the rate $n^{-1/8}$.
For the GWB and JMB, since this conditional linear component is already Gaussian, the leading bootstrap-to-Gaussian error is instead generated by covariance estimation and is of order
$n^{-1/6}.$

\begin{proposition} \label{prop: BA}
    In addition to the setup and conditions of Proposition \ref{prop: GA}, suppose that Condition~\textnormal{\ref{A3}} holds and $\tr(\Ga)\asymp d^{q_3},$ for some $q_2\leq q_3\leq 1.$ Then we have
    \begin{equation*}
\E\big[d_{\mathrm{Kol}}^\star(\|T_n^\star\|,\|Z_n\|)\big] \lesssim \Delta_{n,d}^*,
\end{equation*}
where
\begin{equation} \label{delta_n*r}
    \Delta_{n,d}^{*} = \max\Bigg\{ \Bigg(\frac{d^{2-2q_2}}{n \ell_n^{3/4}}\Bigg)^{1/4}, \Bigg(\frac{d^{3+3q_1-6q_2}}{n \ell_n^{3/2}}\Bigg)^{1/8}, \Bigg(\frac{d^{3+q_1-4q_2}}{n \ell_n}\Bigg)^{1/6}, \Bigg(\frac{d^{2+\frac{2}{3}q_3-\frac{8}{3}q_2}}{n \ell_n^{2/3}}\Bigg)^{1/4} \Bigg\}.
\end{equation}
Furthermore, if we assume
\begin{enumerate}[label={(B.\arabic*)}, start=3]
\item\label{BN} For all $u\in \R^d,$
\[\max_{1\leq i\leq n}\|u^\top g^{(i)}(X_i)\|_{\psi_2} \leq C \sqrt{u^\top \Ga u}.\]
\end{enumerate}
Then we have 
\begin{equation*}
\E\big[d_{\mathrm{Kol}}^\star(\|T_n^\star\|,\|Z_n\|)\big] \lesssim \Delta_{n,d},
\end{equation*}
where $\Delta_{n,d}$ is as defined in Proposition~\ref{prop: GA}.
\end{proposition}

The first three terms in $\Delta_{n,d}^*$ coincide with the terms governing the Gaussian approximation in Proposition~\ref{prop: GA}. The additional term 
is the price for estimating the covariance geometry under only coordinatewise fourth-moment assumptions. In particular, Condition~\ref{B1} gives
$
\mathbb{E}
\bigl\|
\widehat{\Gamma}_n^{\star}-\Gamma_n
\bigr\|_{\mathrm{F}}
\lesssim
\frac{d}{\sqrt{n}},
$
and
$
\sum_{m=1}^d
\mathbb{E}
\left|
\widehat{\Gamma}_{n,mm}^{\star}
-
\Gamma_{n,mm}
\right|
\lesssim
\frac{d}{\sqrt{n}}.
$
The accumulation of these coordinatewise estimation errors is responsible for the fourth term in $\Delta_{n,d}^*$. It is therefore a covariance-estimation effect, rather than an additional error caused by the higher-order Hoeffding remainder.

Condition~\ref{BN} strengthens the control of the first-order projection in every direction:
$
\max_{1\leq i\leq n}
\bigl\|
u^\top g^{(i)}(X_i)
\bigr\|_{\psi_2}
\leq
C\sqrt{u^\top\Gamma_n u},
\qquad
u\in\mathbb{R}^d.
$
This condition is imposed on the H\'ajek projection $g^{(i)}(X_i)$, rather than on the full kernel $\widetilde h$. Its role is specifically to sharpen the operator-norm control of the bootstrap covariance estimator:
$
\mathbb{E}
\bigl\|
\widehat{\Gamma}_n^{\star}
\bigr\|_{\mathrm{op}}
\lesssim
d^{q_1}.
$
It is not used to control the higher-order degenerate components of the $U$-statistic, which remain governed by the fourth-moment condition on the kernel. Under Condition~\ref{BN}, the covariance-estimation terms are no larger than the terms already present in the Gaussian approximation, and the bootstrap-to-Gaussian error therefore has the same order $\Delta_{n,d}$ as the Gaussian approximation in Proposition~\ref{prop: GA}. Thus, the bootstrap matches the Gaussian approximation rate whenever the additional covariance-estimation term is dominated by the first three terms in $\Delta_{n,d}^*$.

To have a better understanding of the bounds of the bootstrap approximation rate, we explore the bounds under the following setups as considered in Remark~\ref{rem: RdGA}.

\begin{remark}\label{rem: RdBA}
    Note that the rates of the bootstrap approximations can be slower than the Gaussian approximation rates in Proposition \ref{prop: GA} in a few cases where the last term in $\Delta_{n,d}^*$ dominates the other terms. We compare the bootstrap approximation rates through the examples considered in Remark~\ref{rem: RdGA}.
    \begin{enumerate}[label={(\alph*)}]
    \item\label{(iv)} Consider $\Ga$ as in Remark~\ref{rem: RdGA}\ref{(i)}. Then $\tr(\Ga)\asymp d$ and we have $\Delta_{n,d}^{*} \lesssim $ $ \max\{(d^{4/3}/n)^{1/4}, $ $ (1/n)^{1/8}\},$ which is slower than the Gaussian approximation case if $d \gtrsim n^{3/8}.$
    \item\label{(v)} Consider $\Ga = EQ_{\rho} = (1-\rho){\rm{I}}_d+\rho J_d$ as Remark~\ref{rem: RdGA}\ref{(ii)}. In this case, we have $\tr(\Ga) \asymp d$ and we have $\Delta_{n,d}^{*} \lesssim (d^{3/2}/n)^{1/8},$ which is the same as the Gaussian approximation case. 
    \item\label{(vi)} For $\Ga$ as the block diagonal structure considered in Remark~\ref{rem: RdGA}\ref{(iii)}, we have $\tr(\Ga) \asymp d$ and $\Delta_{n,d}^{*} \lesssim (1/n)^{1/8},$ which is again the same as the Gaussian approximation case.
    \end{enumerate}
\end{remark}

\section{Applications to statistical inference}\label{sec: stat app}
This section illustrates how several widely used semi-parametric and non-parametric statistics in high-dimensional inference and functional data analysis can serve as examples of $U$-statistics. Consider the testing problem $H_0: \theta = 0 $ vs $H_a: \theta \neq 0$, where $\theta \in \h$. 
For the testing problem, we use $T_n$ to denote the observable
null-centered statistic
$
T_n:=\sqrt{n}U_n/r.
$
Accordingly, the centered statistic defined in Section~\ref{sec:prelim} is denoted in
this section by
$
T_n^{(0)}:=\sqrt{n}\{U_n-\theta\}/r.
$

Thus,
$T_n=T_n^{(0)}+\mu_n,$
where $\mu_n:=\sqrt{n}\theta/r$.
The testing criteria would be to reject $H_0$ at level of significance $\alpha \in (0,1)$, if $\|T_n\| > q_n^\star(1-\alpha),$ where $q_n^\star(\alpha) = \inf\{ x \ge 0: \Prob^\star( \|T_n^\star\| \le x) \ge \alpha \}$.  The following result establishes the consistency of the proposed test.

\begin{proposition} \label{prop: power}
     Under the setting of Propositions~\ref{prop: HGA} and \ref{prop: HBA}, we have
    \begin{equation*}
        \Prob_{H_0}( \|T_n\| > q_n^\star(1-\alpha) ) \to \alpha.
    \end{equation*}
     
     Furthermore, if there exists a sequence $a_n$ of positive numbers diverging  to infinity with $n,$ such that $\|\theta\| \ge C_{\alpha, r} \frac{ \{\tr (\Ga^2)\}^{1/4}}{\sqrt{n}}  a_n ,$ then we have,
    \begin{equation*}
        \Prob_{H_a}( \|T_n\| > q_n^\star(1-\alpha) ) \to 1,
    \end{equation*}
    where $C_{\alpha, r}$ is a positive constant that depends only on $\alpha$ and the order $r$.
\end{proposition}
Under the settings of Propositions~\ref{prop: GA} and \ref{prop: BA}, this result also holds in $\R^d$ provided $\Delta_{n,d}^*\to 0.$

\begin{remark}\label{rem: concistency} The consistency statement in Proposition~\ref{prop: power} is most naturally interpreted on the squared-norm scale. Under the alternative, write
$
T_n
=
T_n^{(0)}+\mu_n,
$
where
$
T_n^{(0)}
:=
\sqrt{n}(U_n(\tilde{h})-\theta)/{r},
$ and 
$
\mu_n
:=\sqrt{n}\theta/{r}.
$
The centered component $T_n^{(0)}$ is approximated by
$Z_n\sim N_{\mathcal H_{n}}(0,\Gamma_n)$, so the corresponding Gaussian
approximation under the alternative is $Z_n+\mu_n$. The squared Gaussian
norm satisfies
$
\mathbb{E}\|Z_n\|^2
=
\operatorname{tr}(\Gamma_n),
$
and
$
\operatorname{Var}\bigl(\|Z_n\|^2\bigr)
=
2\operatorname{tr}(\Gamma_n^2)
=
2\|\Gamma_n\|_\hs^2.
$
Moreover,
$
\|Z_n+\mu_n\|^2
=
\|Z_n\|^2
+
2\langle Z_n,\mu_n\rangle
+
\|\mu_n\|^2.
$
Thus, the deterministic contribution of the alternative on the squared
scale is
$
\|\mu_n\|^2
=
n\|\theta\|^2/{r^2},
$
whereas the stochastic fluctuation of the squared null norm is of order
$
\bigl\{\operatorname{tr}(\Gamma_n^2)\bigr\}^{1/2}
=
\|\Gamma_n\|_\hs.
$
This leads to the signal-to-noise ratio
$
\operatorname{SNR}_n(\theta)
:=
n\|\theta\|^2/
     ({r^2\{\operatorname{tr}(\Gamma_n^2)\}^{1/2}}).
$
The separation assumption in Proposition~\ref{prop: power},
$
\|\theta\|
\geq
(C_{\alpha, r}
\{\operatorname{tr}(\Gamma_n^2)\}^{1/4}a_n)/{\sqrt{n}}
$
as
$
a_n\to\infty,
$
implies
$
\operatorname{SNR}_n(\theta)
\geq
C_{\alpha, r}^2a_n^2
\to \infty.
$
Consequently, the deterministic displacement of the squared norm
dominates both the null fluctuation and the random cross term
$2\langle Z_n,\mu_n\rangle$, yielding power converging to one.

It is therefore useful to define the separation radius
$
\rho_n(\Gamma_n)
:= \|\Gamma_n\|_\hs^{1/2}/\sqrt{n}.
$
Proposition~\ref{prop: power} proves consistency whenever
$
\|\theta\|/{\rho_n(\Gamma_n)}
\to\infty.
$
The scale $\rho_n(\Gamma_n)$ is the natural transition scale suggested by
the Gaussian proxy: alternatives satisfying
$\|\theta\|\asymp\rho_n(\Gamma_n)$ are expected to have nontrivial local
power, whereas the proposition establishes uniform consistency above this
scale by a diverging factor. The result is direction-free, since it depends
on $\theta$ only through $\|\theta\|$. Accordingly, it gives a uniform
sufficient condition over all directions of departure from the null. For a
specific direction lying in a low-variance eigenspace of $\Gamma_n$, a
sharper direction-dependent separation condition may be possible.

The spectral dependence of the separation radius is also informative. In $\R^d$, if
$\Gamma_n$ has $d$ eigenvalues of constant order, then
$
\operatorname{tr}(\Gamma_n^2)\asymp d
$
and hence
$
\rho_n(\Gamma_n)
\asymp d^{1/4}/{\sqrt{n}}.
$
By contrast, if only a fixed number of eigenvalues are nonnegligible, then
$\operatorname{tr}(\Gamma_n^2)\asymp 1$ and the separation radius reduces to
the parametric order $n^{-1/2}$. Thus, the relevant dimensional quantity is
the quadratic spectral mass of the H\'ajek covariance operator, rather
than the algebraic dimension of the ambient Hilbert space. The separation bound obtained here is the same as that in \cite{li2012two}; the detailed discussion of that has been moved to the supplement. 
\end{remark}

The preceding analysis provides a sufficient separation condition for consistency, but does not determine whether this rate can be improved by a different testing procedure. Although it is possible to give a general condition of minimax optimality for a $U$-statistic, it remains an abstract condition since the separation thresholds are highly dependent on the choice of the kernel $\tilde h$. Therefore, we choose Kendall’s tau for exposition purposes. The minimax analysis developed here shows that this separation threshold is rate-optimal up to a universal multiplicative constant.

Let \(d=p(p-1)/2\), and suppose that \(d/n\to0\). For a
\(p\times p\) matrix \(A\), write
\[
\operatorname{vech}(A)
=
\left(A_{jk}:1\le j<k\le p\right)^\top
\in\mathbb R^d,
\]
so that \(\operatorname{vech}(A)\) collects the strictly upper-triangular
entries of \(A\). For \(x,y\in\mathbb R^p\), let
\[
s(x,y)
=
\left(
\operatorname{sign}(x_1-y_1),\ldots,
\operatorname{sign}(x_p-y_p)
\right)^\top,
\]
and define the vector-valued symmetric kernel
\[
h(x,y)
=
\operatorname{vech}\bigl(s(x,y)s(x,y)^\top\bigr)
\in\mathbb R^d.
\]
Equivalently,
\[
h(x,y)
=
\left(
\operatorname{sign}(x_j-y_j)
\operatorname{sign}(x_k-y_k):
1\le j<k\le p
\right)^\top.
\]

For a distribution \(P\) on \(\mathbb R^p\) having continuous marginal
distributions, define the vector of population Kendall's tau coefficients by
\[
\tau(P)
=
\E_P[h(X_1,X_2)]
=
\left(
\E_P\left[
\operatorname{sign}(X_{1j}-X_{2j})
\operatorname{sign}(X_{1k}-X_{2k})
\right]:
1\le j<k\le p
\right)^\top,
\]
where \(X_1\) and \(X_2\) are independent with common distribution \(P\).

The corresponding empirical vector is the order-two \(U\)-statistic
\[
\widehat{\tau}_n
=
\binom{n}{2}^{-1}
\sum_{1\le i<l\le n}h(X_i,X_l).
\]
Equivalently,
\[
\widehat{\tau}_n
=
\left(
\binom{n}{2}^{-1}
\sum_{1\le i<l\le n}
\operatorname{sign}(X_{ij}-X_{l j})
\operatorname{sign}(X_{ik}-X_{lk}):
1\le j<k\le p
\right)^\top.
\]

We consider
\[
H_0:\tau(P)=0_d
\qquad\text{against}\qquad
H_1:\tau(P)\ne0_d,
\]
where $0_d$ is the $d\times 1 $ vector of zeroes.
Define
\[
T_n=\frac{\sqrt n}{2}\widehat{\tau}_n.
\]
Let \(T_n^\star\) denote the corresponding centered bootstrap statistic and
set
\[
q_n^\star(\alpha)
=
\inf\left\{
t\ge0:
\Pp^\star\bigl(\|T_n^\star\|\le t\bigr)\ge\alpha
\right\}.
\]
The bootstrap test is
\[
\widehat\phi_n(\alpha)
=
\mathbf 1\left\{
\frac{\sqrt n}{2}\|\widehat{\tau}_n\|
>
q_n^\star(1-\alpha)
\right\}.
\]

Let $\cP_n$ denote the class of data-generating distributions over which we establish uniform size control under the null and consistency against the specified alternatives for the proposed bootstrap tests, and define
\[
\cP_{0,n}
=
\left\{
P\in\cP_n:
\tau(P)=0_d
\right\}.
\]

Fix a constant \(\eta\in(0,1)\) and for \(\rho>0\), define the local Gaussian
correlation alternatives
\[
\cG_n^{\mathrm{loc}}(\eta,\rho)
=
\left\{
N_p(0,R):
\begin{array}{l}
R\text{ is a positive-definite }p\times p\text{ correlation matrix},\\
\|R-I_p\|_{\opr}\le\eta,
\|\tau(R)\|\ge\rho
\end{array}
\right\},
\]
where
\[
\tau(R)
:=
\tau\bigl(N_p(0,R)\bigr)
=
\left(
\frac{2}{\pi}\arcsin(R_{jk}):
1\le j<k\le p
\right)^\top
\in\mathbb R^d.
\]
Assume that, for all sufficiently large \(n\), \(\cP_n\) contains every
\(N_p(0,R)\) with \(R\) as above and also contains all centered Gaussian
product distributions \(N_p(0,\sigma^2I_p)\), for \(\sigma^2\) in a fixed
neighborhood of one.

\begin{proposition}[Minimax lower bound for the vector of Kendall's tau coefficients]
\label{prop:kendall-minimax}
For every sequence \(c_n\downarrow0\) and every sequence of tests \(\phi_n\)
satisfying
\[
\sup_{P\in\cP_{0,n}}
\E_{P}\phi_n
\le
\alpha+o(1),
\]
one has
\[
\inf_{P\in
\cG_n^{\mathrm{loc}}\left(
\eta,c_n d^{1/4}/\sqrt n
\right)}
\E_{P}\phi_n
\le
\alpha+o(1).
\]
Consequently, no asymptotic level-\(\alpha\) test can be uniformly
consistent at separation radii
$
o\left(\frac{d^{1/4}}{\sqrt n}\right).
$
\end{proposition}
Proposition~\ref{prop:kendall-minimax} provides the complementary minimax lower bound. Indeed,
the calculations in its proof verify that the
conditions of Proposition~\ref{prop: power} hold uniformly for Kendall's tau over the model class under
consideration. 
Consequently, Proposition~\ref{prop: power} implies that, for every sequence
$M_n\to\infty$, the power of the bootstrap Kendall's tau test goes to 1. 
Therefore, combining Proposition~\ref{prop: power} and Proposition~ \ref{prop:kendall-minimax}, we can conclude that $(p/n)^{1/2}$
is the minimax separation rate, up to multiplicative constants, for dense
alternatives in the vector of pairwise Kendall's tau coefficients. 

Proposition \ref{prop: power} applies to both $d$-dimensional and Hilbert-valued $U$-statistics, which find applications to statistical inference problems described in the following.

\subsection{Spearman's rank correlation}
Let $X_1, X_2, \cdots, X_n$ be i.i.d random vectors having continuous marginals in $\R^p$. We denote the element-wise ranks of $X_{iu}$ and $X_{iv}$ as $R_i^u$ and $R_i^v$. Then the Spearman's correlation coefficient (Spearman's rho) is defined as
\begin{align}\label{ex: sprho}
    \rho_{uv} = \frac{\sum_{i=1}^n(R_i^u - (n+1)/2)(R_i^v - (n+1)/2)}
    {\{\sum_{i=1}^n(R_i^u - (n+1)/2)^2 \sum_{i=1}^n(R_i^v - (n+1)/2)^2\}^{1/2}}.
\end{align}
Although \eqref{ex: sprho} in itself is not a $U$-statistic, it can be represented as a $U$-statistic in the following manner. For all $u,v \in\{1,2, \dots, p\}$ such that $u \neq v,$ we write
$$
\rho_{uv} = \dfrac{n-2}{n+1}\hat{\rho}_{uv} + \frac{3}{n+1}\hat{\tau}_{uv},
$$
where $$\hat{\rho}_{uv} = \frac{3}{n(n-1)(n-2)} \sum_{i \neq j \neq k}  sign \{(X_{iu} - X_{ju}) (X_{iv} - X_{kv})\},  $$ 
is a $U$-statistic of order $3$,
and 
$$\hat{\tau}_{uv} = \frac{2}{n(n-1)}\sum_{1 \le i \neq j \le n} sign\{(X_{iu}-X_{ju})(X_{iv}-X_{jv})\},$$
is the Kendall's tau $U$-statistic of order $2$ and $sign(x) = x/|x|$ for any $x \neq 0$ and $sign(0) = 0$. We exclude the diagonal elements as they take deterministic unit values. Define $\rho_n, \hat{\rho}_n$ and $\hat\tau_n$ as the centered vectorized versions of the strictly upper-triangle of the matrices $\sqrt{n}(\rho_{uv})^{p\times p}, \sqrt{n}(\hat\rho_{uv})^{p\times p},$ and $\sqrt{n}(\hat\tau_{uv})^{p\times p}$ respectively. Further, let \(Z_n\sim N(0,\Gamma_n)\), where \(\Gamma_n\) is the Hájek covariance of the order-three Spearman kernel. Then we have the following proposition.

\begin{proposition}\label{prop:kendall}
    Under assumptions \ref{B1}--\ref{B2}, \[d_{\mathrm{Kol}}(\|\rho_n\|,\|Z_n\|) \lesssim \Delta_{n,d},\] where $d=p(p-1)/2.$
\end{proposition}
\begin{proof}
    This proposition is a slight modification of the proof of Theorem~\ref{thm: GA}. Following the derivations in \eqref{bd}, for any $a>0,$
\begin{equation}
    \begin{split}
    d_{\mathrm{Kol}}(\|\rho_n\|,\|Z_n\|)
    &\le \sup_{\eta \in \R}|\E [v_{\eta,a}(\rho_n)  -v_{\eta,a}(Z_n)]|  + \sup_{\eta \ge 0} \Prob(\eta-a < \|Z_n\|^2 \le \eta + a)\\
    &\le \underbrace{\sup_{\eta \in \R}|\E [v_{\eta,a}(\rho_n) -v_{\eta,a}(\hat\rho_n)]|}_{(I)} + \underbrace{\sup_{\eta \in \R}|\E[v_{\eta,a}(\hat\rho_n)- v_{\eta,a}(Z_n)]|}_{(II)} \\ & \qquad + \underbrace{\sup_{\eta \ge 0} \Prob(\eta -a < \|Z_n\|^2 \le \eta + a)}_{(III)},
    \end{split}
\end{equation}
then Lemma~\ref{lem:u_function_results}~(iv) combined with Cauchy-Schwarz inequality yields,
\begin{equation} 
\begin{split}
(I) &\lesssim a^{-1} \left(\E\left[n^{-1}\|\hat\tau_n\|^2 + n^{-1/2}\|\hat\tau_n\|\|\hat\rho_n\|\right]\right) \\
&\le a^{-1} \Big\{ n^{-1} \Exp [\|\hat\tau_n\|^2] + n^{-1/2} \Big(\E[\|\hat\tau_n\|^2] \E[\|\hat\rho_n\|^2]\Big)^{1/2} \Big\},
\end{split}
\end{equation}
and by considering H\'ajek projection for the scaled $U$-statistics $\hat\rho_n$ and $\hat\tau_n$ together with Lemma~\ref{lem: boundlrRd} it follows that $\E\|\hat\rho_n\|^2 \lesssim d$ and $\E\|\hat\tau_n\|^2 \lesssim d$. Then the proposition follows by directly following the proof of Proposition~\ref{prop: GA}.
\end{proof}

Both $\rho_{uv}$ and $\hat{\tau}_{uv}$ have been studied in the existing literature under max-type test statistics in \cite{han2017distribution}, where the limiting distribution is Gumbel and its multiplier bootstrap version can be found in  \cite{chen2019randomized}. Under the assumption that the components of the observed random vector are independent, \cite{li2021central} studied the Marchenko-Pastur limit of Kendall's tau correlation matrix. Proposition \ref{prop:kendall} and Theorem \ref{thm: GA} allow us to consider the Frobenius norm-based inference for both Spearman's correlation and Kendall's tau in high dimensions. 

\subsection{Kendall's tau for elliptical distribution}

Let $X_1, X_2, \cdots, X_T \in \R^{p \times q}$ be $T$ independent and identically distributed data points of a random matrix $X \sim E_{p,q} (M, \Sigma \otimes \Omega, \gamma),$ where $\Sigma$ and $\Omega$ are positive semi-definite matrices with $rank(\Sigma) = m,$ $rank(\Omega) = n$ and $\gamma: [0,\infty) \to \R$. For more details on elliptical distributions, see \cite{gupta2018matrix}. Recently, \cite{he2025new} proposed robust estimators based on Kendall's tau for analyzing heavy-tailed high-dimensional data. The estimators considered there are of the form 
\begin{align}
    K_r = \frac{2}{T(T-1)} \sum_{t < t'} \frac{(X_t - X_{t'})(X_t - X_{t'})^\top}{\|X_t - X_{t'}\|_F^2}, \label{kr} \\
    K_c = \frac{2}{T(T-1)} \sum_{t < t'} \frac{(X_t - X_{t'})^\top(X_t - X_{t'})}{\|X_t - X_{t'}\|_F^2} .\label{kc}
\end{align}
Both \eqref{kr} and \eqref{kc} are shown to be robust against heavy-tailed distributions and serve as an alternative to Tyler's M-estimator \cite{tyler1987distribution} to assess the scatter in high dimensions; see \cite{he2025new} for more details. 

 \subsection{Maximum mean discrepancy (MMD) estimator}

Let $(\mathcal X,\mathcal B)$ be a measurable space and let
$k:\mathcal X\times\mathcal X\to\mathbb R$ be a measurable positive definite kernel with reproducing kernel Hilbert space (RKHS) $\h$. Define the feature map
\[ \varphi(x):=k(\cdot,x)\in\h. \]
Let
\[ X_{11},\ldots,X_{1n}\stackrel{i.i.d.}{\sim}F_1, \qquad X_{21},\ldots,X_{2n}\stackrel{i.i.d.}{\sim}F_2, \]
be independent, and define
\[ Z_i:=(X_{1i},X_{2i}), \qquad i=1,\ldots,n. \]
Set
\[ g(Z_i):=\varphi(X_{1i})-\varphi(X_{2i})\in\h, \]
and define the symmetric kernel
\[ h(z_1,z_2):=\frac12\big(g(z_1)+g(z_2)\big). \]
Consider the second-order Hilbert-valued $U$-statistic
\[ U_n := \frac1{n(n-1)} \sum_{1\le i\neq j\le n} h(Z_i,Z_j). \]
A direct calculation yields
\[ U_n = \frac1n\sum_{i=1}^n \big(\varphi(X_{1i}) - \varphi(X_{2i})\big) =:\hat\mu_1-\hat\mu_2, \]
where $\hat\mu_1$ and $\hat\mu_2$ denote the empirical kernel mean embeddings. Therefore,
\[ \widehat{\mathrm{MMD}}_n := \|\hat\mu_1  -\hat\mu_2\| = \|U_n\|. \]
Although the feature map $\varphi$ need not be explicitly computable, the statistic is computable entirely through the kernel trick
\[ \|U_n\| = \frac1{n} \bigg[\sum_{i,j=1}^n \Big\{ k(X_{1i},X_{1j}) +k(X_{2i},X_{2j}) -2k(X_{1i},X_{2j}) \Big\}\bigg]^{1/2}. \]
The first-order Hoeffding projection is
\[ h_1(z) := \mathbb E\{h(z,Z_{2n})\}-\theta = \frac12\big(g(z)-\theta\big), \]
where $\theta:=\mathbb E h(Z_{1n},Z_{2n})=\mathbb E g(Z_{1n}).$ Hence the kernel is nondegenerate whenever \(E\|g(Z_{1n})-Eg(Z_{1n})\|^2>0.\)
In particular, under $H_0:F_1=F_2$, we have $\theta=0$, but the Hilbert-valued $U$-statistic remains nondegenerate unless the feature map is almost surely constant. This contrasts with the usual unbiased estimator of $\mathrm{MMD}^2$ \cite{gretton2012kernel,zhang2022two}, which is a degenerate scalar $U$-statistic under the null.

\subsection{Mean and covariance operators in Hilbert space} 

Let $X_1, X_2, \cdots, X_n$ be random elements in $\h$ with finite eighth moment. The sample mean (a.k.a. Fr\'echet mean) can be represented as a second order $U$-statistic as 
\begin{align*}
    \bar{X}_n = \frac{1}{n(n-1)} \sum_{1 \le i \neq j \le n} \frac{(X_i + X_j)}{2}.
\end{align*}
Accordingly, the empirical covariance operator using Fr\'echet means can be expressed as 
\begin{align}\label{Frechetcov}
    \hat\Sigma_n = \frac{1}{n(n-1)} \sum_{1 \le i \neq j \le n} \frac{(X_i - X_j)\otimes(X_i - X_j)}{2},
\end{align}
which takes values in the Hilbert space of Hilbert–Schmidt operators on $\h$.
For the metric induced by the Hilbert-space norm, the Fr\'echet mean coincides with the Bochner mean. The corresponding sample mean and sample covariance operator both admit representations as Hilbert-valued $U$-statistics. Multiple works on the Fr\'echet mean for functional datasets have derived asymptotic normality-type results for these quantities in Hilbert spaces. 
We refer the readers to \cite{kraus2012dispersion}, \cite{dubey2019frechet} and references therein. The covariance operator in \eqref{Frechetcov} has been used for testing the equality of the covariance structures of two functional samples in \cite{panaretos2010second}, \cite{boente2011testing}, \cite{fremdt2013testing}.

\subsection{Wilcoxon-Mann-Whitney type test for Hilbert valued observations}

For a location-shift model with $X, X^{'}, Y \in \h$ , let $Y \stackrel{d}{=} X^{'} + \Delta$, where $X^{'}$ is an independent copy of $X$. Consider testing
\begin{align}\label{onesampmed}
    H_0: \Delta = 0 \quad \text{vs} \quad H_a: \Delta \neq 0.
\end{align}
Define $\mu(\Delta) = E[S(Y-X)] = E[S(X^{'} - X + \Delta)]$, where $S(x) = \frac{x}{\|x\|}$ denotes the spatial sign of $x$, with the convention $S(0) = 0$. Because $X^{'} - X$ is centrally symmetric, $\mu(0) = 0$. If the distribution of $X^{'} - X$ is non-atomic and not concentrated on a line, then $\mu(\Delta) \neq 0$ whenever $\Delta \neq 0$. Hence testing $\Delta = 0$ is equivalent to testing $\mu(\Delta) = 0 $ within this location shift model. 
For functional data with $x \in \h$, this spatial median has been studied in \cite{gervini2008robust, chakraborty2015wilcoxon}. 

 To make this precise, let $X_1, X_2, \cdots, X_m$ and $Y_1, Y_2, \cdots, Y_n$ be i.i.d observations from probability measures $F_1$ and $F_2$ in $\h$, where $Y \stackrel{d} = X + \Delta$.
To test \eqref{onesampmed}, one can use Wilcoxon-Mann-Whitney type statistic $W_{m,n} = \frac{1}{mn}\sum_{i=1}^m \sum_{j=1}^n \frac{(Y_j - X_i)}{\|Y_j - X_i\|},$ which is unbiased for $\mu(\Delta)$. $W_{m,n}$ can be expressed as a one-sample $U$-statistic, and therefore the results of our paper can be directly utilized. To proceed, construct a pooled data set $Z_{1n}, Z_{2n}, \cdots, Z_{m+n}$. Next, we define $L_i = 1$ for $i = 1,2,\cdots,m$ and $L_i = 0$ for $i = m+1, m+2, \cdots, m+n$ and the kernel 
\begin{align*}
    h((Z_i, L_i), (Z_j, L_j)) = 1( L_i =1, L_j = 0) \frac{Z_j - Z_i}{\|Z_j - Z_i\|} + 1( L_i = 0, L_j = 1) \frac{Z_i - Z_j}{\|Z_j - Z_i\|}.
\end{align*}
The kernel $h((Z_i, L_i), (Z_j, L_j))$ is a symmetric kernel on pairs $\{(Z_i, L_i)\}_{i=1}^{m+n},$ which are independent but not identical, and therefore
\begin{align*}
    U_{m,n} = \frac{1}{\binom{m+n}{ 2}}\sum_{i < j}h((Z_i,L_i), (Z_j, L_j)). 
\end{align*}
For the two comparable sample sizes of $m$ and $n$, $\binom{m+n}{ 2}/mn = O(1),$ $W_{m,n}$ and $U_{m,n}$ are equivalent up to a constant factor. 

\bibliographystyle{abbrvnat}
\bibliography{references}

\newpage
\appendix

\section{Appendix: Preliminary lemmas} 
This section collects preliminary lemmas that will be used in the proofs of Theorems \ref{thm: GA} and Propositions \ref{prop: HGA} and \ref{prop: GA}.

\begin{lemma}
\label{lem:u_function_results}
The following holds.
    \begin{enumerate}[label=(\roman*)]
        \item There exists a $C^\infty$-function $u:\R \to [0,1]$ such that  $u(x) = 1$ for  $x < 0$ and $u(x) = 0$ for $x > 1$.
        \item Define for any $a\in \R$ and $\tau>0,$ $v_{\eta,\tau}(x) := u(\tau^{-1}(\|x\|^2 - \eta))$ on $\h$, where $u(.)$ is as in Part (i). Then, for any random variable $X \in \h$, 
        \[
        \Prob(\|X\|^2 \leq a^2) \leq \Exp[v_{a^2,\tau}(X)] \leq \Prob(\|X\|^2 \leq a^2 + \tau), \quad \forall a \in \R, \tau > 0.
        \]
        \item For $v_{\eta, \tau}$ as defined in Part (ii), we have \[\nabla v_{\eta,\tau}(x) = 2\tau^{-1} u'(\tau^{-1}(\|x\|^2 - \eta)) x, \;x\in \h,\] and \[\nabla^2 v_{\eta,\tau}(x) = 4 \tau^{-2} u''(\tau^{-1}(\|x\|^2 - \eta)) (x \otimes x) + 2 \tau^{-1} u'(\tau^{-1}(\|x\|^2 - \eta)) {{\rm{I}}}, \; x\in \h, \] where $\nabla v_{\eta, \tau}$ and $\nabla^2 v_{\eta, \tau}$  denotes respectively the gradient and Hessian of the functional $v_{\eta, \tau}$, and ${\rm{I}}$ is the identity operator in $\h$.
        \item For $v_{\eta, \tau}$ as defined in Part (ii), 
        \[
        |v_{\eta, \tau}(x) - v_{\eta, \tau}(y)| \lesssim \tau^{-1} (\|x - y\|^2 + \|x\|\|x - y\|), \quad \forall x,y \in \h.
        \]
    \end{enumerate}
\end{lemma}

\begin{proof}
For claim (i), set $\bar{u}(x) = e^{-1/x}$ for $x > 0$ and $\bar{u}(x) = 0$ for $x \leq 0$. A desired function is given by
\[
u(x) = \frac{\bar{u}(1-x)}{\bar{u}(x) + \bar{u}(1-x)}.
\]

Claim (ii) follows from 
\[
\ind_{\{\|x\|^2 \leq a^2 + \tau\}} \geq  v_{a^2, \tau}(x) \geq \ind_{\{\|x\|^2 \leq a^2\}}.
\]

For (iii) we compute the Fr\'echet derivative. Let
\[
\phi(x) := \tau^{-1}(\|x\|^2 - \eta), \qquad x \in \h.
\]
Then $\phi \in C^\infty(\h)$ with derivative 
\[
\nabla\phi(x)(h) = \frac{2}{\tau} \langle x, h \rangle, \quad \text{for all}\,\, h \in \h.
\]
By the chain rule in Hilbert spaces,
\[
\nabla v_{\eta,\tau}(x)(h) 
= u'(\phi(x))\,\nabla\phi(x)(h) 
= \frac{2}{\tau}\,u'\!\left(\tau^{-1}(\|x\|^2 - \eta)\right) \langle x, h \rangle.
\]
Identifying the Riesz representative, the gradient is
\[
\nabla v_{\eta,\tau}(x) 
= \frac{2}{\tau}\,u'\!\left(\tau^{-1}(\|x\|^2 - \eta)\right) x \in \h.
\]
$\nabla^2 v_{\eta,\tau}(x)$ can be computed in a similar manner and we omit the details for the interest of space. 

Finally, for (iv), by the fundamental theorem of calculus along the segment $x_\theta := x + \theta(y - x)$, $\theta\in[0,1]$, we have
\[
v_{\eta,\tau}(y) - v_{\eta,\tau}(x) 
= \int_0^1 \langle \nabla v_{\eta,\tau}(x_\theta),\, y - x \rangle \, dt.
\]
Hence, using $\|u'\|_\infty < \infty$,
\begin{align*}
|v_{\eta,\tau}(y) - v_{\eta,\tau}(x)| 
&\le \frac{2\|u'\|_\infty}{\tau} \int_0^1 \|x_\theta\|\,\|y - x\| \, dt \\
&\le \frac{2\|u'\|_\infty}{\tau} \, (\|x\| + \|y - x\|) \, \|y - x\| \\
&\lesssim \tau^{-1} \big( \|x\|\,\|x - y\| + \|x - y\|^2 \big),
\end{align*}
which is the desired bound.

\end{proof}

\begin{lemma}[Gaussian approximation for smooth functions]
\label{lem: stein}
Let $\xi_1,\dots,\xi_n$ be independent random elements in $\h$ with mean zero and $\E\|\xi_i\|^4 < \infty$ for all $i = 1,2, \cdots,n$ . Let $\Sigma_n = \sum_{i=1}^n\E[\xi_i\otimes\xi_i]$, $S_n = \sum_{i=1}^n \xi_i$, and $Z_n \sim N_{\h}(0,\Sigma_n)$ that is assumed to be independent of $\xi_i$'s. For a smooth function $v_{\eta,\tau}$ given in Lemma \ref{lem:u_function_results}, we have
\[
\begin{split}
|\E v_{\eta,\tau}(S_n) - \E v_{\eta,\tau}(Z_n)| &\lesssim \tau^{-3}\left (\|\Sigma_n\|_\hs \sum_{i=1}^n\E[\|\xi_i\|^4]+\|\Sigma_n\|_\opr^{3/2}\sum_{i=1}^n\E[\|\xi_i\|^3] \right) \\
&\quad +\tau^{-2}\left ( \| \Sigma_n \|_\opr^{1/2}\sum_{i=1}^n\E[\|\xi_i\|^3] + \sum_{i=1}^n\E[\|\xi_i\|^4]\right).
\end{split}
\]
\end{lemma}

\begin{proof}
The proof follows in a similar spirit to that of
\cite{fang2024large}. We establish the result in the Hilbert-space
setting by explicitly addressing the operator-theoretic issues arising
in infinite dimensions. 
Note that, $\Sigma_n$ is a positive self-adjoint trace-class
operator, since
$
\operatorname{tr}(\Sigma_n)
=
\sum_{i=1}^n \E\|\xi_i\|^2
<\infty.
$
Consequently, $Z_n\sim N_{\h}(0,\Sigma_n)$ is a well-defined Gaussian random
element in $\h$, and $\Sigma_n$ is also Hilbert--Schmidt.

The Ornstein-Uhlenbeck integral  
\begin{align}\label{OrnUhl}
    \mathcal{F}_n(s)=\int_0^1 -\frac{1}{2(1-a)}\,\E\!\big[v_{\eta,\tau}(\sqrt{1-a}\;s+\sqrt{a}\;Z_n)-\E v_{\eta,\tau}(Z_n)\big]\,da.
\end{align}
provides the solution to the Stein equation in the following. 
The Stein equation associated for the function $v_{\eta,\tau}$ and covariance operator $\Sigma_n$ is given as
\begin{align}\label{Steineq}
\operatorname{tr}\!\left(\Sigma_n\nabla^2\mathcal{F}_n(s)\right)
-
\la s,\nabla\mathcal{F}_n(s)\ra
=
v_{\eta,\tau}(s)-\E[v_{\eta,\tau}(Z_n)],
\end{align}
for $s \in \h.$ Here $\nabla \mathcal{F}_n(s)$ is the Fr\'echet gradient of $\mathcal{F}_n$ at $s$ and $\nabla^2 \mathcal{F}_n(s)$ is the Fr\'echet-Hessian which is a symmetric bounded operator on $\h$.
For fixed $\eta$ and $\tau$, $v_{\eta,\tau}$ has bounded Fr\'echet
derivatives of the required orders, so differentiation under the
integral sign is justified by dominated convergence.
Differentiating \eqref{OrnUhl} under the integral sign yields 
$$
\nabla^2 \mathcal{F}_n(s)=\int_0^1 -\frac{1}{2}\,\E\big[\nabla^2 v_{\eta,\tau}(\sqrt{1-a}\;s+\sqrt{a}\;Z_n)\big]\,da.  
$$
For the function $v_{\eta,\tau}(s) = u(\tau^{-1}(\|s\|^2 -\eta))$, from Lemma~\ref{lem:u_function_results}~(iii),  we get
$\nabla v_{\eta,\tau}(s) = 2\tau^{-1}u'(\tau^{-1}(\|s\|^2 -\eta)) s$
and 
\begin{equation}\label{hess-hilbert}
\nabla^2 v_{\eta,\tau}(s)
=4\tau^{-2} u''\!\big(\tau^{-1}(\|s\|^2-\eta)\big)\,(s\otimes s)
+2\tau^{-1} u'\!\big(\tau^{-1}(\|s\|^2-\eta)\big)\,{\rm{I}}.
\end{equation}
Let $S_n^{(i)}=S_n-\xi_i$. Then,
$
\E\la S_n,\nabla\mathcal{F}_n(S_n)\ra
=
\sum_{i=1}^n
\E\la \xi_i,\nabla\mathcal{F}_n(S_n)\ra.
$

A Fr\'echet Taylor expansion of the gradient about $S_n^{(i)}$ yields
\begin{align}\label{covS}
\E\la S_n,\nabla\mathcal{F}_n(S_n)\ra
&=
\sum_{i=1}^n
\E\la \xi_i,\nabla\mathcal{F}_n(S_n^{(i)})\ra
+
\sum_{i=1}^n
\E\la
\xi_i,\nabla^2\mathcal{F}_n(S_n^{(i)})\xi_i
\ra
+
R_1,
\end{align}
where
\begin{align}
R_1
=
\sum_{i=1}^n
\E\la
\xi_i,
\left\{
\nabla^2\mathcal{F}_n(S_n^{(i)}+U\xi_i)
-
\nabla^2\mathcal{F}_n(S_n^{(i)})
\right\}\xi_i
\ra,
\end{align}
and $U$ is uniform on $[0,1]$, independent of all other random
variables. Since $\xi_i$ and $S_n^{(i)}$ are independent and
$\E\xi_i=0$,
$
\sum_{i=1}^n
\E\la \xi_i,\nabla\mathcal{F}_n(S_n^{(i)})\ra
=0.
$
Hence,
\begin{align}
\E\la S_n,\nabla\mathcal{F}_n(S_n)\ra
=
\sum_{i=1}^n
\E\la
\xi_i,\nabla^2\mathcal{F}_n(S_n^{(i)})\xi_i
\ra
+
R_1.
\end{align}

On the other hand,
\begin{align}\label{covstein}
\E\operatorname{tr}\!\left(
\Sigma_n\nabla^2\mathcal{F}_n(S_n)
\right)
&=
\sum_{i=1}^n
\E\la
\xi_i,\nabla^2\mathcal{F}_n(S_n^{(i)})\xi_i
\ra
+
R_2,
\end{align}
where
\begin{align}
R_2
:=
\sum_{i=1}^n
\E\operatorname{tr}\!\left(
\E[\xi_i\otimes\xi_i]
\left\{
\nabla^2\mathcal{F}_n(S_n)
-
\nabla^2\mathcal{F}_n(S_n^{(i)})
\right\}
\right).
\end{align}
Combining \eqref{Steineq}, \eqref{covS}, and \eqref{covstein}, we get
\begin{align}\label{mainbound}
\left|
\E v_{\eta,\tau}(S_n)-\E v_{\eta,\tau}(Z_n)
\right|
\leq
|R_1|+|R_2|.
\end{align}
For $a\in[0,1]$, set
$
S_n^{(i),a}
:=
\sqrt{1-a}\,S_n^{(i)}+\sqrt{a}\,Z_n
$
and
$
\tilde{\xi}_i^a
:=
\sqrt{1-a}\,U\xi_i.
$
Similar to \cite{fang2024large}, we decompose $R_1$ as
$$
R_1 = R_{11} + R_{12} + R_{13} + R_{14},
$$
where 
\begin{align*}
R_{11}
&=-2\tau^{-2}\sum_{i=1}^n\int_0^1 (1-a)\,\E\big[U^2\,u''(\tau^{-1}(\|S_n^{(i),a}+\tilde\xi_i^a\|^2-\eta))\,
\|\xi_i\|^4\big]\,da,\\
R_{12}
&=-4\tau^{-2}\sum_{i=1}^n\int_0^1 \sqrt{1-a}\,\E\big[U\,u''(\tau^{-1}(\|S_n^{(i),a}+\tilde\xi_i^a\|^2-\eta))\,
\|\xi_i\|^2\langle \xi_i, S_n^{(i),a}\rangle\big]\,da,\\
R_{13}
&=-2\tau^{-2}\sum_{i=1}^n\int_0^1
\E\bigg[
\langle \xi_i,S_n^{(i),a}\rangle^2
\bigg\{
u''\!\left(
\tau^{-1}
\big(\|S_n^{(i),a}+\tilde\xi_i^a\|^2-\eta\big)
\right)
-
u''\!\left(
\tau^{-1}
\big(\|S_n^{(i),a}\|^2-\eta\big)
\right)
\bigg\}
\bigg]\,da,\\
R_{14}
&=-\tau^{-1}\sum_{i=1}^n\int_0^1
\E\bigg[
\|\xi_i\|^2
\bigg\{
u'\!\left(
\tau^{-1}
\big(\|S_n^{(i),a}+\tilde\xi_i^a\|^2-\eta\big)
\right)
-
u'\!\left(
\tau^{-1}
\big(\|S_n^{(i),a}\|^2-\eta\big)
\right)
\bigg\}
\bigg]\,da.
\end{align*}
Next, by noting the fact that $\E(U^2) = 1/3$ and $u''$ is bounded, we have
$$
|R_{11}| \le C \tau^{-2} \sum_{i=1}^n \E[\|\xi_i\|^4].
$$
 Application of Lemma 5.9 in \cite{fang2024large} along with the fact that $\|\Sigma_n\|_\opr \le \|\Sigma_n\|_\hs$ yields the following



\begin{align}
|R_{12}| & \le C \tau^{-2} \sum_{i=1}^n \E[\|\xi_i\|^3 \|\Sigma_n\|_{\opr}^{1/2}],\\
|R_{13}| & \le C \tau^{-3} \sum_{i=1}^n \E[\|\Sigma_n\|_\hs\|\xi_i\|^4 + \|\Sigma_n\|_\opr^{3/2} \|\xi_i\|^3],\\
|R_{14}| & \le C \tau^{-2} \sum_{i=1}^n \E[\|\Sigma_n\|_\opr^{1/2}\|\xi_i\|^3 + \|\xi_i\|^4].
\end{align}
Then, the bound for $R_1$ becomes,
\begin{equation}\label{boundR1}
    \begin{split}
        |R_1| & \le C \tau^{-3}\left (\|\Sigma_n\|_\hs\sum_{i=1}^n\E[\|\xi_i\|^4]+\|\Sigma_n\|_\opr^{3/2}\sum_{i=1}^n\E[\|\xi_i\|^3] \right) \\&\hspace{3cm} + C \tau^{-2}\left ( \| \Sigma_n \|_\opr^{1/2}\sum_{i=1}^n\E[\|\xi_i\|^3] + \sum_{i=1}^n\E[\|\xi_i\|^4]\right).
    \end{split}
\end{equation}
Using similar arguments as in $R_1$, it can be shown than
\begin{equation}\label{boundR2}
    \begin{split}
        |R_2| & \le C \tau^{-3}\left (\|\Sigma_n\|_\hs\sum_{i=1}^n\E[\|\xi_i\|^4]+\|\Sigma_n\|_\opr^{3/2}\sum_{i=1}^n\E[\|\xi_i\|^3] \right) \\&\hspace{3cm} + C \tau^{-2}\left ( \| \Sigma_n \|_\opr^{1/2}\sum_{i=1}^n\E[\|\xi_i\|^3] + \sum_{i=1}^n\E[\|\xi_i\|^4]\right).
    \end{split}
\end{equation}
Finally combining \eqref{boundR1}, \eqref{boundR2} and \eqref{mainbound}, the claim follows and hence completes the proof. 
\end{proof}

\begin{lemma}[Gaussian anticoncentration for balls; Theorem 2.7 in \cite{gotze2019large}]
\label{lem: AC}
Let $Z_n$ be a centered Gaussian element in $\h$ with covariance operator $\Sigma_n$ with at least two nonzero eigenvalues. Recall $\Lambda_k (\Sigma_n) = (\sum_{j=k}^d \lambda_{(j)}^2(\Sigma_n))^{1/2}$. Then
\[
\sup_{\eta \ge 0} \Prob (\eta < \| Z_n \|^2 \le \eta+\tau) \lesssim \aleph (\Sigma_n) \tau, \quad \forall \tau > 0, 
\]
where $\aleph (\Sigma_n) = (\Lambda_1 (\Sigma_n)\Lambda_2 (\Sigma_n))^{-1/2}$. 
\end{lemma}
\begin{lemma} \label{lem: GaussianComp}
[Gaussian comparison over centered balls]
Let $Z_{1n}$ and $Z_{2n}$ be two independent Gaussian elements in $\h$ with zero mean and covariance operators $\Sigma_{1n}$ and $\Sigma_{2n}$, respectively. For a smooth function $v_{\eta,\tau}$ given in Lemma \ref{lem:u_function_results}, we have
\[
\begin{split}
|\E v_{\eta,\tau}(Z_{1n}) - \E v_{\eta,\tau}(Z_{2n})| &\lesssim  \tau^{-1} |\tr(\Sigma_{1n}-\Sigma_{2n})| \\&\qquad + \tau^{-2}\sqrt{(\tr(\Sigma_{1n})+ \tr(\Sigma_{2n}))(\|\Sigma_{1n}\| + \|\Sigma_{2n}\|)} \|\Sigma_{1n} -\Sigma_{2n}\|_\hs.
\end{split}
\]
\end{lemma}
\begin{proof}
    The Stein's identity for any Gaussian element $Z_n \in \h$ in Hilbert space yields 
\begin{align}
    \E [\la \nabla \mathcal{F}_n(Z_n), Z_n\ra] = \E[\tr\la \nabla^2\mathcal{F}_n(Z_n), \Sigma_n\ra],
\end{align}
where $\nabla^2\mathcal{F}_n(Z_n)$ is a Hessian operator.
Thus, the Stein's equation in \eqref{Steineq} yields

\begin{align}\label{steinbound_GA_GA}
    | \E(v_{\eta,\tau}(Z_{1n})) - \E(v_{\eta,\tau}(Z_{2n}))| \le |\E[\tr \la \nabla^2\mathcal{F}_n(Z_n) (\Sigma_{1n} - \Sigma_{2n}) \ra]|.
\end{align}
Define the Slepian-Stein interpolant as 
\begin{align}
    Z_n(a) := \sqrt{1-a}Z_{1n} + \sqrt{a}Z_{2n}, \quad a\,\in [0,1].
\end{align}
Then we obtain,
\begin{align}\label{rembound_GA_GA}
    \E[\tr \la \nabla^2\mathcal{F}_n(Z_n) (\Sigma_{1n} - \Sigma_{2n}) \ra] = R_3 + R_4,
\end{align}
where,
$$
R_3 = -2\tau^{-2} \int_{0}^{1} \E[u{''}(\tau^{-1}(\|Z_n(a)\|^2 -\eta))\la Z_n(a), (\Sigma_{1n} - \Sigma_{2n})Z_n(a)\ra]da,
$$
and
$$
R_4 = \tau^{-1} \int_{0}^{1}  \E[u{'}(\tau^{-1}(\|Z_n(a)\|^2 -\eta)) \tr(\Sigma_{1n} - \Sigma_{2n})]da.
$$
Then,
\begin{align}\label{R_2bound_GA_GA}
|R_4| \le C\tau^{-1} |\tr(\Sigma_{1n} - \Sigma_{2n})|.
\end{align}
and by observing the fact that $|\la Z_n(a), (\Sigma_{1n} - \Sigma_{2n})Z_n(a)\ra| \le \|Z_n(a)\| \|(\Sigma_{1n} - \Sigma_{2n})Z_n(a)\|$ and Cauchy-Schwarz inequality yields
$$
|R_3| \le 2\tau^{-2} \int_{0}^{1} \sqrt{\E\|Z_n(a)\|^2 \E\|(\Sigma_{1n} - \Sigma_{2n})Z_n(a)\|^2}da.
$$
Next,
\begin{align*}
\E \|Z_n(a)\|^2 = \tr(\text{Cov}(Z_n(a))) = (1-a)\,\tr(\Sigma_{1n}) + a\,\tr(\Sigma_{2n}),
\end{align*}
and
\begin{align*}
\E\|(\Sigma_{1n} - \Sigma_{2n})Z_n(a)\|^2 & = \E \la (\Sigma_{1n} - \Sigma_{2n})Z_n(a), (\Sigma_{1n} - \Sigma_{2n})Z_n(a) \ra \\
& = \E \la Z_n(a), (\Sigma_{1n} - \Sigma_{2n})^2 Z_n(a)\ra  \\
& = \tr((\Sigma_{1n} - \Sigma_{2n})^2 \text{Cov}(Z_n(a))).
\end{align*}
Thus, we have
\begin{align}\label{R_1bound_GA_GA}
    |R_3| \le C \tau^{-2} \sqrt{(\tr(\Sigma_{1n})+ \tr(\Sigma_{2n}))(\|\Sigma_{1n}\| + \|\Sigma_{2n}\|)} \|\Sigma_{1n} -\Sigma_{2n}\|_\hs.
\end{align}
The final conclusion holds by combining \eqref{steinbound_GA_GA}, \eqref{rembound_GA_GA}, \eqref{R_2bound_GA_GA} and  \eqref{R_1bound_GA_GA} and we get

\begin{equation*}
\begin{split}
\left|
\E v_{\eta,\tau}(Z_{1n})
-
\E v_{\eta,\tau}(Z_{2n})
\right|
\lesssim{}&
\tau^{-2}
\left[
\{\tr(\Sigma_{1n})+\tr(\Sigma_{2n})\}
\{\|\Sigma_{1n}\|_{\opr}+\|\Sigma_{2n}\|_{\opr}\}
\right]^{1/2}
\|\Sigma_{1n}-\Sigma_{2n}\|_{\hs}
\\
&\hspace{7cm}+
\tau^{-1}
\left|
\tr(\Sigma_{1n}-\Sigma_{2n})
\right|.
\end{split}
\end{equation*}
\end{proof}

\newpage
\section{Appendix: Proof of results in Section~\ref{sec:main}}

\textbf{Proof of Theorem~\ref{thm: GA}.}
\label{subsec:lem1_proof}
We will assume without loss of generality that $Z_n$ is independent of $X_i$'s. 
Furthermore, we will assume $\Delta_n \leq 1$,  as otherwise the conclusion holds trivially. We divide the proof into several steps. Let $v_{\eta,\tau}$ be a smooth function on $\h$ given in Lemma \ref{lem:u_function_results}~(ii). 

\medskip

\noindent\textbf{Step 1.}
By Lemma \ref{lem:u_function_results}, observe that for any $a, \tau > 0$,
\[
\begin{split}
\Prob (\|T_n\| \le a) &= \Prob (\|T_n\|^2 \le a^2) \le \E[v_{a^2,\tau}(T_n)] \\
&= \E[v_{a^2,\tau}(Z_n)] + \E[v_{a^2,\tau}(T_n) - v_{a^2,\tau}(Z_n)] \\
&\le \Prob(\|Z_n\|^2 \le a^2+\tau) + \E[v_{a^2,\tau}(T_n) - v_{a^2,\tau}(Z_n)] \\
&= \Prob (\|Z_n\| \le a) + \Prob (a^2 < \|Z_n\|^2 \le a^2 + \tau) + \E[v_{a^2,\tau}(T_n) - v_{a^2,\tau}(Z_n)].
\end{split}
\]
Likewise, one has
\[
\Prob (\|T_n\| \le a) \ge  \Prob (\|Z_n\| \le a) - \Prob (a^2 -\tau< \|Z_n\|^2 \le a^2) + \E[v_{a^2-\tau,\tau}(T_n) - v_{a^2-\tau,\tau}(Z_n)].
\]
Maximizing over $a$ gives
\begin{equation} \label{bd}
    \begin{split}
    d_{\mathrm{Kol}}(\|T_n\|,\|Z_n\|)
    &\le \sup_{\eta \in \R}|\E [v_{\eta,\tau}(T_n)  -v_{\eta,\tau}(Z_n)]|  + \sup_{\eta \ge 0} \Prob(\eta-\tau < \|Z_n\|^2 \le \eta + \tau)\\
    &\le \underbrace{\sup_{\eta \in \R}|\E [v_{\eta,\tau}(T_n) -v_{\eta,\tau}(L_n)]|}_{(I)} + \underbrace{\sup_{\eta \in \R}|\E[v_{\eta,\tau}(L_n)- v_{\eta,\tau}(Z_n)]|}_{(II)} \\ & \qquad + \underbrace{\sup_{\eta \ge 0} \Prob(\eta -\tau< \|Z_n\|^2 \le \eta + \tau)}_{(III)},
    \end{split}
\end{equation}
where the second step follows from the triangle inequality.

\medskip

\noindent\textbf{Step 2.} Next, we find upper bounds on $(I), (II)$, and $(III)$. Recall that $R_n = T_n - L_n$, then Lemma~\ref{lem:u_function_results}~(iv) combined with the Cauchy-Schwarz inequality yields,
\begin{equation} 
\begin{split}
(I) &\lesssim \tau^{-1} \left(\E\left[\|R_n\|^2 + \|R_n\|\|L_n\|\right]\right) \\
&\le \tau^{-1} \Big\{\Exp [\|R_n\|^2] + \Big(\E[\|R_n\|^2] \E[\|L_n\|^2]\Big)^{1/2} \Big\} \\
&=: \tau^{-1} \psi_3(L_n,R_n). \label{boundrem1}
\end{split}
\end{equation}
Next, by Lemma \ref{lem: stein}, it is straightforward to obtain
\begin{align}\label{linbound1}
  (II)
  \lesssim \tau^{-3} \psi_1(L_n) + \tau^{-2} \psi_2(L_n),  
\end{align}
where 
\begin{align*}
     \psi_1(L_n) &:= \frac{\|\Ga\|_\hs}{n^2} \sni \Exp[\|g^{(i_1)}(X_{i_1})\|^4] + \frac{\|\Ga\|_\opr^{3/2}}{n^{3/2}} \sni \Exp[\|g^{(i_1)}(X_{i_1})\|^3],  \\
    \psi_2(L_n) &:= \frac{\|\Ga\|_\opr^{1/2}}{n^{3/2}} \sni \Exp[\|g^{(i_1)}(X_{i_1})\|^3] + \frac{1}{n^{2}} \sni \Exp[\|g^{(i_1)}(X_{i_1})\|^4]. 
\end{align*}
Finally, by Lemma \ref{lem: AC}, one has 
\begin{equation}
(III) \lesssim \aleph (\Gamma_n) \tau, \label{antconc1}
\end{equation}
where $\aleph (\Gamma_n) = (\Lambda_1(\Gamma_n)\Lambda_2(\Gamma_n))^{-1/2}$. 
Therefore, combining \eqref{boundrem1}, \eqref{linbound1}, and \eqref{antconc1}, we obtain
\[
    d_{\mathrm{Kol}}(\|T_n\|,\|Z_n\|) \lesssim  \tau^{-3} \psi_1(L_n) + \tau^{-2} \psi_2(L_n) + \tau^{-1} \psi_3(L_n, R_n) + \aleph (\Gamma_n) \tau.
\]
Now, setting
\[
        \tau = \big(\psi_1(L_n)/\aleph(\Gamma_n)\big)^{1/4} +  \big(\psi_2(L_n)/\aleph(\Gamma_n)\big)^{1/3} + \big(\psi_3 (L_n, R_n)/\aleph(\Gamma_n)\big)^{1/2},
\]
one has
\begin{equation} \label{interim1}
    d_{\mathrm{Kol}}(\|T_n\|,\|Z_n\|) \lesssim \aleph (\Gamma_n)^{3/4} \psi_1(L_n)^{1/4} + \aleph (\Gamma_n)^{2/3}\psi_2(L_n)^{1/3} + \aleph (\Gamma_n)^{1/2}\psi_3(L_n,R_n)^{1/2}. 
\end{equation}
Finally, by Jensen's inequality we get $$\frac{1}{n} \sni \Exp[\|g^{(i_1)}(X_{i_1})\|^k] \leq m_n^{k/4},\;k=3,4,$$ and thus the final claim follows by substituting $\aleph(\Ga) = \|\Ga\|_\hs^{-1} \ell_n^{-1/4}$ and in \eqref{interim1}. 

\qed

The following auxiliary lemma helps in finding the bounds for Proposition~\ref{prop: HGA}.
\begin{lemma} \label{lem: boundlrHilb1}
    Under Assumption \ref{A1}, we have $\E\|L_n\|^2 = tr(\Ga) \lesssim 1$ and $\E\|R_n\|^2 \lesssim 1/n,$ where $L_n$ and $R_n$ are as defined in \eqref{lin} and \eqref{rem}
\end{lemma}

\textbf{Proof of Proposition~\ref{prop: HGA}.} The claim follows by combining \eqref{interim1} along with Assumptions \ref{A1}--\ref{A2} and Lemma \ref{lem: boundlrHilb1}. 

\qed

Before beginning the proof of Proposition \ref{prop: GA}, we state the following lemma.
 \begin{lemma} \label{lem: boundlrRd}
    Under Assumption \ref{B1}, we have for $l=2,3,4 $ and for any $1\leq i_1\leq n,$ $\Exp\Big(\|g^{(i_1)}(X_{i_1})\|^l\Big) \lesssim d^{l/2}.$ Furthermore, we have $\E\|L_n\|^2 = \tr(\Ga) \lesssim d$ and $\E\|R_n\|^2 \lesssim d/n,$ where $L_n$ and $R_n$ are as defined in \eqref{lin} and \eqref{rem}.
\end{lemma}

\textbf{Proof of Proposition \ref{prop: GA}.}
     The claim follows by combining \eqref{interim1} along with Assumptions \ref{B1}, \ref{B2} and Lemma \ref{lem: boundlrRd}.
    We skip the details for the interest of space.
\qed






\textbf{Proof of Theorem \ref{thm: BA}.}
We first present the proof for Jackknife multiplier bootstrap, i.e for \{$\star$ = {JMB}\}.
Recall, from \eqref{condJMB}, we have $ T_n^\text{JMB}\mid \bm X \sim N_{\h}(0,\Ghat)$ and $Z_n\sim N_{\h}(0,\Ga)$ as in Theorem \ref{thm: GA}. Then applying Lemma~\ref{lem:u_function_results} (ii) conditional on the data $\bm X$ and taking supremum over $a$, we get
\begin{equation} \label{BA1r}
    d_{\mathrm{Kol}}^B(\|T_n^\text{JMB}\|,\|Z_n\|) \lesssim \underbrace{\sup_{\eta \in \R}|\E^\star v_{\eta,\tau}(T_n^\text{JMB})- \E v_{\eta,\tau}(Z_n)]}_{(I)} + \underbrace{\sup_{\eta \ge 0} \Prob(\eta -\tau< \|Z_n\|^2 \le \eta + \tau)}_{(II)}.
\end{equation}
By Lemma~\ref{lem: GaussianComp}, we obtain,
\begin{align}\label{linbound}
  (I) \lesssim \tau^{-2} \psi_4(\Gtill,\Ga) + \tau^{-1} \psi_5(\Gtill,\Ga),  
\end{align}
where
\begin{align} 
        &\psi_4(\Gtill,\Ga) = \sqrt{(\trc(\Gtill) + \trc(\Ga))(\|\Gtill\|_\opr+\|\Ga\|_\opr)}\; \|\Gtill-\Ga\|_\hs, \label{psi4} \\
        &{\psi_5}(\Gtill,\Ga) = \Big|\tr\Big({\hat{\Gamma}^\text{JMB}}_n - \Ga\Big)\Big|, \label{psi5}
\end{align}
and by Lemma \ref{lem: AC}, one has 
\begin{equation}
(II) \lesssim \aleph (\Gamma_n) \tau, \label{antconc}
\end{equation}
where $\aleph (\Gamma_n) = (\Lambda_1(\Gamma_n)\Lambda_2(\Gamma_n))^{-1/2}$.
Therefore, combining \eqref{linbound}, and \eqref{antconc}, we obtain
\[
    d_{\mathrm{Kol}}^B(\|T_n^\text{JMB}\|,\|Z_n\|) \lesssim  \tau^{-2} \psi_4(\Gtill,\Ga) + \tau^{-1} \psi_5(\Gtill,\Ga) + \aleph (\Gamma_n) \tau.
\]
Now, setting
\[
    \tau = \big(\psi_4(\Gtill,\Ga)/\aleph(\Gamma_n)\big)^{1/3}+\big(\psi_5(\Gtill,\Ga)/\aleph(\Gamma_n)\big)^{1/2},
\]
one has
\begin{equation} \label{interim2}
    d_{\mathrm{Kol}}^B (\|T_n^\text{JMB}\|,\|Z_n\|) \lesssim \aleph (\Gamma_n)^{2/3} \psi_4(\Gtill,\Ga)^{1/3} + \aleph (\Gamma_n)^{1/2}\psi_5(\Gtill,\Ga)^{1/2}. 
\end{equation}

Thus, the final claim follows by taking expectation on both sides together with substituting $\aleph(\Ga) = \|\Ga\|_\hs^{-1} \ell_n^{-1/4}$ and applying H\"older's inequality. 




\bigskip

Next, we present the proof for Gaussian weighted bootstrap, i.e for \{$\star$ = {GWB}\}.
Applying Lemma~\ref{lem:u_function_results} (ii) conditional on the data $\bm X$ and taking supremum over $a$, we get
\begin{equation} \label{bwd}
    \begin{split}
    &\sup_{a \ge 0} |\Prob^\star(\|T_n^\text{GWB}\| \le a) -\Prob(\|Z_n\| \le a)|\\
    &\qquad \le \sup_{\eta \in \R}|\E^\star  v_{\eta,\tau}(T_n^\text{GWB}) - \E v_{\eta,\tau}(Z_n)|  + \sup_{\eta \ge 0} \Prob(\eta < \|Z_n\|^2 \le \eta + \tau)\\
    & \qquad \le \underbrace{\sup_{\eta \in \R}|\E^\star  v_{\eta,\tau}(T_n^\text{GWB}) - \E^\star  v_{\eta,\tau}(L_n^\text{GWB})|}_{(I)} + \underbrace{\sup_{\eta \in \R}|\E^\star  v_{\eta,\tau}(L_n^\text{GWB}) - \E v_{\eta,\tau}(Z_n)|}_{(II)} \\ & \hspace{6.5cm} + \underbrace{\sup_{\eta \ge 0} \Prob(\eta < \|Z_n\|^2 \le \eta + \tau)}_{(III)},
    \end{split}
\end{equation}
where the second step follows from the triangle inequality. 

Next, we bound $(I), (II)$ and $(III)$. Following similar calculations as in \eqref{boundrem1}, we have
\begin{equation} \label{boundrem2}
\begin{split}
    (I) &\lesssim \tau^{-1} \Big\{ \E^\star \|R_n^\text{GWB}\|^2 + \Big( \E^\star \|R_n^\text{GWB}\|^2 \E^\star \|L_n^\text{GWB}\|^2 \Big)^{1/2} \Big\} \\& =\vcentcolon \tau^{-1} \psi_8(L_n^\text{GWB},R_n^\text{GWB}).
\end{split}
\end{equation}
Since $L_n^\text{GWB}\mid \bm X \sim N_{\h}(0,\Ghat^\text{GWB}),$ using lemma~\ref{lem: GaussianComp}, we obtain,
\begin{align}\label{linbound2}
  (II) \lesssim \tau^{-2} \psi_6(\Ghat^\text{GWB},\Ga) + \tau^{-1} \psi_7(\Ghat^\text{GWB},\Ga),  
\end{align}
where
\begin{align} 
        & \psi_6(\Ghat^\text{GWB},\Ga) = \sqrt{(\trc(\Ghat^\text{GWB}) + \trc(\Ga))(\|\Ghat^\text{GWB}\|_\opr+\|\Ga\|_\opr)}\; \|\Ghat^\text{GWB}-\Ga\|_\hs, \\
        & \psi_7(\Ghat^\text{GWB},\Ga) = \Big|\tr\Big({\hat{\Gamma}^\text{GWB}}_n - \Ga\Big)\Big|, 
\end{align}
and by Lemma \ref{lem: AC}, one has 
\begin{equation}
(III) \lesssim \aleph (\Gamma_n) \tau, \label{antconc2}
\end{equation}
where $\aleph (\Gamma_n) = (\Lambda_1(\Gamma_n)\Lambda_2(\Gamma_n))^{-1/2}$. Therefore, combining \eqref{boundrem2}, \eqref{linbound2}, and \eqref{antconc2}, we obtain
\[
    d_{\mathrm{Kol}}^B(\|T_n^\text{GWB}\|,\|Z_n\|) \lesssim  \tau^{-2} \psi_6(\Ghat^\text{GWB},\Ga) + \tau^{-1} \psi_7(\Ghat^\text{GWB},\Ga) + \tau^{-1} \psi_8(L_n^\text{GWB},R_n^\text{GWB}) + \aleph (\Gamma_n) \tau.
\]
Now, setting
\[
    \tau = \big( \psi_6(\Ghat^\text{GWB},\Ga)/\aleph(\Gamma_n)\big)^{1/3} +\big( \psi_7(\Ghat^\text{GWB},\Ga)/\aleph(\Gamma_n)\big)^{1/2} +\big( \psi_8(L_n^\text{GWB},R_n^\text{GWB})/\aleph(\Gamma_n)\big)^{1/2},
\]
one has
\begin{multline*}
    d_{\mathrm{Kol}}^B (\|T_n^\text{GWB}\|,\|Z_n\|) \lesssim \aleph (\Gamma_n)^{2/3} \psi_6(\Ghat^\text{GWB},\Ga)^{1/3} + \aleph (\Gamma_n)^{1/2} \psi_7(\Ghat^\text{GWB},\Ga)^{1/2} \\+ \aleph (\Gamma_n)^{1/2} \psi_8(L_n^\text{GWB},R_n^\text{GWB})^{1/2}.
\end{multline*}
Again, the claim follows by taking expectation on both sides together with substituting $\aleph(\Ga) = \|\Ga\|_\hs^{-1} \ell_n^{-1/4}$ and applying H\"older's inequality. 

\bigskip

Finally, we present the proof for empirical bootstrap, i.e for \{$\star$ = {EB}\}.
Applying Lemma~\ref{lem:u_function_results} (ii) conditional on the data $\bm X$ and taking supremum over $a$, we get
\begin{equation} \label{ebwd}
    \begin{split}
    &\sup_{a \ge 0} |\Prob^\star(\|T_n^\text{EB}\| \le a) -\Prob(\|Z_n\| \le a)|\\
    &\qquad \le \sup_{\eta \in \R}|\E^\star  v_{\eta,\tau}(T_n^\text{EB}) - \E v_{\eta,\tau}(Z_n)|  + \sup_{\eta \ge 0} \Prob(\eta < \|Z_n\|^2 \le \eta + \tau)\\
    & \qquad \le \underbrace{\sup_{\eta \in \R}|\E^\star  v_{\eta,\tau}(T_n^\text{EB}) - \E^\star  v_{\eta,\tau}(L_n^\text{EB})|}_{(I)} + \underbrace{\sup_{\eta \in \R}|\E^\star  v_{\eta,\tau}(L_n^\text{EB}) - \E^\star  v_{\eta,\tau}(Z_n^\text{EB})|}_{(II)} \\ & \hspace{2.5cm} + \underbrace{\sup_{\eta \in \R}|\E^\star  v_{\eta,\tau}(Z_n^\text{EB}) - \E v_{\eta,\tau}(Z_n)|}_{(III)} + \underbrace{\sup_{\eta \ge 0} \Prob(\eta < \|Z_n\|^2 \le \eta + \tau)}_{(IV)},
    \end{split}
\end{equation}
where the second step follows from the triangle inequality. 

Next, we bound $(I), (II), (III)$ and $(IV)$. Following similar calculations as in \eqref{boundrem1}, we have
\begin{equation} \label{boundrem4}
\begin{split}
    (I) &\lesssim \tau^{-1} \Big\{ \E^\star \|R_n^\text{EB}\|^2 + \Big( \E^\star \|R_n^\text{EB}\|^2 \E^\star \|L_n^\text{EB}\|^2 \Big)^{1/2} \Big\} \\& =\vcentcolon \tau^{-1} \psi_{13}(L_n^\text{EB},R_n^\text{EB}).
\end{split}
\end{equation}
Next, using Lemma \ref{lem: stein}, conditional on the data $\bm X$, it is straightforward to obtain
\begin{align}\label{linbound4}
  (II)
  \lesssim \tau^{-3} \psi_9(L_n^\text{EB}) + \tau^{-2} \psi_{10}(L_n^\text{EB}),  
\end{align}
where 
\begin{align}
     & \psi_9(L_n^\text{EB}) = \frac{\|\Ghat^\text{EB}\|_\hs}{n^2} \sni \E^\star (\|g^{\bm X}(\xi_{i_1})\|^4) + \frac{\|\Ghat^\text{EB}\|_\opr^{3/2}}{n^{3/2}} \sni \E^\star (\|g^{\bm X}(\xi_{i_1})\|^3), \\
    & \psi_{10}(L_n^\text{EB}) = \frac{\|\Ghat^\text{EB}\|_\opr^{1/2}}{n^{3/2}} \sni \E^\star (\|g^{\bm X}(\xi_{i_1})\|^3) + \frac{1}{n^{2}} \sni \E^\star (\|g^{\bm X}(\xi_{i_1})\|^4).
\end{align}
Since $Z_n^\text{EB}\mid \bm X \sim N_{\h}(0,\Ghat^\text{EB}),$ using lemma~\ref{lem: GaussianComp}, we obtain,
\begin{align}\label{linbound5}
  (III) \lesssim \tau^{-2} \psi_{11}(\Ghat^\text{EB},\Ga) + \tau^{-1} \psi_{12}(\Ghat^\text{EB},\Ga),  
\end{align}
where
\begin{align} 
        & \psi_{11}(\Ghat^\text{EB},\Ga) = \sqrt{(\trc(\Ghat^\text{EB}) + \trc(\Ga))(\|\Ghat^\text{EB}\|_\opr+\|\Ga\|_\opr)}\; \|\Ghat^\text{EB}-\Ga\|_\hs, \\
        & \psi_{12}(\Ghat^\text{EB},\Ga) = \Big|\tr\Big({\hat{\Gamma}^\text{EB}}_n - \Ga\Big)\Big|,
\end{align}
and by Lemma \ref{lem: AC}, one has 
\begin{equation} 
(IV) \lesssim \aleph (\Gamma_n) \tau, \label{antconc4}
\end{equation}
where $\aleph (\Gamma_n) = (\Lambda_1(\Gamma_n)\Lambda_2(\Gamma_n))^{-1/2}$. Therefore, combining \eqref{boundrem4}, \eqref{linbound4}, \eqref{linbound5}, and \eqref{antconc4}, and setting
\begin{multline*}
    \tau = \big(\psi_9(L_n^\text{EB})/\aleph(\Gamma_n)\big)^{1/4} +  \big(\psi_{10}(L_n^\text{EB})/\aleph(\Gamma_n)\big)^{1/3} \\+ \big( \psi_{11}(\Ghat^\text{EB},\Ga)/\aleph(\Gamma_n)\big)^{1/3} +\big( \psi_{12}(\Ghat^\text{EB},\Ga)/\aleph(\Gamma_n)\big)^{1/2} \\+\big( \psi_{13}(L_n^\text{EB},R_n^\text{EB})/\aleph(\Gamma_n)\big)^{1/2},
\end{multline*}
one has
\begin{multline*}
    d_{\mathrm{Kol}}^B (\|T_n^\text{EB}\|,\|Z_n\|) \lesssim \aleph (\Gamma_n)^{3/4} \psi_9(L_n^\text{EB})^{1/4} + \aleph (\Gamma_n)^{2/3}\psi_{10}(L_n^\text{EB})^{1/3} + \aleph (\Gamma_n)^{2/3} \psi_{11}(\Ghat^\text{EB},\Ga)^{1/3} \\+ \aleph (\Gamma_n)^{1/2} \psi_{12}(\Ghat^\text{EB},\Ga)^{1/2} + \aleph (\Gamma_n)^{1/2} \psi_{13}(L_n^\text{EB},R_n^\text{EB})^{1/2}.
\end{multline*}
Note that, for $l=3,4,$ 
\begin{equation} \label{egx}
    \begin{split}
        \E^\star \|g^{\bm X}(\xi_1)\|^l = \frac{1}{n} \sni \Bigg\| \frac{1}{n^{r-1}} \sum_{1\leq i_2,\ldots,i_r\leq n} \widetilde h(\xir) - V_n\Bigg\|^l.
    \end{split}
\end{equation}
For $k=3,4$, Jensen's inequality gives
\begin{align}
\E\E^\star
\left[
\|g^{\bm X}(\xi_1)\|^l
\right]
&=
\frac{1}{n}
\sum_{i_1=1}^n
\E
\left\|
\frac{1}{n^{r-1}}
\sum_{1\leq i_2,\ldots,i_r\leq n}
\widetilde h(\xir)
-
V_n
\right\|^l
\nonumber\\
&\lesssim
\frac{1}{n^r}
\sum_{\iicr}
\E
\left[
\|\widetilde h(\xir)\|^l
\right]
\lesssim 1,
\label{eq:EB-projection-moment}
\end{align}
where the last inequality follows from Assumption~\ref{A1}.
Since, $\xi_1,\dots,\xi_n$ are i.i.d. random variables, we have 
\[\frac{1}{n} \sni \E(\|g^{\bm X}(\xi_{i_1})\|^l) \lesssim m_n^{l/4}.\]

Finally, taking the expectation on both sides and applying H\"older's inequality together with  $\aleph(\Ga) = \|\Ga\|_\hs^{-1} \ell_n^{-1/4}$ for the right-hand side quantity completes the proof of Theorem \ref{thm: BA} for EB.

\qed

The following auxiliary lemma helps in finding the bounds for Proposition~\ref{prop: HBA}.

\begin{lemma} \label{lem: boundlrHilb2}
The following hold. 
\begin{itemize}
    \item[(i)]  Under Assumption \ref{A1} and \ref{A3}, we have $\E\|L_n^\star\|^2 \lesssim 1$ and $\E\|R_n^\star\|^2 \lesssim 1/n$, for $\star\in\{\text{EB, GWB}\}$.
    \item[(ii)] Under Assumptions \ref{A1} and \ref{A3}, we have $(a) \;  \Exp(\trc(\hat\Gamma_n^\star)) \lesssim 1$, $(b)\; \Exp\|\hat\Gamma_n^\star-\Ga\|_\hs \lesssim n^{-1/2},$ $(c)\; \Exp\|\hat\Gamma_n^\star\|_\opr \lesssim 1,$ and $(d) \;\E|\tr(\hat\Gamma_n^\star - \Ga)| \lesssim  n^{-1/2}.$ 
\end{itemize}
\end{lemma}



\textbf{Proof of Proposition~\ref{prop: HBA}.} This proof follows from Theorem~\ref{thm: BA} by using assumptions \ref{A1}, \ref{A2}, and Lemma~\ref{lem: boundlrHilb2}.

\qed

\begin{lemma} \label{lem: boundlrRdboots}
Under assumptions \ref{B1}, \ref{B2} and \ref{A3}, we have $(a) \;  \Exp(\trc(\Ghat^\star)) \lesssim \max\{ d/\sqrt{n},$ $ d^{q_3}\}$, $(b)\; \Exp\|\Ghat^\star-\Ga\|_\hs \lesssim d/\sqrt{n},$  $(c)\; \Exp\|\Ghat^\star\|_\opr \lesssim \max\{d/\sqrt{n},d^{q_1}\}$ and $(d)\;\sm \Exp|{\hat{\Gamma}}^\star_{n,mm} - \Gamma_{n,mm}|\lesssim d/\sqrt{n}$ for $\star \in \{\text{EB, GWB, JMB}\}.$ Furthermore, under the Assumption~\ref{BN}, we get an improved rate $(e)\; \Exp\|\Ghat^\star\|_\opr \lesssim d^{q_1}.$
\end{lemma}

\textbf{Proof of Proposition \ref{prop: BA}.}
    The proof follows similarly to Proposition \ref{prop: HBA} using the assumptions \ref{B1}, \ref{B2} and Lemma \ref{lem: boundlrRdboots}.
    We skip the details for the interest of space.
\qed

\newpage
\section{Appendix: Proofs of results in Section~\ref{sec: stat app}}

\textbf{Proof of Proposition \ref{prop: power}.}
The first claim for empirical size follows immediately from Propositions
\ref{prop: GA} and \ref{prop: BA} and Proposition 3.2 of
\cite{koike2019mixed}.

For the second claim, we denote
\[
    \mu_n=\frac{\sqrt{n}}{r}\theta,
    \qquad
    T_n^{(0)}=T_n-\mu_n.
\]
Let
\[
    \Gamma_u
    =
    \Exp(T_n^{(0)}\otimes T_n^{(0)})
\]
be the covariance operator of the centered statistic, and let $Z_n$
be a centered Gaussian random element with covariance operator
$\Ga$, where $\Ga$ is the covariance operator of the linear
projection of $T_n^{(0)}$. Set
\[
    \tau_n
    =
    \{\trc(\Ga^2)\}^{1/2}
    =
    \|\Ga\|_\hs.
\]

Since $\Delta_n^\star\to0$ (and consequently $\Delta_n\to0$), there
exists a sequence of positive numbers $\varepsilon_n\to0$ such that
\[
    d_{\mathrm{Kol}}(\|T_n^{(0)}\|,\|Z_n\|)
    \leq\varepsilon_n,
    \qquad
    \Prob(\mathcal A_n^c)
    \leq\varepsilon_n,
\]
where
\[
    \mathcal A_n
    =
    \left\{
        d_{\mathrm{Kol}}^B
        (\|T_n^\star\|,\|Z_n\|)
        \leq\varepsilon_n
    \right\}.
\]
Furthermore, for all sufficiently large \(n\), \(1-\alpha+\varepsilon_n<1\), on $\mathcal A_n$,
\begin{align*}
    \Prob^\star\bigl(
        \|T_n^\star\|
        \leq q_n(1-\alpha+\varepsilon_n)
    \bigr)
    &\geq
    \Prob\bigl(
        \|Z_n\|
        \leq q_n(1-\alpha+\varepsilon_n)
    \bigr)
    -\varepsilon_n
    =1-\alpha,
\end{align*}
where for \(\alpha\in(0,1)\), $q_n(\alpha)$ is defined as the lower $\alpha$-th quantile of the distribution of $\|Z_n\|$. It follows from the definition of the bootstrap lower quantile that
\begin{equation}
    q_n^\star(1-\alpha)
    \leq
    q_n(1-\alpha+\varepsilon_n)
    \qquad\text{on }\mathcal A_n.
    \label{eq:bootstrap-quantile-power}
\end{equation}

We next use the Gaussian quadratic-form inequality
\begin{equation}
    \Prob\left(
        \left|
            \|Z_n\|^2-\trc(\Ga)
        \right|
        \geq t
    \right)
    \leq
    2\exp\left[
        -c\min\left\{
            \frac{t^2}{\tau_n^2},
            \frac{t}{\|\Ga\|_{\opr}}
        \right\}
    \right],
    \qquad t>0.
    \label{eq:GaussianQuadratic}
\end{equation}
Since $\|\Ga\|_{\opr}\leq\tau_n$, for every fixed $\alpha$
there exists a constant $C>0$ such that, for all sufficiently
large $n$,
\[
    \Prob\left(
        \|Z_n\|^2
        \leq
        \trc(\Ga)+C\tau_n
    \right)
    \geq 1 - \alpha+\varepsilon_n.
\]
Consequently,
\[
    q_n^2(1-\alpha+\varepsilon_n)
    \leq
    \trc(\Ga)+C\tau_n.
\]
Combining this with \eqref{eq:bootstrap-quantile-power}, we obtain
\begin{equation}
    \{q_n^\star(1-\alpha)\}^2
    \leq
    \trc(\Ga)+C\tau_n
    \qquad\text{on }\mathcal A_n.
    \label{eq:BootstrapQuantileSquared}
\end{equation}

Define
\[
    W_n
    =
    \|T_n^{(0)}\|^2-\trc(\Ga),
    \qquad
    V_n
    =
    2\langle\mu_n,T_n^{(0)}\rangle,
\]
and set
\[
    s_n
    =
    \|\mu_n\|^2-C\tau_n.
\]
Notice that $W_n$ is centered at $\trc(\Ga)$ rather than at
\[
    \Exp\|T_n^{(0)}\|^2
    =
    \trc(\Gamma_u).
\]
This is intentional, since $\Ga$ is the covariance operator of
the Gaussian comparator. We then have the exact decomposition
\begin{align*}
    \|T_n\|^2
    =
    \|T_n^{(0)}+\mu_n\|^2 
    =
    \trc(\Ga)
    +\|\mu_n\|^2
    +W_n+V_n.
\end{align*}
It follows from \eqref{eq:BootstrapQuantileSquared} that
\begin{align}
    \Prob_{H_a}\bigl(
        \|T_n\|\leq q_n^\star(1-\alpha)
    \bigr)
    &\leq
    \Prob_{H_a}(W_n+V_n\leq-s_n)
    +\Prob(\mathcal A_n^c)
    \notag\\
    &\leq
    \Prob_{H_a}(|W_n|\geq s_n/2)
    +
    \Prob_{H_a}(|V_n|\geq s_n/2)
    +\varepsilon_n.
    \label{eq:PowerDecomposition}
\end{align}

We first control $W_n$. By the Gaussian approximation result,
\begin{align*}
    \Prob_{H_a}(|W_n|\geq s_n/2)
    &\leq
    \Prob\left(
        \left|
            \|Z_n\|^2-\trc(\Ga)
        \right|
        \geq s_n/2
    \right)
    +2\varepsilon_n \\
    &\leq
    2\exp\left[
        -c\min\left\{
            \frac{s_n^2}{\tau_n^2},
            \frac{s_n}{\|\Ga\|_{\opr}}
        \right\}
    \right]
    +2\varepsilon_n,
\end{align*}
where the constant $c>0$ may change from line to line.

By the assumed signal-strength condition, for some sequence
$a_n\to\infty$,
\[
    \|\theta\|
    \geq
    C \frac{r}{\sqrt{n}}
    \{\trc(\Ga^2)\}^{1/4}a_n
    =
    C \frac{r}{\sqrt{n}}\tau_n^{1/2}a_n.
\]
Hence
\[
    \frac{\|\mu_n\|^2}{\tau_n}
    =
    \frac{n\|\theta\|^2}{r^2\tau_n}
    \to\infty,
\]
and therefore
\[
    \frac{s_n}{\tau_n}
    =
    \frac{\|\mu_n\|^2}{\tau_n}-C
    \to\infty.
\]
Since $\|\Ga\|_{\opr}\leq\tau_n$, it follows that
\begin{equation}
    \Prob_{H_a}(|W_n|\geq s_n/2)
    \to0.
    \label{eq:WControl}
\end{equation}

It remains to control $V_n$. By the Hoeffding decomposition and the
orthogonality of its components,
\[
    \Gamma_u
    =
    \Ga+\Gamma_R,
\]
where $\Gamma_R$ is the covariance operator of the degenerate
remainder. Note that as $\tr(\Gamma_{n,R}) = \E\|R_n\|^2,$ then it is straightforward to show that
\begin{equation}
    \|\Gamma_u\|_{\opr}
    \leq C\tau_n,
    \label{eq:CovarianceComparison}
\end{equation}
for all sufficiently large $n$.

Consequently,
\begin{align*}
    \Exp_{H_a}V_n^2
    =
    4\langle\mu_n,\Gamma_u\mu_n\rangle 
    \leq
    4\|\Gamma_u\|_{\opr}\|\mu_n\|^2 
    \leq
    C\tau_n\|\mu_n\|^2.
\end{align*}
Thus, by Chebyshev's inequality,
\begin{align*}
    \Prob_{H_a}(|V_n|\geq s_n/2)
    \leq
    C\frac{\tau_n\|\mu_n\|^2}{s_n^2} 
    =
    C\frac{\tau_n}{s_n}
    \frac{\|\mu_n\|^2}{s_n}
    \to0,
\end{align*}
because $s_n/\tau_n\to\infty$ and
$\|\mu_n\|^2/s_n\to1$.

Combining this with \eqref{eq:PowerDecomposition} and
\eqref{eq:WControl}, we obtain
\[
    \Prob_{H_a}\bigl(
        \|T_n\|\leq q_n^\star(1-\alpha)
    \bigr)
    \to0.
\]
The desired conclusion follows from
\[
    \Prob_{H_a}\bigl(
        \|T_n\|>q_n^\star(1-\alpha)
    \bigr)
    =
    1-
    \Prob_{H_a}\bigl(
        \|T_n\|\leq q_n^\star(1-\alpha)
    \bigr).
\]
\qed

\noindent\textbf{Proof of Proposition \ref{prop:kendall-minimax}.} Set
$
\rho_n
=
c_n\frac{d^{1/4}}{\sqrt n}.
$
Therefore, for \(d/n\to0\), we have
$
\rho_n^2
=
c_n^2\frac{\sqrt d}{n}
=
c_n^2\frac{d/n}{\sqrt d}
\to 0.
$ Let \(\varepsilon_1,\ldots,\varepsilon_p\) be independent Rademacher
random variables and define
$
v_\varepsilon
=
p^{-1/2}
(\varepsilon_1,\ldots,$ $\varepsilon_p)^\top
\in\mathbb R^p.
$
Then \(\|v_\varepsilon\|=1\). 

Choose \(a_n>0\), so that
$
\frac{2\sqrt d}{\pi}
\arcsin\left(\frac{a_n}{p+a_n}\right)
=
\rho_n
$ 
and define
$
s_n
=
\sin\left(\frac{\pi\rho_n}{2\sqrt d}\right).
$ Then, we have
$
a_n=\frac{ps_n}{1-s_n}.
$
Since \(\rho_n/\sqrt d = o(1)\), we have 
$
s_n
=
\frac{\pi\rho_n}{2\sqrt d}\{1+o(1)\}.
$

Moreover,
\[
\frac{p}{\sqrt d}
=
\left(\frac{2p}{p-1}\right)^{1/2}
\asymp1.
\]
Consequently, we have
\[
a_n
=
\frac{\pi p}{2\sqrt d}\rho_n\{1+o(1)\}
\asymp\rho_n.
\]

Next, for every realization of \(\varepsilon\), define the \(p\times p\) matrix
\[
R_\varepsilon
=
\frac{I_p+a_nv_\varepsilon v_\varepsilon^\top}
     {1+a_n/p}.
\]
Since \((v_\varepsilon)_j^2=1/p\), every diagonal entry of
\(R_\varepsilon\) equals one. Its eigenvalue in the direction
\(v_\varepsilon\) is 
$
\frac{1+a_n}{1+a_n/p},
$
whereas each of its remaining \(p-1\) eigenvalues are
$
\frac{1}{1+a_n/p}.
$
Thus, \(R_\varepsilon\) is a positive-definite correlation matrix.

Furthermore, note that
\[
\|R_\varepsilon-I_p\|_{\opr}
=
\frac{a_n(p-1)}{p+a_n}
\le a_n.
\]
Hence,
$
\|R_\varepsilon-I_p\|_{\opr}\le\eta
$ for all sufficiently large $n$.

For \(j\ne k\),
\[
(R_\varepsilon)_{jk}
=
\frac{a_n}{p+a_n}\varepsilon_j\varepsilon_k.
\]
The Gaussian identity for Kendall's tau therefore gives
\[
\tau_{jk}(R_\varepsilon)
=
\frac{2}{\pi}\arcsin\bigl((R_\varepsilon)_{jk}\bigr)
=
\frac{2}{\pi}\varepsilon_j\varepsilon_k
\arcsin\left(\frac{a_n}{p+a_n}\right).
\]
Also, we have
\[
\begin{aligned}
\|\tau(R_\varepsilon)\|
&=
\left\{
\sum_{1\le j<k\le p}
\tau_{jk}(R_\varepsilon)^2
\right\}^{1/2} \\
&=
\frac{2\sqrt d}{\pi}
\arcsin\left(\frac{a_n}{p+a_n}\right)
=
\rho_n.
\end{aligned}
\]
It follows that
$
P_{\varepsilon,n}
:=
N_p(0,R_\varepsilon)
\in
\cG_n^{\mathrm{loc}}(\eta,\rho_n)
$
for all sufficiently large $n$.

Next, set
\[
\sigma_n^2
=
\left(1+\frac{a_n}{p}\right)^{-1}
\qquad \text{and} \qquad
P_{0,n}
=
N_p(0,\sigma_n^2I_p).
\]
Because \(a_n\to0\), we have \(\sigma_n^2\to1\). Hence, by the
model-class assumption, we have
$
P_{0,n}\in\cP_{0,n}
$
for all sufficiently large \(n\). Indeed, \(P_{0,n}\) has independent
coordinates and therefore
\(\tau(P_{0,n})=0_d\). Therefore, under the common invertible transformation \(x\mapsto x/\sigma_n\), the
pair \((P_{0,n},P_{\varepsilon,n})\) becomes
$(
N_p(0,I_p),
N_p\bigl(0,I_p+a_nv_\varepsilon v_\varepsilon^\top\bigr)
)$
respectively. 

Next, let \(L_\varepsilon\) be the likelihood ratio of
\(P_{\varepsilon,n}\) with respect to \(P_{0,n}\), and define
$
\overline L_n
=
\E_\varepsilon L_\varepsilon,
$
where the expectation is with respect to the independent Rademacher
variables. Further, let \(\varepsilon'\) be an independent copy of \(\varepsilon\)
and set
$
q_{\varepsilon,\varepsilon'}
=
v_\varepsilon^\top v_{\varepsilon'}
=
\frac1p\sum_{j=1}^p\varepsilon_j\varepsilon_j'.
$

We first calculate the cross-moment of two likelihood ratios. For a unit
vector \(v\in\mathbb R^p\), write
$
Q_v=I_p+a_nvv^\top,
$
then
$
Q_v^{-1}
=
I_p-\frac{a_n}{1+a_n}vv^\top.
$

Consequently, the one-observation likelihood ratio with respect to
\(N_p(0,I_p)\) is
\[
\ell_v(x)
=
(1+a_n)^{-1/2}
\exp\left\{
\frac{a_n}{2(1+a_n)}(v^\top x)^2
\right\}.
\]
For unit vectors \(v,w\), put \(q=v^\top w\) and
$
b_n=\frac{a_n}{1+a_n}.
$
For all sufficiently large \(n\), \(a_n<1\), and the Gaussian
quadratic-form identity gives
\[
\E_0[\ell_v\ell_w]
=
(1+a_n)^{-1}
\det\left\{
I_p-b_n(vv^\top+ww^\top)
\right\}^{-1/2}.
\]
The rank-two determinant satisfies
\[
\begin{aligned}
\det\left\{
I_p-b_n(vv^\top+ww^\top)
\right\}
=
(1-b_n)^2-b_n^2q^2 
=
\frac{1-a_n^2q^2}{(1+a_n)^2}.
\end{aligned}
\]
Therefore,
\[
\E_0[\ell_v\ell_w]
=
(1-a_n^2q^2)^{-1/2},
\]
and it follows that
\[
\E_{0,n}\overline L_n^2
=
\E_{\varepsilon,\varepsilon'}
\left(
1-a_n^2q_{\varepsilon,\varepsilon'}^2
\right)^{-n/2}.
\]
For all sufficiently large \(n\), \(a_n<1/2\). Since
\[
-\log(1-u)\le\frac{u}{1-u},
\qquad 0\le u<1,
\]
we obtain
\[
\E_{0,n}\overline L_n^2
\le
\E\exp\left\{
\frac{na_n^2}
     {2(1-a_n^2)}
q_{\varepsilon,\varepsilon'}^2
\right\}.
\]

Let
$
S_p
=
\sum_{j=1}^p\varepsilon_j\varepsilon_j'
$
and define
$
\lambda_n
=
\frac{na_n^2}{2p(1-a_n^2)}.
$
Because \(q_{\varepsilon,\varepsilon'}=S_p/p\), the preceding bound becomes
\[
\E_{0,n}\overline L_n^2
\le
\E\exp\left(
\lambda_n\frac{S_p^2}{p}
\right).
\]

Let \(G\sim N(0,1)\) be independent of \(S_p\). For
\(0\le\lambda<1/2\), Gaussian linearization and the inequality
\(\cosh(x)\le e^{x^2/2}\) give
\[
\begin{aligned}
\E\exp\left(\lambda\frac{S_p^2}{p}\right)
&=
\E_G\E\exp\left(
\sqrt{\frac{2\lambda}{p}}\,G S_p
\right)\\
&=
\E_G
\prod_{j=1}^p
\cosh\left(
\sqrt{\frac{2\lambda}{p}}\,G
\right)\\
&\le
\E_G e^{\lambda G^2}
=
(1-2\lambda)^{-1/2}.
\end{aligned}
\]
Using \(a_n\asymp\rho_n\), we have
\[
\lambda_n
=
O\left(\frac{n\rho_n^2}{p}\right)
=
O\left(
c_n^2\frac{\sqrt d}{p}
\right)
=
O(c_n^2) = o(1),
\]
where we used \(\sqrt d/p\asymp1\). Hence, for all sufficiently large
\(n\), \(\lambda_n<1/2\), and
\[
\E_{0,n}\overline L_n^2
\le
(1-2\lambda_n)^{-1/2}.
\]
On the other hand,
\[
\E_{0,n}\overline L_n^2
\ge
\left(\E_{0,n}\overline L_n\right)^2
=
1.
\]
Therefore,
\[
\E_{0,n}\overline L_n^2\to1.
\]
Let
\[
\overline P_n
=
\E_\varepsilon P_{\varepsilon,n}.
\]
Since \(\overline L_n\) is the likelihood ratio of \(\overline P_n\)
with respect to \(P_{0,n}\),
\[
\chi^2\left(
\overline P_n,P_{0,n}
\right)
=
\E_{0,n}\overline L_n^2 = 1 + o(1).
\]
Therefore, from \cite{baraud01}, it follows that 
\[
\left\|
\overline P_n-P_{0,n}
\right\|_{\TV} = o(1).
\]
Let \(\phi_n\) satisfy the nominal level condition. Since
\(P_{0,n}\in\cP_{0,n}\), then we have
\[
\E_{P_{0,n}}\phi_n
\le
\alpha+o(1).
\]
Consequently, we have
\[
\begin{aligned}
\E_{\overline P_n}\phi_n
&\le
\E_{P_{0,n}}\phi_n
+
\left\|
\overline P_n-P_{0,n}
\right\|_{\TV}\\
&\le
\alpha+o(1).
\end{aligned}
\]
Moreover,
\[
\E_{\overline P_n}\phi_n
=
\E_\varepsilon
\E_{P_{\varepsilon,n}}\phi_n.
\]
Because every \(P_{\varepsilon,n}\) belongs to
\(\cG_n^{\mathrm{loc}}(\eta,\rho_n)\), we conclude that
\[
\inf_{P\in\cG_n^{\mathrm{loc}}(\eta,\rho_n)}
\E_{P}\phi_n
\le
\E_{\overline P_n}\phi_n
\le
\alpha+o(1).
\]
Substituting
$
\rho_n=c_n\frac{d^{1/4}}{\sqrt n}
$
proves the asserted lower bound. The final claim follows by taking any
separation sequence
$
r_n=o\left(\frac{d^{1/4}}{\sqrt n}\right)
$
and setting \(c_n=r_n\sqrt n/d^{1/4} = o(1)\).
\qed

\newpage
\section{Appendix: Proofs of auxiliary lemmas}

\textbf{Proof of Lemma~\ref{lem: boundlrHilb1}.}
\noindent
First note that, for $l = 2,3,4,$ and every $\bm i=(i_1,\ldots,i_r)\in
\mathcal I_n^r$,
\begin{equation*}
\begin{split}
\E\left\|
h_{\bm i}(\xir)
\right\|^l
=
\E\left\|
\widetilde h(\xir)-\theta_{\bm i}
\right\|^l \lesssim
\E\left\|
\widetilde h(\xir)
\right\|^l
+
\|\theta_{\bm i}\|^l \lesssim
\E\left\|
\widetilde h(\xir)
\right\|^l
\lesssim 1,
\end{split}
\end{equation*}
where the last inequality follows from Assumption~\textnormal{\ref{A1}}.

For every $\bm i\in\mathcal I_n^r$ and $c=1,\ldots,r$,
conditional Jensen's inequality gives
\begin{equation*}
\begin{split}
\E\left\|
\vartheta_{\bm i,c}(X_{i_c})
\right\|^l
&=
\E\left\|
\E\left[
h_{\bm i}(\xir)
\mid X_{i_c}
\right]
\right\|^l \leq
\E\left\|
h_{\bm i}(\xir)
\right\|^l
\lesssim 1.
\end{split}
\end{equation*}
Using the definition of $g^{(i_1)}(X_{i_1})$ and Jensen's
inequality once more, we obtain, for $l=2,3,4$,
\begin{equation*}
\begin{split}
\E\left\|
g^{(i_1)}(X_{i_1})
\right\|^l
&\leq
\frac{1}{P_{n-1,r-1}}
\sum_{(i_2,\ldots,i_r)\in
\mathcal I_{n,-i_1}^{r-1}}
\E\left\|
\vartheta_{\bm i,1}(X_{i_1})
\right\|^l \lesssim 1.
\end{split}
\end{equation*}

Since $g^{(i_1)}(X_{i_1})$, $i_1=1,\ldots,n$, are independent
and mean zero, it follows that
\begin{equation*}
\E\|L_n\|^2
=
\frac{1}{n}
\sum_{i_1=1}^n
\E\left\|
g^{(i_1)}(X_{i_1})
\right\|^2
=
\tr(\Gamma_n)
\lesssim 1.
\end{equation*}

It remains to bound the remainder. For
$\bm i=(i_1,\ldots,i_r)\in\mathcal I_n^r$, write
\begin{equation*}
F_{\bm i}
:=
f^{\bm i}(X_{i_{[1:r]}})
=
h_{\bm i}(\xir)
-
\sum_{c=1}^r
\vartheta_{\bm i,c}(X_{i_c}).
\end{equation*}
By the preceding moment bounds,
\begin{equation*}
\begin{split}
\E\|F_{\bm i}\|^2
&\leq
(r+1)
\left\{
\E\left\|
h_{\bm i}(\xir)
\right\|^2
+
\sum_{c=1}^r
\E\left\|
\vartheta_{\bm i,c}(X_{i_c})
\right\|^2
\right\} \lesssim 1,
\end{split}
\end{equation*}
uniformly over $\bm i\in\mathcal I_n^r$.

We next consider the covariance between two remainder kernels.
If $\bm i,\bm j\in\mathcal I_n^r$ have disjoint index sets, then
$F_{\bm i}$ and $F_{\bm j}$ are independent and mean zero, and
hence
\begin{equation*}
\E\langle F_{\bm i},F_{\bm j}\rangle=0.
\end{equation*}
Suppose instead that the two tuples have exactly one common
index, say $i_a=j_b=c$. Conditional on $X_c$, the remaining
observations entering $F_{\bm i}$ and $F_{\bm j}$ are independent.
Moreover, the first-order degeneracy of the remainder kernel gives
\begin{equation*}
\E[F_{\bm i}\mid X_c]=0,
\qquad
\E[F_{\bm j}\mid X_c]=0.
\end{equation*}
It follows that
\begin{equation*}
\begin{split}
\E\langle F_{\bm i},F_{\bm j}\rangle
=
\E\left[
\E\left[
\langle F_{\bm i},F_{\bm j}\rangle
\mid X_c
\right]
\right]
=
\E\left[
\left\langle
\E[F_{\bm i}\mid X_c],
\E[F_{\bm j}\mid X_c]
\right\rangle
\right]
=0.
\end{split}
\end{equation*}
Thus, only pairs of tuples having at least two common indices
contribute to $\E\|R_n\|^2$.

For tuples having $c\geq2$ common indices, the Cauchy--Schwarz
inequality gives
\begin{equation*}
\left|
\E\langle F_{\bm i},F_{\bm j}\rangle
\right|
\leq
\left(\E\|F_{\bm i}\|^2\right)^{1/2}
\left(\E\|F_{\bm j}\|^2\right)^{1/2}
\lesssim 1.
\end{equation*}
The number of ordered pairs
$(\bm i,\bm j)\in\mathcal I_n^r\times\mathcal I_n^r$
having exactly $c$ common indices is bounded by
$C_r n^{2r-c}$. Therefore,
\begin{equation*}
\begin{split}
\E\|R_n\|^2
=
\frac{n}{r^2P_{n,r}^2}
\sum_{\iinr}
\sum_{\jjnr}
\E\langle F_{\bm i},F_{\bm j}\rangle \lesssim
\frac{n}{P_{n,r}^2}
\sum_{c=2}^r n^{2r-c} \lesssim
\sum_{c=2}^r n^{1-c}
\lesssim
\frac{1}{n},
\end{split}
\end{equation*}
where we used that $r$ is fixed and $P_{n,r}\asymp n^r$.
This proves the claim.
\qed

\textbf{Proof of Lemma~\ref{lem: boundlrRd}.} 
    By Assumption~\textnormal{\ref{B2}}, Jensen's inequality, and the
definition
$
h_{\bm i,m}(\xir)
=
\widetilde h_m(\xir)-\theta_{\bm i,m},
$
we have, for $l=2,3,4$,
\begin{equation*}
\max_{1\leq m\leq d}
\max_{\iinr}
\E\left|
h_{\bm i,m}(\xir)
\right|^l
\lesssim 1.
\end{equation*}
Therefore,
\begin{equation*}
\begin{split}
\max_{\iinr}
\E\left\|
h_{\bm i}(\xir)
\right\|^l
&\leq
d^{l/2-1}
\max_{\iinr}
\sum_{m=1}^d
\E\left|
h_{\bm i,m}(\xir)
\right|^l\\
&\lesssim d^{l/2}.
\end{split}
\end{equation*}
Conditional Jensen's inequality and the definition of
$g^{(i_1)}(X_{i_1})$ then give
\begin{equation*}
\max_{1\leq i_1\leq n}
\E\left\|
g^{(i_1)}(X_{i_1})
\right\|^l
\lesssim d^{l/2},
\qquad l=2,3,4.
\end{equation*}
Similarly, for
$
F_{\bm i}
=
f^{\bm i}(X_{i_{[1:r]}}),
$
we have
\begin{equation*}
\max_{\iinr}
\E\|F_{\bm i}\|^2
\lesssim d.
\end{equation*}
    Then the rest of the proof follows directly from the proof of Lemma~\ref{lem: boundlrHilb1} in $\R^d$.

\qed

\textbf{Proof of Lemma~\ref{lem: boundlrHilb2}.}
We first prove part $(i)$ for $\text{GWB}$. Note that for any $k=2,\dots,r,$ we can write
\begin{equation*}
\begin{split}
    &\E^\star \|T_n^{(k)}\|^2\\
    &= \frac{ {\binom{r}{k}}^2 }{r^2n {P^2_{n-1,r-1}}} \sum_{(i_1,\dots, i_k) \in \mathcal{I}_n^k} \sum_{(j_1,\dots, j_k) \in \mathcal{I}_n^k} \E^\star \bigg(\prod_{l=1}^k (w_{i_l}-1)(w_{j_l}-1)\bigg) \\&\hspace{3cm}
       \sum_{(i_{k+1},\dots,i_r) \in \mathcal{I}_{n,-\{i_1,i_2,\dots,i_k\}}^{r}} \sum_{(j_{k+1},\dots,j_r) \in \mathcal{I}_{n,-\{j_1,j_2,\dots,j_k\}}^{r}} \la \widetilde h(\xir) , \widetilde h(\xjr) \ra\\
       &= \frac{ {\binom{r}{k}}^2 }{r^2n {P^2_{n-1,r-1}}} \sum_{(i_1,\dots, i_k) \in \mathcal{I}_n^k} \E^\star \bigg(\prod_{l=1}^k (w_{i_l}-1)^2\bigg) \\&\hspace{2cm}
       \sum_{(i_{k+1},\dots,i_r) \in \mathcal{I}_{n,-\{i_1,i_2,\dots,i_k\}}^{r}} \sum_{(j_{k+1},\dots,j_r) \in \mathcal{I}_{n,-\{i_1,i_2,\dots,i_k\}}^{r}} \la \widetilde h(\xir) , \widetilde h(X_{i_{[1:k]}}, X_{j_{[(k+1):r]}}) \ra\\
       &= \frac{ {\binom{r}{k}}^2 }{r^2n {P^2_{n-1,r-1}}} \sum_{\iinr} \sum_{(j_{k+1},\dots,j_r) \in \mathcal{I}_{n,-\{i_1,i_2,\dots,i_k\}}^{r}} \la \widetilde h(\xir) , \widetilde h(X_{i_{[1:k]}}, X_{j_{[(k+1):r]}}) \ra\\
\end{split}
\end{equation*}
where the second equality follows from as $w_1,\dots,w_n$ are independently distributed. Consequently, using Assumption \ref{A1}, we get
\begin{equation*}
    \Exp\|R_n^\text{GWB}\|^2 = \sum_{k=2}^r \Exp\|T_n^{(k)}\|^2 \lesssim \sum_{k=2}^r 1/n^{k-1}  \lesssim 1/n
\end{equation*}
Similarly,
\begin{equation*}
\begin{split}
    & \Exp \|L_n^\text{GWB}\|^2\\
    & = \Exp \Bigg\{\frac{n}{\npr^2} \sni \sum_{\iinri}\sum_{\jjnri} \Big\la \widetilde h(\xiir) - U_n, \widetilde h(\xijr) - U_n \Big\ra \E^\star (w_{i_1}-1)^2 \Bigg\} \\
    & = \frac{n}{\npr^2} \sni \sum_{\iinri}\sum_{\jjnri} \Exp \Big\la \widetilde h(\xiir) - U_n, \widetilde h(\xijr) - U_n \Big\ra\\
    & \lesssim\;\underset{\iinr}{\max} \Exp\|\widetilde h(\xir)-U_n\|^2\\
    & \lesssim\; \underset{\iinr}{\max} \Exp\|\widetilde h(\xir)\|^2 + \E\|U_n\|^2 \lesssim 1,
\end{split}
\end{equation*}
where the first inequality follows from the Cauchy-Schwarz inequality, the second inequality follows Jensen's inequality from tri and the final inequality follows from Assumption \ref{A1}.

\bigskip
Next, to prove part (i) for EB, note that for $l=2,3,4,$ using Assumption \ref{A1}, we have
\begin{equation*}
    \Exp\|g^{\bm X}(\xi_1)\|^l \lesssim 1,
\end{equation*}
following \eqref{egx}.

As $\xi_1,\dots,\xi_n$ are independent and identically distributed and $\E^\star (f^{\bm X}(\xixr)) = 0$ almost surely, it is easy to verify that for $c=0,1,$
\begin{equation*}
    \E^\star \Big\la f^{\bm X}(\ixrr),f^{\bm X}(\ixrrr) \Big\ra = 0,
\end{equation*}
and for $c\geq 2,$
\begin{equation*}
\begin{split}
    &\E^\star \Big\la f^{\bm X}(\ixrr),f^{\bm X}(\ixrrr) \Big\ra \\&= \frac{1}{n^{2r-c}}\;\sum_{\iicrc} \Big\la f^{(i_{[1:c]},i_{[(c+1):r]})}(\xcri), f^{(i_{[1:c]},i_{[(r+1):(2r-c)]})}(\xcrii) \Big\ra.
\end{split}
\end{equation*}
Therefore, for any $m=1,\dots,d,$ almost surely
\begin{equation*}
\begin{split}
    &\E^\star \|R_n^\text{EB}\|^2\\
    &= \frac{n}{r^2\npr^2} \sum_{\iinr} \sum_{\jjnr} \E^\star \Big\la f^{\bm X}(\ixr), f^{\bm X}(\ixjr) \Big\ra\\
    &= \frac{n}{r^2\npr^2} \sum_{c=0}^r \frac{\cpr^2}{c!}\sum_{\icnri} \E^\star \Big\la f^{\bm X}(\ixr,\ixcor), f^{\bm X}(\ixr,\ixccr) \Big\ra\\
    &= \frac{n}{r^2\npr^2} \sum_{c=2}^r \frac{\cpr^2}{c!}\sum_{\icnri} \E^\star \Big\la f^{\bm X}(\ixr,\ixcor), f^{\bm X}(\ixr,\ixccr) \Big\ra\\
    &= \frac{n}{r^2 \npr^2}\; \sum_{c=2}^r \frac{\cpr^2\; \nprc}{c!\;n^{2r-c}} \sum_{\iicrc}\Big\la f^{(i_{[1:c]},i_{[(c+1):r]})}(\xcri), \\ &\hspace{8cm} f^{(i_{[1:c]},i_{[(r+1):(2r-c)]})}(\xcrii) \Big\ra,
\end{split}
\end{equation*}
and using Assumption \ref{A1}, it is straightforward to show that
\begin{equation*}
\begin{split}
    \Exp\|R_{n}^\text{EB}\|^2 &\lesssim  \frac{1}{ n^{2r-1}}\; \sum_{c=2}^r \sum_{\iicrc} \Exp\Big\la f^{(i_{[1:c]},i_{[(c+1):r]})}(\xcri), \\ &\hspace{6cm} f^{(i_{[1:c]},i_{[(r+1):(2r-c)]})}(\xcrii) \Big\ra\\& \lesssim 1/n.
\end{split}
\end{equation*}
Similarly, we get, almost surely
\begin{equation*}
\begin{split}
    \E^\star \|L_n^\text{EB}\|^2 &= \E^\star \|g^{\bm X}(\xi_1)\|^2 \lesssim 1,
\end{split}
\end{equation*}
and from Assumption \ref{A1}, $\Exp\|L_n^\text{EB}\|^2 \lesssim 1.$

 \bigskip
 
Next we prove part (ii) of the lemma. We first consider the jackknife multiplier bootstrap. Subtracting the deterministic constant \(\theta\) from the observable kernel does not change any of the bootstrap coefficients. We may therefore work with the globally centered kernel
\[ H(x_1,\ldots,x_r):=\widetilde h(x_1,\ldots,x_r)-\theta \]
and with
\[ U_n^\circ:=U_n-\theta =\frac{1}{P_{n,r}} \sum_{\iinr} H(\xir). \]
Assumption \ref{A1} and Jensen's inequality imply
\begin{equation} \label{8.1}
    \max_{\boldsymbol i\in C_n^r} E\|H(\xir)\|^4\lesssim 1.    
\end{equation}

For $i_1=1,\ldots,n$, define the leave-one-in average
\[
\overline g^{(i_1)}(X_{i_1})
:=
\frac{1}{P_{n-1,r-1}}
\sum_{\jjnri}
H(\xijr),
\]
and its conditionally centered remainder
\begin{equation}
    \eta_{i_1}(X_{i_1})
:=
\overline g^{(i_1)}(X_{i_1})
-E\!\left[\overline g^{(i_1)}(X_{i_1})\mid X_{i_1}\right].
\label{8.2}
\end{equation}

Therefore
\[
\begin{aligned}
E\!\left[\overline g^{(i_1)}(X_{i_1})\mid X_{i_1}\right]
=
\frac{1}{P_{n-1,r-1}}
\sum_{\jjnri}
\left\{
E[\widetilde h(\xijr)\mid X_{i_1}]-\theta
\right\} =
g^{(i_1)}(X_{i_1})+\overline\theta_{i_1}-\theta.
\end{aligned}
\]
Consequently, Condition \ref{A3} gives
\begin{equation}
    E\!\left[\overline g^{(i_1)}(X_{i_1})\mid X_{i_1}\right]
=g^{(i_1)}(X_{i_1}),
\qquad
\overline g^{(i_1)}(X_{i_1})=g^{(i_1)}(X_{i_1})+\eta_{i_1}(X_{i_1}).
\label{8.3}
\end{equation}
This is the only point at which the exact balance condition \ref{A3} is needed.

We first record the moment bounds used below. Conditional Jensen's inequality and \eqref{8.1} give
\begin{equation}
    \max_{1\le i\le n}E\|g^{(i_1)}(X_{i_1})\|^4\lesssim1.
\label{8.4}
\end{equation}

Moreover,
\begin{equation}
    \max_{1\le i\le n}E\|\eta_{i_1}(X_{i_1})\|^2\lesssim n^{-1}.
\label{8.5}
\end{equation}

To verify \eqref{8.5}, define for $\jjnri$,
\[
\zeta_{i_1,j_2,\dots,j_r}
:=
H(\xijr)
-E[H(\xijr)\mid X_{i_1}].
\]
Then
\begin{equation}
    E\|\eta_{i_1}(X_{i_1})\|^2
=
\frac{1}{P_{n-1,r-1}^2}
\sum_{\jjnri} \sum_{\kknri}
E\langle\zeta_{i_1,j_2,\dots,j_r},\zeta_{i_1,k_2,\dots,k_r}\rangle.
\label{8.6}
\end{equation}

If the partner-index sets of $\{j_2,\dots,j_r\}$ and $\{k_2,\dots,k_r\}$ are disjoint, then, conditionally on \(X_{i_1}\), the two random elements in \eqref{8.6} are independent and centered. Their contribution is therefore zero. The number of pairs whose partner-index sets have a nonempty intersection is \(O(n^{2r-3})\). By Cauchy–Schwarz and \eqref{8.1}, every remaining summand is bounded in absolute value by a constant. Since \(P_{n-1,r-1}^2\asymp n^{2r-2}\), \eqref{8.5} follows.

The identity \(P_{n,r}=nP_{n-1,r-1}\) gives
\begin{equation}
    \frac1n \sni \overline g^{(i_1)}(X_{i_1})=U_n^\circ.
\label{8.7}
\end{equation}

Hence the jackknife coefficient satisfies
\[
\widehat g^{(i_1)}(X_{i_1})
=
\overline g^{(i_1)}(X_{i_1})-U_n^\circ
=
g^{(i_1)}(X_{i_1})+\eta_{i_1}(X_{i_1})-U_n^\circ.
\label{8.8}
\]
It follows from \eqref{8.7} that
\begin{equation}
    \widehat\Gamma_n^{\mathrm{JMB}}
=
\frac1n\sni
\overline g^{(i_1)}(X_{i_1})\otimes\overline g^{(i_1)}(X_{i_1})
-U_n^\circ\otimes U_n^\circ.
\label{8.9}
\end{equation}

Define
\[
\begin{aligned}
A_n
&:=
\frac1n\sni
\{g^{(i_1)}(X_{i_1})\otimes g^{(i_1)}(X_{i_1})-E(g^{(i_1)}(X_{i_1})\otimes g^{(i_1)}(X_{i_1}))\},\\
B_n
&:=
\frac1n\sni g^{(i_1)}(X_{i_1})\otimes\eta_{i_1}(X_{i_1}),\\
C_n
&:=
\frac1n\sni\eta_{i_1}(X_{i_1})\otimes\eta_{i_1}(X_{i_1}).
\end{aligned}
\]
Since
\[
\Gamma_n=\frac1n\sni E(g^{(i_1)}(X_{i_1})\otimes g^{(i_1)}(X_{i_1})),
\]
equations \eqref{8.3} and \eqref{8.9} yield
\begin{equation}
     \widehat\Gamma_n^{\mathrm{JMB}}-\Gamma_n
=
A_n+B_n+B_n^*+C_n-U_n^\circ\otimes U_n^\circ.
\label{8.10}
\end{equation}
The variables
\[
g^{(i_1)}(X_{i_1})\otimes g^{(i_1)}(X_{i_1})-E(g^{(i_1)}(X_{i_1})\otimes g^{(i_1)}(X_{i_1})),
\qquad i=1,\ldots,n,
\]
are independent, centered random elements of the Hilbert space of Hilbert–Schmidt operators. Therefore, using
\(\|x\otimes x\|_\hs=\|x\|^2\) and \eqref{8.4},
\[
\begin{aligned}
E\|A_n\|_\hs^2
&=
\frac1{n^2}\sni
E\|g^{(i_1)}(X_{i_1})\otimes g^{(i_1)}(X_{i_1})-E(g^{(i_1)}(X_{i_1})\otimes g^{(i_1)}(X_{i_1}))\|_\hs^2\\
&\lesssim
\frac1{n^2}\sni E\|g^{(i_1)}(X_{i_1})\|^4
\lesssim n^{-1}.
\end{aligned}
\]
Thus,
\begin{equation}
    E\|A_n\|_\hs\lesssim n^{-1/2}.
\label{8.11}
\end{equation}
By Cauchy–Schwarz, \eqref{8.4}, and \eqref{8.5},
\begin{equation}
    \begin{aligned}
E\|B_n\|_\hs
&\le
\frac1n\sni E\{\|g^{(i_1)}(X_{i_1})\|\,\|\eta_{i_1}(X_{i_1})\|\}\\
&\le
\frac1n\sni
\{E\|g^{(i_1)}(X_{i_1})\|^2E\|\eta_{i_1}(X_{i_1})\|^2\}^{1/2}
\lesssim n^{-1/2},
\end{aligned}
\label{8.12}
\end{equation}
and the same bound holds for \(B_n^*\). Also,
\begin{equation}
    E\|C_n\|_\hs
\le
\frac1n\sni E\|\eta_{i_1}(X_{i_1})\|^2
\lesssim n^{-1}.
\label{8.13}
\end{equation}

Finally, Lemma~\ref{lem: boundlrHilb1} and the decomposition
\[
\frac{\sqrt n}{r}U_n^\circ=L_n+R_n,
\]
imply
\begin{equation}
    E\|U_n^\circ\|^2\lesssim n^{-1}.
\label{8.14}
\end{equation}

Combining \eqref{8.10}, \eqref{8.11}, \eqref{8.12}, \eqref{8.13}, and \eqref{8.14} gives
\begin{equation}
    E\|\widehat\Gamma_n^{\mathrm{JMB}}-\Gamma_n\|_\hs
\lesssim n^{-1/2}.
\label{8.15}
\end{equation}

This proves part (b) for the JMB.

We next establish the trace bound. Taking traces in \eqref{8.10} gives
\begin{equation}
    \begin{aligned}
\operatorname{tr}
(\widehat\Gamma_n^{\mathrm{JMB}}-\Gamma_n)
={}&
\frac1n\sni
\{\|g^{(i_1)}(X_{i_1})\|^2-E\|g^{(i_1)}(X_{i_1})\|^2\}\\
&+\frac2n\sni\langle g^{(i_1)}(X_{i_1}),\eta_{i_1}(X_{i_1})\rangle
+\frac1n\sni\|\eta_{i_1}(X_{i_1})\|^2
-\|U_n^\circ\|^2.
\end{aligned}
\label{8.16}
\end{equation}

Independence of the \(g^{(i_1)}(X_{i_1})\)'s and \eqref{8.4} imply
\begin{equation}
    E\left|
\frac1n\sni
\{\|g^{(i_1)}(X_{i_1})\|^2-E\|g^{(i_1)}(X_{i_1})\|^2\}
\right|
\le
\left\{
\frac1{n^2}\sni E\|g^{(i_1)}(X_{i_1})\|^4
\right\}^{1/2}
\lesssim n^{-1/2}.
\label{8.17}
\end{equation}

Furthermore, by \eqref{8.4}, and \eqref{8.5},
\begin{equation}
    E\left|
\frac1n\sni\langle g^{(i_1)}(X_{i_1}),\eta_{i_1}(X_{i_1})\rangle
\right|
\le
\frac1n\sni
\{E\|g^{(i_1)}(X_{i_1})\|^2E\|\eta_{i_1}(X_{i_1})\|^2\}^{1/2}
\lesssim n^{-1/2},
\label{8.18}
\end{equation}

whereas
\begin{equation}
    \frac1n\sni E\|\eta_{i_1}(X_{i_1})\|^2\lesssim n^{-1},
\qquad
E\|U_n^\circ\|^2\lesssim n^{-1}.
\label{8.19}
\end{equation}

Equations \eqref{8.16}, \eqref{8.17}, \eqref{8.18}, and \eqref{8.19} prove
\begin{equation}
    E\left|
\operatorname{tr}
(\widehat\Gamma_n^{\mathrm{JMB}}-\Gamma_n)
\right|
\lesssim n^{-1/2}.
\label{8.20}
\end{equation}

This proves part (d).

By \eqref{8.4},
\[
\operatorname{tr}(\Gamma_n)
=
\frac1n\sni E\|g^{(i_1)}(X_{i_1})\|^2
\lesssim1.
\]
Consequently, \eqref{8.20} gives
\[
E\operatorname{tr}(\widehat\Gamma_n^{\mathrm{JMB}})
\le
\operatorname{tr}(\Gamma_n)
+
E\left|
\operatorname{tr}
(\widehat\Gamma_n^{\mathrm{JMB}}-\Gamma_n)
\right|
\lesssim1,
\label{8.21}
\]
which proves part (a). Similarly, since
\(\|T\|_{\mathrm{op}}\le\|T\|_\hs\),
\begin{equation}
    \begin{aligned}
E\|\widehat\Gamma_n^{\mathrm{JMB}}\|_{\mathrm{op}}
&\le
\|\Gamma_n\|_{\mathrm{op}}
+
E\|\widehat\Gamma_n^{\mathrm{JMB}}-\Gamma_n\|_{\mathrm{op}}\\
&\le
\operatorname{tr}(\Gamma_n)
+
E\|\widehat\Gamma_n^{\mathrm{JMB}}-\Gamma_n\|_\hs
\lesssim1.
\end{aligned}
\label{8.22}
\end{equation}

This proves part (c).

By construction,
\[
\widehat\Gamma_n^{\mathrm{GWB}}
=
\widehat\Gamma_n^{\mathrm{JMB}}.
\]
Thus, all four conclusions also hold for the Gaussian weighted bootstrap.

It remains to prove the results for the empirical bootstrap. Because the covariance estimators are invariant under deterministic shifts of the kernel, we continue to use \(H=\widetilde h-\theta\). Write
\[
\widehat\Gamma_n^{\mathrm{EB}}
=
\widehat\Gamma_{n,1}^{\mathrm{EB}}
-
\widehat\Gamma_{n,2}^{\mathrm{EB}},
\qquad
\widehat\Gamma_n^{\mathrm{JMB}}
=
\widehat\Gamma_{n,1}^{\mathrm{JMB}}
-
\widehat\Gamma_{n,2}^{\mathrm{JMB}},
\]
where
\[
\widehat\Gamma_{n,1}^{\mathrm{EB}}
=
\frac1{n^{2r-1}}
\sni
\sum_{\jjnci} \sum_{\kknci}
H(\xijr)
\otimes
H(X_{i_1},X_{k_{[2:r]}}),
\]
\[
\widehat\Gamma_{n,2}^{\mathrm{EB}}
=
\frac1{n^{2r}}
\sum_{\iicr} \sum_{\jjcr}
H(\xir)\otimes H(\xjr),
\]
and
\[
\widehat\Gamma_{n,1}^{\mathrm{JMB}}
=
\frac1{nP_{n-1,r-1}^2}
\sni
\sum_{\jjnri} \sum_{\kknri}
H(\xijr)
\otimes
H(X_{i_1},X_{k_{[2:r]}}),
\]
\[
\widehat\Gamma_{n,2}^{\mathrm{JMB}}
=
\frac1{P_{n,r}^2}
\sum_{\iinr} \sum_{\jjnr}
H(\xir)\otimes H(\xjr).
\]
Here $\mathcal{C}_{n, -i}^r$ is defined similarly to as $\mathcal{I}_{n, -i}^r$ by removing the fixed index $i$ from $\mathcal{C}_{n}^r.$ 
Set
\[
c_{n,r}:=\frac{P_{n,r}^2}{n^{2r}}.
\]
For fixed \(r\),
\begin{equation}
    |c_{n,r}-1|\lesssim n^{-1}.
\label{8.23}
\end{equation}

Let \(\mathcal B_{1,n}\) be the collection of indices
\[
(i_1,j_1,\dots,j_r,k_1,\dots,k_r)
\in
\mathcal{C}_n^{2r-1}
\]
for which either \((i_1,j_1,\dots,j_r)\notin\mathcal I_n^r\) or
\((i,k_1,\dots,k_r)\notin\mathcal I_n^r\). Similarly, let
\[
\mathcal B_{2,n}
:=
(\mathcal C_n^r\times \mathcal C_n^r)\setminus(\mathcal I_n^r\times\mathcal I_n^r).
\]
No restriction is imposed on intersections between the two different \(r\)-tuples. Since failure of pairwise distinctness imposes at least one index equality,
\begin{equation}
    |\mathcal B_{1,n}|\lesssim n^{2r-2},
\qquad
|\mathcal B_{2,n}|\lesssim n^{2r-1}.
\label{8.24}
\end{equation}

It follows that
\begin{equation}
    \widehat\Gamma_n^{\mathrm{EB}}
=
c_{n,r}\widehat\Gamma_n^{\mathrm{JMB}}
+D_{1n}-D_{2n},
\label{8.25}
\end{equation}
where
\[
D_{1n}
=
\frac1{n^{2r-1}}
\sum_{(i_1,j_1,\dots,j_r,k_1,\dots,k_r)\in\mathcal B_{1,n}}
H(\xijr)\otimes H(X_{i_1},X_{k_{[2:r]}})
\]
and
\[
D_{2n}
=
\frac1{n^{2r}}
\sum_{(i_1,\dots,i_r,j_1,\dots,j_r)\in\mathcal B_{2,n}}
H(\xir)\otimes H(\xjr).
\]

Using
\[
\|x\otimes y\|_\hs=\|x\|\,\|y\|
\]
together with \eqref{8.1}, \eqref{8.24}, and Cauchy–Schwarz gives
\begin{equation}
    E\|D_{1n}\|_\hs
+
E\|D_{2n}\|_\hs
\lesssim n^{-1}.
\label{8.26}
\end{equation}

Since \(\widehat\Gamma_n^{\mathrm{JMB}}\) is nonnegative,
\[
E\|\widehat\Gamma_n^{\mathrm{JMB}}\|_\hs
\le
E\operatorname{tr}(\widehat\Gamma_n^{\mathrm{JMB}})
\lesssim1.
\]
Therefore, \eqref{8.23}, \eqref{8.25}, and \eqref{8.26} imply
\begin{equation}
    E\|\widehat\Gamma_n^{\mathrm{EB}}
-\widehat\Gamma_n^{\mathrm{JMB}}\|_\hs
\lesssim n^{-1}.
\label{8.27}
\end{equation}

Combining \eqref{8.27} with \eqref{8.15} yields
\begin{equation}
    E\|\widehat\Gamma_n^{\mathrm{EB}}-\Gamma_n\|_\hs
\lesssim n^{-1/2}.
\label{8.28}
\end{equation}

For the trace, use
\[
|\operatorname{tr}(x\otimes y)|
=
|\langle y,x\rangle|
\le\|x\|\,\|y\|.
\]
The same counting argument that gives \eqref{8.26} therefore gives
\[
E|\operatorname{tr}(D_{1n})|
+
E|\operatorname{tr}(D_{2n})|
\lesssim n^{-1}.
\]
Together with \eqref{8.20}, \eqref{8.23}, and \eqref{8.25}, this yields
\begin{equation}
    E\left|
\operatorname{tr}
(\widehat\Gamma_n^{\mathrm{EB}}-\Gamma_n)
\right|
\lesssim n^{-1/2}.
\label{8.29}
\end{equation}
Finally, \eqref{8.28} and \eqref{8.29}, \(\operatorname{tr}(\Gamma_n)\lesssim1\), and
\(\|\cdot\|_{\mathrm{op}}\le\|\cdot\|_\hs\) imply
\[
E\operatorname{tr}(\widehat\Gamma_n^{\mathrm{EB}})
\lesssim1,
\qquad
E\|\widehat\Gamma_n^{\mathrm{EB}}\|_{\mathrm{op}}
\lesssim1.
\]
This proves parts (a)–(d) for the empirical bootstrap and completes the proof.

\qed

\textbf{Proof of Lemma~\ref{lem: boundlrRdboots}.}
    Parts (a)--(d) of this proof are a direct consequence of Lemma~\ref{lem: boundlrHilb2} in $\R^d$. For the proof of part (e), we only provide the proof for JMB, and the others will follow similarly.  We may assume that $\Delta_{n,d}<1,$ since otherwise the corresponding probability bound is trivial.

We first consider the JMB. Retain the notation from the proof of Lemma~\ref{lem: boundlrHilb2}(ii) and write
\[
\bar g^{(i_1)}(X_{i_1})=g^{(i_1)}(X_{i_1})+\eta_{i_1}(X_{i_1}) ,
\qquad
\E[\eta_{i_1}(X_{i_1}) \mid X_i]=0.
\]
Define
\[
\bar\Gamma_n^{\mathrm{JMB}}
:=
\frac1n\sni
\bar g^{(i_1)}(X_{i_1})\bar g^{(i_1)}(X_{i_1})^\top.
\]
Since
\[
U_n^\circ
=
\frac1n\sni\bar g^{(i_1)}(X_{i_1}),
\]
we have
\[
\widehat\Gamma_n^{\mathrm{JMB}}
=
\bar\Gamma_n^{\mathrm{JMB}}
-U_n^\circ(U_n^\circ)^\top
\preceq
\bar\Gamma_n^{\mathrm{JMB}},
\]
where for two positive semi-definite matrices $A$ and $B$, $A \preceq B$ denotes $(B-A)$ is also positive semi-definite. Therefore,
\begin{equation}
\E\|\widehat\Gamma_n^{\mathrm{JMB}}\|_{\mathrm{op}}
\le
\|\Gamma_n\|_{\mathrm{op}}
+
\E\|\bar\Gamma_n^{\mathrm{JMB}}-\Gamma_n\|_{\mathrm{op}}.
\label{eq:lem64-reduction}
\end{equation}

Put
\[
\Gamma_n^{g}
:=
\frac1n\sni
g^{(i_1)}(X_{i_1})g^{(i_1)}(X_{i_1})^\top
\]
and
\[
\Gamma_n^{\eta}
:=
\frac1n\sni\eta_{i_1}(X_{i_1}) \eta_{i_1}(X_{i_1}) ^\top.
\]
Expanding \(\bar g^{(i)}=g^{(i)}+\eta_{i_1}\), we obtain
\begin{align}
\E\|\bar\Gamma_n^{\mathrm{JMB}}-\Gamma_n\|_{\mathrm{op}}
&\le S_1+2S_2+S_3,
\label{eq:lem64-decomposition}
\end{align}
where
\[
S_1
:=
\E\|\Gamma_n^{g}-\Gamma_n\|_{\mathrm{op}},
\]
\[
S_2
:=
\E\left\|
\frac1n\sni
g^{(i_1)}(X_{i_1})\eta_{i_1}(X_{i_1}) ^\top
\right\|_{\mathrm{op}},
\]
and
\[
S_3
:=
\E\|\Gamma_n^\eta\|_{\mathrm{op}}.
\]

By Proposition~\ref{prop: power} of \cite{fang2024large} and Assumption~\ref{BN},
\begin{align}
S_1
\lesssim
\sqrt{\frac{\|\Gamma_n\|_{\mathrm{op}}
                  \tr(\Gamma_n)}{n}}
+\frac{\tr(\Gamma_n)}{n}  \lesssim
\left(\frac{d^{q_1+q_3}}{n}\right)^{1/2}
+\frac{d^{q_3}}{n}.
\label{eq:lem64-S1}
\end{align}
Moreover, since \(\Gamma_n^\eta\) is nonnegative,
\begin{align}
S_3
&\le
\E\tr(\Gamma_n^\eta)
=
\frac1n\sni\E\|\eta_{i_1}(X_{i_1}) \|^2
\lesssim\frac dn.
\label{eq:lem64-S3}
\end{align}

For the cross term, the matrix Cauchy--Schwarz inequality gives
\[
\left\|
\frac1n\sni
g^{(i_1)}(X_{i_1})\eta_{i_1}(X_{i_1}) ^\top
\right\|_{\mathrm{op}}
\le
\|\Gamma_n^g\|_{\mathrm{op}}^{1/2}
\|\Gamma_n^\eta\|_{\mathrm{op}}^{1/2}.
\]
Hence, by Cauchy--Schwarz,
\begin{align}
S_2
\le
\{\E\|\Gamma_n^g\|_{\mathrm{op}}\}^{1/2}
\{\E\|\Gamma_n^\eta\|_{\mathrm{op}}\}^{1/2} \le
\{S_1+\|\Gamma_n\|_{\mathrm{op}}\}^{1/2}
S_3^{1/2}.
\label{eq:lem64-S2-compact}
\end{align}
Substituting \eqref{eq:lem64-S1} and
\eqref{eq:lem64-S3} yields
\begin{align}
S_2
\lesssim{}&
\left(
\frac{d^{q_1+q_3+2}}{n^3}
\right)^{1/4}
+
\left(
\frac{d^{1+q_3}}{n^2}
\right)^{1/2}
+
\left(
\frac{d^{1+q_1}}{n}
\right)^{1/2}.
\label{eq:lem64-S2}
\end{align}
We now use the assumption \(\Delta_{n,d}<1\). Since \(\Gamma_n\) is nonnegative,
\[
\|\Gamma_n\|_{\mathrm F}^2
\le
\|\Gamma_n\|_{\mathrm{op}}\tr(\Gamma_n).
\]
The assumed spectral growth rates therefore imply
\[
2q_2\le q_1+q_3.
\]
Since \(q_3\le1\),
\[
3+q_1-4q_2
\ge
3+q_1-2q_1-2q_3
\ge
1-q_1.
\]
On the other hand, the third term in \(\Delta_{n,d}\) gives
\[
\Delta_{n,d}
\ge
\left(
\frac{d^{3+q_1-4q_2}}{n\ell_n}
\right)^{1/6}.
\]
Because \(\ell_n< 1\), the condition \(\Delta_{n,d}<1\) implies
\begin{equation}
n\gtrsim d^{1-q_1}.
\label{eq:lem64-growth}
\end{equation}
Using \(q_3\le1\) and \eqref{eq:lem64-growth}, we obtain
\[
\left(\frac{d^{q_1+q_3}}n\right)^{1/2}
\lesssim d^{q_1},
\qquad
\frac{d^{q_3}}n\lesssim d^{q_1},
\qquad
\frac dn\lesssim d^{q_1}.
\]
Thus, \eqref{eq:lem64-S1}, \eqref{eq:lem64-S3}, and \eqref{eq:lem64-S2-compact} imply
\[
S_1+S_2+S_3\lesssim d^{q_1}.
\]
Combining this with \eqref{eq:lem64-reduction} and
\eqref{eq:lem64-decomposition} gives
\[
\E\|\widehat\Gamma_n^{\mathrm{JMB}}\|_{\mathrm{op}}
\lesssim d^{q_1}.
\]
\qed

\end{document}